\documentclass[11pt]{amsart}
\usepackage{amscd,amsxtra,amssymb,mathrsfs, bbm}
\usepackage{stmaryrd}

\usepackage{geometry}
        {\begin{quotation}\begin{center}\begin{em}}
        {\par\end{em}\end{center}\end{quotation}}

\newtheorem{theorem}{Theorem}[section]
\newtheorem{corollary}[theorem]{Corollary}
\newtheorem{lemma}[theorem]{Lemma}
\newtheorem{proposition}[theorem]{Proposition}
\newtheorem{conjecture}[theorem]{Conjecture}

\theoremstyle{definition}
\newtheorem{definition}[theorem]{Definition}
\newtheorem{remark}[theorem]{Remark}
\newtheorem{example}[theorem]{Example}

\DeclareMathOperator{\Perf}{\mathsf{Perf}}

\DeclareMathOperator{\add}{\mathsf{add}}

\DeclareMathOperator{\Coh}{\mathsf{Coh}}

\DeclareMathOperator{\Tri}{\mathsf{Tri}}

\DeclareMathOperator{\Hom}{\mathsf{Hom}}

\DeclareMathOperator{\Ext}{\mathsf{Ext}}

\DeclareMathOperator{\Aut}{\mathsf{Aut}}
\DeclareMathOperator{\End}{\mathsf{End}}
\DeclareMathOperator{\Mat}{\mathsf{Mat}}

\DeclareMathOperator{\Ob}{\mathsf{Ob}}
\DeclareMathOperator{\Rep}{\mathsf{Rep}}

\DeclareMathOperator{\catD}{\mathfrak{D}}
\DeclareMathOperator{\catE}{\mathfrak{E}}
\DeclareMathOperator{\catF}{\mathfrak{F}}
\DeclareMathOperator{\catA}{\mathfrak{A}}
\DeclareMathOperator{\catB}{\mathfrak{B}}
\DeclareMathOperator{\catC}{\mathfrak{C}}
\DeclareMathOperator{\dX}{\mathfrak{X}}
\DeclareMathOperator{\ddX}{X}
\DeclareMathOperator{\dY}{\mathfrak{Y}}
\DeclareMathOperator{\dF}{F}
\DeclareMathOperator{\dE}{E}
\DeclareMathOperator{\el}{\mathsf{El}}

\newcommand{\bul}{\scriptstyle{\bullet}}
\def\bop{\bigoplus}

\newcommand{\kk}{\mathbbm{k}}
\newcommand{\CC}{\mathbb{C}}
\newcommand{\NN}{\mathbb{N}}
\newcommand{\ZZ}{\mathbb{Z}}

\newcommand{\sm}{\mathsf{m}}

\newcommand{\EE}{\mathsf{E}}

\newcommand{\MM}{\mathsf{M}}

\newcommand{\PP}{\mathbb{P}}

\newcommand{\sfK}{\mathsf{K}}

\newcommand{\kB}{\mathcal{B}}

\newcommand{\kF}{\mathcal{F}}

\newcommand{\kH}{\mathcal{H}}
\newcommand{\kI}{\mathcal{I}}
\newcommand{\kJ}{\mathcal{J}}
\newcommand{\kO}{\mathcal{O}}
\newcommand{\kL}{\mathcal{L}}
\newcommand{\kP}{\mathcal{P}}
\newcommand{\kQ}{\mathcal{Q}}

\newcommand{\kS}{\mathcal{S}}

\newcommand{\lar}{\longrightarrow}

\def\lha{\leftharpoondown}
\def\rha{\rightharpoondown\phantom{I\!}\!}

 \def\sP{\mathsf P}
\def\sD{\mathsf D} 
 
\def\sF{\mathsf F} \def\sS{\mathsf S}
\def\sG{\mathsf G} \def\sT{\mathsf T}
\def\sH{\mathsf H} 
\def\sI{\mathsf I}

\def\sL{\mathsf L}

\def\kron#1#2{\xymatrix@C=2em{{#1}\ar@/^3pt/[r]\ar@/_3pt/[r]&{#2}}}
\def\bu{{\scriptscriptstyle\bullet}}

\usepackage{tikz}
\usetikzlibrary{calc, arrows, positioning, shapes, fit, matrix, decorations}
\usetikzlibrary{decorations.shapes, decorations.pathreplacing}

\newcommand{\hdline}[3]{\draw[dashed] (#1-#2-1.south west) -- (#1-#2-#3.south east);}
\newcommand{\vdline}[3]{\draw[dashed] (#1-1-#3.north east) -- (#1-#2-#3.south east);}

\tikzset{
  decorate with/.style={decorate,decoration={shape backgrounds,shape=#1,shape size=1.5mm}},
   deco/.style={decorate with=dart},
   ordi/.style={draw,-stealth,  thick},
   conj/.style={dashed, draw, thick},
   ve/.style={circle, draw, thick, fill=blue!20, inner sep=1pt, outer sep=2pt, minimum size=7pt},
    dot/.style={fill=blue!10,circle,draw, inner sep=1pt, minimum size=5pt},
  dv/.style={star,star points=5,
star point ratio=2, draw, thick, fill=green!20, inner sep=1pt,outer sep=2pt,minimum size=7pt}
}

\tikzset{
    tbl5 nodes/.style={
        rectangle,
        execute at begin node=$,
       execute at end node=$,
       fill=blue!5,
        align=center,
        text depth=0.5ex,
        text height=2ex,
        inner xsep=0pt,
        outer sep=0pt,
           },
    tbl5/.style={
        matrix of nodes,
        row sep=-\pgflinewidth,
        column sep=-\pgflinewidth,
        nodes={
            tbl5 nodes
        },
        execute at empty cell={\node[draw=none]{};}
    }
  }

\input xy
\xyoption{all}

\title[Derived categories of gentle and skew--gentle algebras]{On the derived categories of gentle and skew--gentle algebras: homological algebra and matrix problems}

\author{Igor Burban}
\address{
Universit\"at zu K\"oln,
Mathematisches Institut,
Weyertal 86-90,
D-50931 K\"oln,
Germany
}
\email{burban@math.uni-koeln.de}
\urladdr{www.mi.uni-koeln.de/$\sim$burban/}

\author{Yuriy Drozd}
\address{
 Institute of Mathematics\\
National Academy of Sciences of Ukraine,
Tereschenkivska str. 3,
01004 Kyiv, Ukraine}
\email{drozd@imath.kiev.ua, y.a.drozd@gmail.com}
\urladdr{www.imath.kiev.ua/$\sim$drozd}


\begin{document}
\maketitle

\begin{abstract}
In this paper,  we investigate properties of the bounded derived category of finite dimensional modules over a gentle or skew--gentle algebra.
We show that the Rouquier dimension of the derived category of such an algebra is at most one. Using this  result, we prove that the Rouquier dimension of an arbitrary tame projective curve is equal to one, too. Finally, we elaborate the classification of indecomposable objects of the (possibly unbounded) homotopy category of projective modules of a gentle algebra.
\end{abstract}

\tableofcontents

\section{Introduction}
Motivated by the study of finite length admissible representations of the Lie group $\mathsf{SL}_2(\CC)$, Gelfand and Ponomarev classified   in \cite{GP}
all indecomposable finite dimensional modules over  the completed path algebra $C$  of the following quiver with relations:
\begin{equation*}\label{E:ZhelobQuiver}
\xymatrix
{
\bul  \ar@(ul, dl)_{x} \ar@/^/[rr]^{y}  & & \bul \ar@/^/[ll]^{z}
} \qquad xz  = 0, \, y x  = 0.
\end{equation*}
In the work \cite{GP} it was also observed that from the representation--theoretic point of view, the algebra $C$ is closely related with the algebra $N := \CC\llbracket u, v\rrbracket/(uv)$ (see also Remark \ref{R:GPRevisited}).  Gelfand and Ponomarev   proved  that there are two types of indecomposable finite--dimensional $C$--modules (as well as $N$--modules): those described by a certain \emph{discrete} parameter $\underline{v}$ and the
ones described by a tuple $(\underline{w}, m, \pi)$, where $\underline{w}$ is a certain \emph{discrete} parameter, $m \in \NN$ and $\pi \in \CC^*$ is a \emph{continuous} moduli parameter. This, in particular, means  that the categories  of finite--dimensional $C$--modules and $N$--modules are  representation--tame.
The work \cite{GP}  was a  starting point of an extensive study of representation--tame  categories, which came in the years afterwards. Appropriate analogues of the indecomposable  modules of the first type were called \emph{strings} (discrete series), whereas generalizations  of the indecomposable modules of the second type were called \emph{bands}. We also want to mention that arround the same time, Nazarova and Roiter obtained a classification of all indecomposable \emph{finitely generated} modules over  a dyad of two discrete valuation rings (in particular, over $N$),  using  a completely different approach  of \emph{matrix problems}; see \cite{NazarovaRoiterDyad}.

For an algebraically closed field $\kk$, Assem and Skowro\'nski introduced in \cite{AssemSkowr} a certain class of \emph{finite--dimensional} analogues of the above algebra $C$
called \emph{gentle}. By definition, these are path algebras  $\kk\bigl[\stackrel{\rightarrow}{Q}\bigr]/\langle L\rangle$ for which the quiver $\stackrel{\rightarrow}{Q}$ and the admissible ideal $L$ satisfy the following conditions.
\begin{itemize}
\item For any vertex $i$ of  $\stackrel{\rightarrow}{Q}$, there exist at most two arrows starting at $i$ and at most two arrows ending at $i$.
\item The ideal $L$ is generated by a set of paths of lengths two.
\item Let $\xymatrix{\stackrel{i}\bul \ar[r]^a & \stackrel{j}\bul}$ be any arrow in $\stackrel{\rightarrow}{Q}$. Then the following is fulfilled:
\begin{itemize}
\item Whenever there are two arrows $b, c$ in $\stackrel{\rightarrow}{Q}$ ending at $i$ then \emph{precisely one} of the paths $\{ab, ac\}$ belongs to $L$.
\item Analogously, whenever there are two arrows $g, h$ in $\stackrel{\rightarrow}{Q}$ starting at $j$ then \emph{precisely one} of the paths $\{ga, ha\}$ belongs to $L$.
\end{itemize}
\end{itemize}
It turned out that gentle  algebras satisfy a number of remarkable properties. First of all, it was shown by  Wald and Waschb\"usch \cite{WaldWaschbuesch}
as well as by  Butler and Ringel \cite{ButlerRingel} that gentle algebras are either representation finite or representation tame. Moreover, in that articles was shown that there are two types of indecomposable modules over a gentle algebra: strings and bands. The combinatorial pattern describing both classes of indecomposables turned out to be essentially the same as in the case of the  algebras $C$ and $N$ studied by Gelfand and Ponomarev \cite{GP} as well as by Nazarova and Roiter \cite{NazarovaRoiterDyad}.

Another key observation established in the late 80s and in the 90s  was that the bounded derived category $D^b(A-\mathsf{mod})$ of a gentle algebra $A$ has remarkable properties, viewed both from the homological as well as from the representation--theoretical perspective.
For an arbitrary finite dimensional $\kk$--algebra $\Lambda$,  Happel constructed in \cite{Happel} a fully faithful functor
$
D^b(\Lambda-\mathsf{mod}) \stackrel{\mathsf{H}}\lar \widehat{\Lambda}-\underline{\mathsf{mod}},
$
where $\widehat{\Lambda}$ is the so--called \emph{repetitive algebra} of $\Lambda$ (which is known to be self--injective) and $\widehat{\Lambda}-\underline{\mathsf{mod}}$ is the stable category of the category of finite--dimensional
$\widehat{\Lambda}$--modules. Moreover,  Happel's functor $\mathsf{H}$ is an equivalence of triangulated categories provided $\mathsf{gl.dim}(A) < \infty$; see \cite{Happel}.

It was shown by Ringel \cite{Ringelrep} (see also \cite{Schroer}) that the repetitive algebra $\widehat{A}$ of a gentle algebra $A$
is \emph{special biserial}. According to \cite{WaldWaschbuesch, ButlerRingel}, the triangulated category $\widehat{\Lambda}-\underline{\mathsf{mod}}$
(and, as a consequence of Happel's embedding, the derived category $D^b(A-\mathsf{mod})$) is either representation--discrete or representation--tame. Happel's embedding $\mathsf{H}$ allowed to prove a number of remarkable properties about the derived category $D^b(A-\mathsf{mod})$.
 For instance, Schr\"oer and Zimmermann showed  in \cite{SchroerZimmermann} that the class of gentle algebras is closed under derived equivalences. As a consequence, any gentle algebra has only finitely many derived Morita partners. Next, Happel's embedding $\mathsf{H}$ provides a powerful tool to study the Auslander--Reiten quiver of the derived category $D^b(A-\mathsf{mod})$. However, despite  the fact that  the repetitive algebra $\widehat{A}$ can be described rather explicitly  as the path algebra of a certain infinite quiver with relations \cite{Ringelrep, Schroer} and tools to study indecomposable
$\widehat{A}$--modules are rather developed, it is a non--trivial problem to describe the preimage of an arbitrary indecomposable object of the stable category $\widehat{A}-\underline{\mathsf{mod}}$ in  the derived category  $D^b(A-\mathsf{mod})$ under the functor $\mathsf{H}$ in a constructive way.

In \cite{BekkertMerklen}, Bekkert and Merklen gave an explicit description of all indecomposable objects of the bounded homotopy category  $\mathsf{Hot}^b(A-\mathsf{pro})$ of projective $A$--modules. Their classification  was   based on a certain matrix problem studied by  Bondarenko in \cite{BondarenkoFirst}.  According to their result, there are two types of indecomposable objects of $\mathsf{Hot}^b(A-\mathsf{pro})$: string complexes
$S^\bu(\overline{v})$ as well as band complexes $B^\bu(\overline{w}, m, \pi)$ (the latter ones exist only in the tame case).
The obtained explicit combinatorics  of the indecomposable objects of $\mathsf{Hot}^b(A-\mathsf{pro})$ was essentially used in several   recent papers; see in particular  \cite{ALP, ALPP, CPS}.

In \cite{Vossieck} Vossieck classified  all finite dimensional $\kk$--algebras, whose derived category $D^b(A-\mathsf{mod})$ is representation--discrete. Quite remarkably, it turned out that those finite dimensional $\kk$--algebras, whose derived category is not \emph{representation--finite},  are necessarily gentle. Moreover, the corresponding quiver $\stackrel{\rightarrow}{Q}$  has precisely one oriented cycle. Vossieck's classification was finalized in \cite{BobinskiGeissSkowronski} by Bobinski, Gei\ss{} and Skowro\'nski. The derived--discrete  algebras were extensively studied by several authors afterwards;  see in particular \cite{BPP, ALP}.

In recent years, several important results on the bounded derived category of an arbitrary gentle algebra were established. In particular,
Avella--Alaminos and Gei\ss{} introduced in  \cite{AAGeiss} a new (and completely   combinatorial) derived invariant of a gentle algebra. In \cite{Kalck},
Kalck gave a full
description of  the singularity category $D^b(A-\mathsf{mod})/\mathsf{Hot}^b(A-\mathsf{mod})$ of an arbitrary gentle algebra $A$ of infinite global dimension.

Quite remarkably, gentle algebras appeared in the following striking new developments, relating the representation theory of finite dimensional algebras with other fields of mathematics.  In the works of Labardini--Fragoso \cite{Labardini}, Asssem, Br\"ustle, Charbonneau--Jodoin and Plamondon \cite{ABCP} and David--Roesler and Schiffler \cite{DRSchiffler} it was shown that many gentle algebras arise as so--called \emph{surface algebras} from special  triangulations of Riemann surfaces. Representation theoretic aspects of that constructions were  studied in particular in  \cite{GLFS} and \cite{AmiotGrimeland}. In a work of  Haiden, Katzarkov and Kontsevich \cite{HaidenKatzarkovKontsevich}, algebraic constructions related with gentle algebras appeared in the context of Fukaya categories of Riemann surfaces. In a recent work of  Lekili and Polishchuk \cite{LekiliPolishchuk},
some gentle algebras appeared in the context of the homological mirror symmetry for Riemann surfaces. Finally, we want to  mention  a work of Cecotti \cite{Cecotti}, in which
gentle algebras arise  in the context of certain  gauge theory models of  theoretical physics.

Around the beginning of the millennium, the authors of this paper suggested a new method to reduce the problem of classification of the indecomposable objects  of the derived categories  of coherent sheaves  on tame singular projective curves  (chains and cycles of projective lines) \cite{Thesis, Burban, Duke} as well of the derived categories  of
various classes of finite--dimensional $\kk$--algebras \cite{BD02, Thesis, Nodal, Toronto} to a certain class of tame matrix problems called \emph{representations of bunches of (semi--)chains} \cite{NazarovaRoiter, BondarenkoMain}. The proposed method confirmed an observation made earlier in \cite{DGVB} that tame singular projective curves must be related with appropriate  gentle algebras. In \cite{Burban}, the first--named author noticed that the derived category of coherent sheaves on a chain of projective lines admits a tilting vector bundle, whose endomorphism algebra is gentle.
The connection between gentle algebras and tame projective curves was made  more concrete in  \cite{bd}, in the context of non--commutative nodal curves. In that work, we in particular showed, that  appropriate  gentle algebras of global dimension two, which appeared in \cite[Appendix A4]{DGVB}, turn out to be  categorical resolutions of singularities of the bounded derived category of coherent sheaves on a cycle of projective lines.

In \cite{GeissDelaPena}, Gei\ss{} and de la Pe\~na introduced the notion of a \emph{skew--gentle} algebra. Such  algebras are known to be tame and even derived--tame \cite{GeissDelaPena, BekkertMarcosMerklen}. Certain classes of  skew--gentle algebras arise as surface algebras \cite{GLFS}, other classes of skew--gentle
algebras  appear in the context of non--commutative nodal curves
\cite{bd}. In what follows,  gentle algebras will be treated as  special cases    of skew--gentle algebras.

In this paper, we show the following results on  (skew--)gentle algebras. Firstly, we prove that the so--called
\emph{Rouquier dimension} $\mathsf{der.dim}(A)$ of the bounded derived
category of a (skew--)gentle algebra $A$  is  \emph{at most one}; see Corollary \ref{C:RouquierDimSkewGentle}. This is achieved using the following method. Let $H$ be the \emph{normalization} of $A$ (see Definition \ref{D:Sets}). It is a hereditary $\kk$--algebra (actually, it is Morita--equivalent
to a product of path algebras of quivers of type $\bul \lar \bul \lar  \dots \lar \bul \lar \bul$) such that $\mathsf{rad}(A) = \mathsf{rad}(H) =: I$. It turns out that the algebra
$B = \left(\begin{array}{cc} A & H \\I & H\end{array} \right)$ is again (skew--)gentle and has global dimension \emph{two}. Moreover,  the following results are true (see Theorem \ref{T:skewgentleResolution}).
\begin{itemize}
\item There exists a fully faithful exact functor $\mathsf{Hot}^b(A-\mathsf{pro}) \lar D^b(B-\mathsf{mod})$. In other words, $D^b(B-\mathsf{mod})$ is a categorical resolution of singularities of the algebra $A$.
\item The derived category $D^b(A-\mathsf{mod})$ is an appropriate  Verdier localization of the derived category $D^b(B-\mathsf{mod})$.
\item We have a semi--orthogonal decomposition
\begin{equation*}\label{E:semiorth}
D^b(B-\mathsf{mod}) = \bigl\langle D^b(\bar{A}-\mathsf{mod}), D^b(H-\mathsf{mod})\bigr\rangle.
\end{equation*}
\end{itemize}
From this  deduce  that $D^b(A-\mathsf{mod}) = \langle Z\rangle_2$, where $Z$ is the  direct sum of all pairwise non--isomorphic indecomposable $H$--modules. In other words, for any object $X^\bu$ of $D^b(A-\mathsf{mod})$, there exists an exact triangle
 $$
 X^\bu_1 \lar X^\bu \lar X^\bu_2 \lar X^\bu_1[1],
 $$
in which  the complexes $X_1^\bu$ and $X_2^\bu$ are appropriate direct sums of certain shifts of some direct summands of $Z$.
By definition, this implies that $\mathsf{der.dim}(A) \le 1$. The developed method is also applicable in some representation--wild cases. For example, we show that
$\mathsf{der.dim}\bigl(\kk[x_1, \dots, x_n]/(x_1, \dots, x_n)^2\bigr) = 1$ for any $n \in \NN$; see Example \ref{Ex:fatpointRouqDim}.

According to a result of Orlov \cite{Orlov}, the Rouquier dimension $\mathsf{der.dim}(X)$ of the derived category of coherent sheaves of any smooth projective curve $X$ is equal to one.
Combined with some results and methods of our previous work \cite{bd}, we deduce that $\mathsf{der.dim}(X) = 1$ for any tame singular projective curve. In Example \ref{Ex:RouqDimWeierCubic}, we  elaborate in detail the case of the nodal Weierstra\ss{} cubic $E = \overline{V(y^2 - x^3 - x^2)} \subset \PP^2$. Namely, if  $s \in E$ is the singular point, $\PP^1 \stackrel{\nu}\lar E$ a normalization map and $\widetilde\kO(n) := \nu_*\bigl(\kO_{\PP^1}(n)\bigr)$ for $n \in \ZZ$, then we have:
$
D^b\bigl(\Coh(E)\bigr) = \langle \kk_s \oplus \widetilde\kO \oplus \widetilde\kO(1)\rangle_2,
$
what implies that  $\mathsf{der.dim}(E) = 1$.

Our next goal was to revise the classification of indecomposable objects of the derived categories (bounded and unbounded)  of a gentle algebra $A$. In fact, this part of our paper is an update of our earlier works \cite{BD02, Nodal, Toronto}. In particular, in \cite{Nodal} we proved that the so--called \emph{nodal orders}
are derived--tamed (these class of rings forms a natural non--commutative generalization of the algebra $N = \kk\llbracket u, v\rrbracket/(uv)$ in the class
of orders. In fact, nodal orders  are infinite--dimensional analogues of gentle and skew--gentle algebras). In \cite{BD02, Toronto} it was shown that the proposed method allows to establish derived tameness of various types of finite--dimensional algebras.
In Section \ref{S:GentleCombinatorics}, we give a new proof of the classification of the indecomposable objects of the homotopy category $\mathsf{Hot}^b(A-\mathsf{mod})$, obtained by Bekkert and Merklen in \cite{BekkertMerklen} and generalize it  on the unbounded homotopy categories. Our method can be briefly described   as follows. For any $\ast \in \{b, +, -, \emptyset\}$, we construct
a pair of functors  $\EE$ and $\MM$:
\begin{equation*}
\mathsf{Hot}^\ast(A-\mathsf{pro}) \stackrel{\EE}\lar \Tri^\ast(A) \stackrel{\MM}\lar \Rep^\ast(\dX),
\end{equation*}
both reflecting  the isomorphism classes and indecomposability of objects. Here, $\Tri^\ast(A)$ is the so--called \emph{category of triples} (see
Definition \ref{D:triples}) and $\Rep^\ast(\dX)$ is the category of locally finite dimensional (but possibly infinite dimensional) representations
 of an appropriate \emph{bunch of chains} $\dX$ (see Theorem \ref{T:BurbanDrozdTriples}). From a description of the essential image of the composition $\MM \circ \EE$ and a classification of the indecomposable objects of $\Rep^\ast(\dX)$, we deduce a classification of the indecomposable objects of the homotopy category $\mathsf{Hot}^\ast(A-\mathsf{pro})$; see Theorem \ref{T:IndecGentle} and Theorem \ref{T:UnboundedIndecom}.

There were  several reasons, which motivated us to return to this old subject. Firstly, the reduction method of the work  \cite{BekkertMerklen} is not applicable to the unbouded homotopy categories $\mathsf{Hot}^\ast(A-\mathsf{pro})$ for
$\ast \in \{+, -, \emptyset\}$, whereas within our approach,  the unboundedness leads only to some  minor technical complications.
In the case $\mathsf{gl.dim}(A) = \infty$,  we elaborate  in detail
an explicit description  of those indecomposable objects of the derived category $D^b(A-\mathsf{mod})$, which have  unbounded minimal projective resolutions; see Theorem \ref{T:boundedInfGlobDim}.  Such objects are certain infinite string complexes $S^\bu(\overline{u})$ (in particular, they have discrete combinatorics). The obtained classification  gives a new insight on Kalck's description of the singularity category of $A$;  see \cite{Kalck}. Another reason concerned the key technical ingredient of the work \cite{BekkertMerklen}. Bekkert and Merklen deduce their classification from a certain tame matrix problem studied by Bondarenko in \cite{BondarenkoFirst}. It was shown
in \cite[Section 1]{BondarenkoFirst} (and then by different means  in \cite[Theorem 3]{BondarenkoDrozd}) that the  matrix problem in question  is tame since it can be reduced to another tame matrix problem, studied by Nazarova and Roiter \cite{NazarovaRoiter}. However, the precise  combinatorics of the indecomposable objects of the original matrix problem \cite{BondarenkoFirst} has never been elaborated. Summing up, the
classification of indecomposable objects of $\mathsf{Hot}^b(A-\mathsf{pro})$ obtained
in \cite{BekkertMerklen}, was heavily based on a  matrix problem \cite{BondarenkoFirst}, whose detailed treatment  was rather badly documented in the literature. We also think that
it is both advantageous and instructive that the problem of description of the  indecomposable objects of representation--tame categories arising from algebraic geometry \cite{DGVB, Duke,SurveyBBDG}, commutative
algebra \cite{BurbanDrozdMemoirs, BurbanGnedin} and representation theory \cite{BD02, Nodal, Toronto} can be solved  within essentially the same reduction scheme.

Another goal of this work is to demonstrate the unity of the homological and represen\-ta\-tion--theoretical methods of study of the derived categories of a (skew--)gentle algebra. Based on the technique of semi--orthogonal decompositions, we propose in Section \ref{S:MPSemiorthDecomp} another reduction of the classification of the isomorphism classes of  objects of
$D^b(B-\mathsf{mod})$ to the category of representations   of an appropriate bunch of (semi--)chains.

\smallskip
\noindent
\emph{Acknowledgement}. The work of the first--named author was partially supported by the DFG project Bu--1866/4--1.
\section{Generalities on gentle and skew--gentle algebras}\label{S:Generalities}

\noindent
In this article, let $\kk$ be an algebraically closed field. All  algebras below are finite dimensional associative algebras over $\kk$. If not explicitly mentioned, a module is meant to be a  left module, finite dimensional over the filed $\kk$.
\begin{definition}\label{D:skewgentleprep}
For any $m \in \NN_{\ge 2}$, let $T_m$ be the algebra of all lower--triangular square matrices of size $m$:
\begin{equation}
T_m :=
\left\{\left.
\left(
\begin{array}{cccc}
\alpha_{11} & 0 & \dots & 0 \\
\alpha_{21} & \alpha_{22} &  \dots  & 0 \\
\vdots & \vdots& \ddots & \vdots \\
\alpha_{m1} & \alpha_{m2} &  \dots  & \alpha_{mm}
\end{array}
\right)\right|  \alpha_{ij} \in \kk
\right\} \cong \kk\bigl[\vec{A}_m\bigr],
\end{equation}
where $\vec{A}_m$ is the quiver  $\xymatrix{\stackrel{1}\bul \ar[r] & \stackrel{2}\bul \ar[r] & \dots \ar[r]& \stackrel{m}\bul}$.
Next,
let $\Sigma \subseteq
\{1, \dots, m\}$ be any subset and $\widetilde{m} := m + |\Sigma|$. Then we get the  algebra $T_{m, \Sigma} \subset \Mat_{\widetilde{m}}(\kk)$ obtained from $T_m$ by the following ``blowing up procedure'': for any $j \in \Sigma$, we replace the $j$--th row (respectively, the $j$--th column) of each element of
$T_{m}$ by two rows (respectively, columns) of the same shape. In particular, for any $j \in \Sigma$, the square block $(j, j)$ is a $(2\times 2)$ matrix:
\begin{equation}\label{E:2times2block}
\alpha_{jj} = \left(
\begin{array}{cc}
\alpha_{jj}^{11} & \alpha_{jj}^{12} \\
\alpha_{jj}^{21} & \alpha_{jj}^{22}
\end{array}
\right).
\end{equation}
\end{definition}

\noindent
It is clear that $T_{m, \emptyset} = T_m$ and that algebras $T_{m, \Sigma}$ and $T_m$ are Morita--equivalent.
\begin{example}
Let $m = 3$ and $\Sigma = \{1, 3\}$. Then $T_{3, \Sigma}$ is the algebra of matrices of size $(5 \times 5)$, having the following form:
$$
T_{m, \Sigma} =
\left\{
\left(
\left.
\begin{array}{ccccc}
\ast & \ast & 0 & 0 & 0 \\
\ast & \ast & 0 & 0 & 0 \\
\ast & \ast & \ast & 0 & 0 \\
\ast & \ast & \ast & \ast & \ast \\
\ast & \ast & \ast & \ast & \ast \\
\end{array}
\right)\right| \; \mbox{where\;}\ast \;  \mbox{ist an arbitrary element of the field\;} \kk
\right\}.
$$
\end{example}

\begin{definition}\label{D:skewgentle} Consider the following datum $(\vec{\sm}, \simeq)$, where
\begin{itemize}
\item $\vec{\sm} = \bigl(m_1, \dots, m_t\bigr) \in \NN^t_{\ge 2}$ for some $t \in \NN$.
\item $\simeq$ is a \emph{symmetric} but not necessarily \emph{reflexive} relation on the set
\begin{equation}\label{E:SetOmega}
\Omega = \Omega(\vec{\sm}):= \bigl\{(i, j)\,\big|\, 1 \le i \le t, 1 \le j \le m_i \bigr\}
\end{equation}
such that for any $\gamma \in \Omega$, there exists \emph{at most one} $\delta \in \Omega$ such that $\gamma \simeq \delta$.
\end{itemize}
The corresponding  \emph{skew--gentle} algebra $A = A(\vec{\sm}, \simeq)$ is defined as follows.
\begin{itemize}
\item For any $1 \le i \le t$, we consider the following set
$$
\Sigma_i := \bigl\{1 \le j \le m_i \, \big| \, (i, j) \simeq (i, j)\bigr\}.
$$
and  denote:  $H_i := T_{m_i, \Sigma_i}$. Next, we denote:
\begin{equation}\label{E:Normalization}
H:= H(\vec{\sm}, \simeq):= H_1 \times \dots \times H_t.
\end{equation}
\item Finally, we put:
\begin{equation*}
H \supseteq A :=
\left\{
\bigl(X(1), \dots, X(t)\bigr) \,\left|\,
\begin{array}{cl}
X(i)_{jj} = X(k)_{ll} & \; \mbox{if}\;  (i,j)\simeq (k, l) \; \mbox{and}\;  (i, j) \ne (k,l) \\
X(i)_{jj}^{12} = 0 = X(i)_{jj}^{21} & \;\mbox{if} \; (i,j)\simeq (i,j)
\end{array}
\right.
\right\}.
\end{equation*}
\end{itemize}
The algebra  $A(\vec{\sm}, \simeq)$ is called \emph{gentle} if $\gamma \not\simeq \gamma$ for all $\gamma \in \Omega$. So, within  our definition, gentle algebras are special cases of skew--gentle algebras.
\end{definition}

\begin{example}
Let $A = \kk[\varepsilon]/(\varepsilon^2)$. Then $A$ is gentle, where the corresponding datum  $(\vec{\sm}, \simeq)$ is the following:
$\vec{\sm} = (2)$ and $(1,1) \simeq (1, 2)$. In the matrix notation, we have:
$$
A \cong  \left\{
\left(\left.
\begin{array}{cc}
\alpha_{11} & 0 \\
\alpha_{21} & \alpha_{22}
\end{array}
\right) \right| \alpha_{11} = \alpha_{22}
\right\}.
$$
\end{example}

\begin{example}\label{Ex:CatO} Let $A$ be the path algebra of the following quiver with relations:
\begin{equation}
\xymatrix{
\stackrel{1}\bul   \ar@/^/[r]^{a}   & \stackrel{2}\bul \ar@/^/[l]^{b}
} \quad ab = 0.
\end{equation}
Then $A$ is a gentle algebra attached to the datum $(\vec{\sm}, \simeq)$, where
$\vec{\sm} = (3)$ and $(1,1) \simeq (1, 3)$. In the matrix notation, we have:
$$
A \cong  \left\{
\left(\left.
\begin{array}{ccc}
\alpha_{11} & 0  & 0\\
\alpha_{21} & \alpha_{22} & 0 \\
\alpha_{31} & \alpha_{32} & \alpha_{33}
\end{array}
\right) \right| \alpha_{11} = \alpha_{33}
\right\}.
$$
The category $A-\mathsf{mod}$ is equivalent to the principal  block of the category $\mathcal{O}$ of  the Lie algebra $\mathfrak{sl}_2(\CC)$; see for instance \cite[Section 3.12]{Humphreys}.
\end{example}

\begin{example}\label{Ex:MyFavorite} Let $A$ be the path algebra of the following quiver with relations:
\begin{equation}
\xymatrix{
\stackrel{1}\bul   \ar@/^/[r]^{a} \ar@/_/[r]_{c}  & \stackrel{2}\bul \ar@/^/[r]^{b} \ar@/_/[r]_{d} & \stackrel{3}\bul
} \quad bc =  da = 0.
\end{equation}
Then $A$ is the  gentle algebra attached to the datum $\vec{\sm} = (3, 3)$ with the relation $(1, j) \simeq (2, j)$ for all $1 \le j \le 3$. In the matrix realization,
$$
A \cong
\left\{\left.
\left(
\left(\begin{array}{ccc}
\alpha_{11} & 0  & 0\\
\alpha_{21} & \alpha_{22} & 0 \\
\alpha_{31} & \alpha_{32} & \alpha_{33}
\end{array}
\right),
\left(\begin{array}{ccc}
\beta_{11} & 0  & 0\\
\beta_{21} & \beta_{22} & 0 \\
\beta_{31} & \beta_{32} & \beta_{33}
\end{array}
\right)
\right)
 \right| \alpha_{jj} = \beta_{jj} \; \mbox{\rm for all}\; 1 \le j \le 3
\right\}.
$$
For nodal cubic
curve $E = \overline{V(y^2-x^3-x^2)} \subset \PP^2$, we have an exact and fully faithful functor
$
\Perf(E) \lar D^b(A-\mathsf{mod}),
$
where $\Perf(E)$ is the perfect derived category of coherent sheaves on $E$;
see \cite[Section 7]{bd}. In other words, the derived category
 $D^b(A-\mathsf{mod})$ is a categorical resolution of singularities of $E$.
\end{example}

\begin{example} For any $n \in \NN$, let $A$ be the path algebra of the following quiver
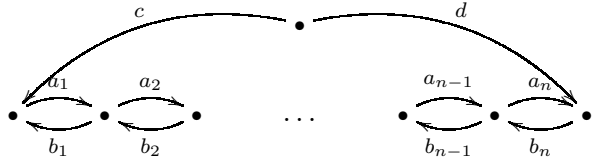
\begin{equation*}
\xymatrix{
& & & \bul \ar@/_3.6ex/[llld]_-{c} \ar@/^3.6ex/[rrrd]^-{d} & & & \\
\bul   \ar@/^/[r]^{a_1}   &  \bul \ar@/^/[l]^{b_1} \ar@/^/[r]^{a_2}   &  \bul \ar@/^/[l]^{b_2} & \dots & \bul   \ar@/^/[r]^{a_{n-1}}   &  \bul \ar@/^/[l]^{b_{n-1}} \ar@/^/[r]^{a_n}   &  \bul \ar@/^/[l]^{b_n}
}
\end{equation*}
subject to the following set of relations:  $a_i b_i = 0$ and  $b_i a_i = 0$ for all $1 \le i \le n$. Then $A$ is the gentle algebra attached to the datum $\vec{\sm} = (n+2, n+2)$ with the relation given by the rule: $(1, 1) \simeq (2, 1)$ and $(1, j) \simeq (2, n+4-j)$ for all $2 \le j \le n+2$. According to \cite[Theorem 2.1]{Burban}, the derived category $D^b(A-\mathsf{mod})$ is equivalent to $D^b\bigl(\Coh(X)\bigr)$, where $X$ is a chain on $n+1$ projective lines.
\end{example}

\begin{example} For $n \in \NN_{\ge 2}$, let $S^{n} := \bigl\{\vec{x} \in \mathbb{R}^{n+1} \,\big|\, \|\vec{x}\| = 1\bigr\}$ be  a real sphere of dimension $n$. We denote:  $S^{n}_\pm := \bigl\{\vec{x} \in \mathbb{R}^{n+1} \,\big|\,  \pm x_{n+1} > 0\bigr\}$ and  $S^{n}_0 := \bigl\{\vec{x} \in  S^{n} \,\big|\,  x_{n+1} = 0\bigr\} \cong S^{n-1}$.
This gives a  stratification $S^{n} = S^{n}_+ \sqcup S^{n}_- \sqcup  S^{n}_0$, which can be inductively extended to a stratification
$S^{n} = \sqcup_{i=0}^n S^{i}_\pm
$
with contractible strata.
For example, for $n = 2$ we obtain:
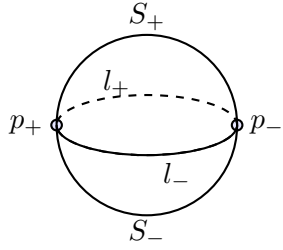
\newcommand{\sfrm}[3]{
\node[draw,solid, thick, fit=(#1-1-1)(#1-#2-#3), inner sep=0pt]{};}
\begin{center}
\begin{tikzpicture}[scale=0.40,
    thick,
    dot/.style={fill=blue!10,circle,draw, inner sep=1pt, minimum size=4pt}]
\draw (0,0) node[dot](p1){} (6,0) node[dot](p2){};
\draw (3,0) circle (3cm);
\draw[dashed] (3,0) circle [ x radius=3cm, y radius=1cm];

\draw (0,0) arc (180:360:3 and 1);

\node at (3,3.6){$S_+$};
\node at (3,-3.6){$S_-$};
\node at (4,-1.7){$l_-$};
\node at (2, 1.4){$l_+$};
\node at (-1, 0){$p_+$};
\node at (7.0, 0){$p_-$};
\end{tikzpicture}
\end{center}
According to results from \cite[Chapter 8]{KashiwaraSchapira} (see in particular \cite[Theorem 8.1.10 and Exercise 8.1]{KashiwaraSchapira}), the corresponding derived category of constructible sheaves on $S^n$ is equivalent to the derived category of representations of the skew--gentle algebra $A(\vec{\sm}, \simeq)$, corresponding to the datum
$\vec{\sm} = (n+1)$ equipped with the  relation  $(1, j) \simeq (1, j)$ for all $1 \le j \le n+1$. For example, for  $n =2$, we get the following algebra:
$$
A = A(\vec{\sm}, \simeq)=
\left\{
\left(
\left.
\begin{array}{cc|cc|cc}
\ast & 0 & 0 & 0 & 0 & 0\\
0 & \ast & 0 & 0 & 0 & 0 \\
\hline
\ast & \ast & \ast & 0  & 0 & 0 \\
\ast & \ast & 0 & \ast & 0 & 0\\
\hline
\ast & \ast & \ast & \ast & \ast & 0 \\
\ast & \ast & \ast & \ast & 0 & \ast \\
\end{array}
\right)\right| \; \mbox{where\;}\ast \;  \mbox{ist an arbitrary element of\;} \kk
\right\}.
$$
Alternatively, $A$  is the path algebra of the quiver
$$
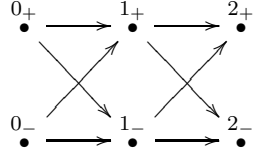
\xymatrix{
\stackrel{0_+}\bul   \ar[r] \ar[rd]  & \stackrel{1_+}\bul \ar[r] \ar[rd] & \stackrel{2_+}\bul \\
\stackrel{0_-}\bul   \ar[r] \ar[ru]  & \stackrel{1_-}\bul \ar[r] \ar[ru] & \stackrel{2_-}\bul
}
$$
subject to the following set of relations: any two paths with the same source and target are equal.
\end{example}

\begin{example} Let $\vec{\sm} = (3, 3)$ and $(1, 1) \simeq (2, 1), (1, 3) \simeq (2, 3), (1, 2) \simeq (1, 2), (2, 2) \simeq (2, 2)$.
Then the corresponding  skew--gentle algebra $A(\vec{\sm}, \simeq)$  is isomorphic to  the path algebra of the quiver
$$
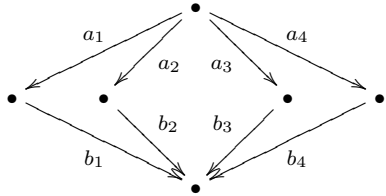
\xymatrix
{
        &           & \bul \ar[lld]_{a_1}
\ar[ld]^{a_2}  \ar[rd]_{a_3}  \ar[rrd]^{a_4}
      &         &       \\
\bul \ar[rrd]_{b_1}  & \bul \ar[rd]^{b_2}
&        & \bul \ar[ld]_{b_3}
& \bul \ar[lld]^{b_4}\\
        &           &\bul &         &       \\
}
$$
subject to the  relations $b_1 a_1 = b_2 a_2$ and  $b_3 a_3 = b_4 a_4$. This algebra
is a degeneration of  a  family of  canonical tubular algebras of type $(2,2,2,2)$,
introduced by Ringel \cite{Ringel}. Moreover, $A(\vec{\sm}, \simeq)$   is derived--equivalent to the category of coherent sheaves on a certain non--commutative nodal curve; see
\cite[Section 8.3]{bd} for details.
\end{example}

\begin{remark} Of course, our definition of gentle and skew--gentle algebras coincides with the original ones \cite{AssemSkowr, GeissDelaPena}.
 For example, let $A= \kk\bigl[\stackrel{\rightarrow}{Q}\bigr]/\langle L\rangle$ a gentle algebra in the sense of the definition from Introduction of this paper. It is easy to see that any arrow of $\stackrel{\rightarrow}{Q}$ belongs to a uniquely determined maximal non--zero path $\varpi$ in $A$. Let $\varpi_1, \dots, \varpi_t$ be the set  of the maximal paths and $m_i$ be the length of  $\varpi_i$ for $1 \le i \le t$. For any $1 \le j \le m_i$, let $(i, j)$ be the $j$--th vertex of $\varpi_i$.
Then we put:
\begin{itemize}
\item $\vec{\sm} = (m_1+1, \dots, m_t+1)$,
\item $(i, j) \simeq (k, l)$ if and only if $(i, j) \ne (k, l)$, but both correspond to the same point of the quiver $\stackrel{\rightarrow}{Q}$. \qed
\end{itemize}
\end{remark}

\begin{definition}\label{D:Sets}
Let $\Omega = \Omega(\vec{\sm})$ be the set defined  by (\ref{E:SetOmega}). We introduce the following  new sets $\overline{\Omega}, \widetilde{\Omega}$ and
$\widehat{\Omega}$.

\smallskip
\noindent
1.~The set $\overline{\Omega}$ is obtained from $\Omega$ by replacing each  element $(i, j) \in \Omega$ such that $(i, j) \simeq (i, j)$ by two new elements
$((i, j), +)$ and $((i, j), -)$, which are now no longer self--equivalent.

\smallskip
\noindent
2.~The set $\widetilde\Omega$ is obtained from $\overline{\Omega}$ by replacing any pair $(i, j) \ne (k, l)$ such that $(i, j) \simeq (k, l)$, by a single  element $\overline{(i, j)} = \overline{(k, l)}$. In other words, the elements of $\widetilde\Omega$ are of the following three types:
\begin{itemize}
\item elements of the \emph{first type}: $\bigl\{(i, j), (k, l)\bigr\} = \overline{(i, j)} = \overline{(k, l)}$, where $(i, j) \simeq (k, l)$ and
$(i, j) \ne (k, l)$ in $\Omega$.
\item elements of the \emph{second type}: $\bigl((i, j),\pm)$, where $(i, j) \in \Omega$ is such that $(i, j) \simeq (k, l)$.
\item elements of the \emph{third type}: $(i, j)$, where $(i, j) \not\simeq (k, l)$ for any $(k, l) \in \Omega$.
\end{itemize}

\smallskip
\noindent
3.~The set $\widehat\Omega := \Omega/\simeq$ is the set of the equivalence classes of elements of $\Omega$.
\end{definition}

\smallskip
\noindent
We call the hereditary algebra $H = H(\vec{\sm}, \simeq)$  the  \emph{normalization} of the skew--gentle algebra $A = A(\vec{\sm}, \simeq)$.

\begin{lemma}
Let $A = A(\vec{\sm}, \simeq)$ be a skew--gentle algebra, $H$ be its normalization, $I$ be the radical of $H$.  For any $\gamma \in \widehat\Omega$ we put:
$$
\bar{H}_\gamma :=
\left\{
\begin{array}{cl}
\kk \times \kk & \mbox{\; \rm if \;} \gamma = \{(i, j), (k, l)\} \mbox{\; \rm and \;} (i, j) \ne (k, l) \\
\Mat_2(\kk) & \mbox{\; \rm if \;} \gamma = \{(i, j)\} \mbox{\; \rm and \;} (i, j) \simeq (i, j)\\
\kk & \mbox{\; \rm if \;} \gamma = \{(i, j)\} \mbox{\; \rm and \;} (i, j) \not\simeq (i, j)\\
\end{array}
\right.
$$
and
$$
\bar{A}_\gamma :=
\left\{
\begin{array}{cl}
\kk  & \mbox{\; \rm if \;} \gamma = \{(i, j), (k, l)\} \mbox{\; \rm and \;} (i, j) \ne (k, l) \\
\kk \times \kk & \mbox{\; \rm if \;} \gamma = \{(i, j)\} \mbox{\; \rm and \;} (i, j) \simeq (i, j)\\
\kk & \mbox{\; \rm if \;} \gamma = \{(i, j)\} \mbox{\; \rm and \;} (i, j) \not\simeq (i, j).\\
\end{array}
\right.
$$
Then the following results are true.
\begin{itemize}
\item $I$ is also the radical of $A$.
\item Let $\bar{A} := A/I$ and $\bar{H} := H/I$. Then we have a commutative diagram
\begin{equation}\label{E:gentle}
\begin{array}{c}
\xymatrix{
\bar{A} \ar[rr]^-{\cong} \ar@{_{(}->}[d] & & \prod\limits_{\gamma \in \widehat\Omega} \bar{A}_\gamma
\ar@{^{(}->}[d]^-{\prod\limits_{\gamma \in \widehat\Omega} \sigma_\gamma}\\
\bar{H} \ar[rr]^-{\cong}  & & \prod\limits_{\gamma \in \widehat\Omega} \bar{H}_\gamma
}
\end{array}
\end{equation}
where $\bar{A}_\gamma \stackrel{\sigma_\gamma}\lar \bar{H}_\gamma$ denotes the natural inclusion.
\end{itemize}
\end{lemma}

\begin{proof} A straightforward computation shows that
$$
I = \Bigl\{\bigl(X(1), \dots, X(t)\bigr) \in H \; \big| \; \bigl(X(i)\bigr)_{k,l} = 0 \; \mbox{\rm for all} \; 1 \le i \le t  \; \mbox{\rm and}\;
 1 \le l < k \le m_i\Bigr\}.
$$
Note  that $I$ is a two-sided ideal in $A$, $I^m = 0$ for $m := \textrm{max}\{m_1, \dots, m_t\}$ and the quotient algebra
$A/I$ is semi--simple. Therefore, $I$ is the radical of $A$, as asserted (see for instance \cite[Proposition 3.1.13]{dk}). The statement about the commutative diagram (\ref{E:gentle}) follows from a straightforward computation.
\end{proof}

\begin{remark}\label{R:skewgentleviagluing} Let $A$ be a skew--gentle algebra and $H$ be its normalization. Then
\begin{equation}\label{E:pullbackgentle}
\begin{array}{c}
\xymatrix{
A \ar[rr] \ar@{_{(}->}[d] & &  \bar{A}
\ar@{^{(}->}[d]\\
H \ar[rr]  & &  \bar{H}
}
\end{array}
\end{equation}
is a pull--back diagram in the category of $\kk$--algebras. Moreover, identifying  the embedding $\bar{A} \lar \bar{H}$ with the
map $\prod\limits_{\gamma \in \widehat\Omega} \sigma_\gamma$ from the diagram (\ref{E:gentle}), we can actually \emph{define} a skew--gentle algebra
$A$ attached to the datum $(\vec{\sm}, \simeq)$ as the pull--back of the pair of algebra homomorphisms $H \lar \bar{H} \longleftarrow \bar{A}$.
\end{remark}

\begin{remark} It is difficult to see that the sets $\overline{\Omega}$ and $\widetilde\Omega$ stand in bijections with  the sets of primitive idempotents of the algebras $H$ and $A$, respectively.
They will be used  upon Subsection \ref{SS:TriplesandMP}.
\end{remark}

\section{Homological properties of skew--gentle algebras and their applications}\label{S:ResultsSkewGentle}
Let $(\vec{\sm}, \simeq)$ be a datum as in Definition \ref{D:skewgentle}. In what follows, we shall assume that there exist $\gamma, \delta \in \Omega$ such that $\gamma \simeq \delta$. This implies that the skew--gentle algebra $A = A(\vec{\sm}, \simeq)$ is not hereditary.  One of the main objects of this paper is the following  algebra
\begin{equation}\label{E:algebraB}
B = B(\vec{\sm}, \simeq) :=
\left(
\begin{array}{cc}
A & H \\
I & H
\end{array}
\right).
\end{equation}
In this section, we study relations between the algebras $A$ and $B$.
\subsection{Recollement and skew--gentle algebras}
Let
$
e = \left(
\begin{array}{cc}
0 & 0 \\
0 & 1
\end{array}
\right) \in B$ and $
Q := B e =
\left(
\begin{array}{c}
H \\
H
\end{array}
\right).
$
Note that $\End_B(Q) \cong H^\circ$ is the opposite algebra of $H$ and $eBe = H$. Actually, we are in the setting of the so--called \emph{minors}
\cite{Drozd}. The rich theory relating the derived categories of the  algebras  $B$, $eBe$ and $B/(e)$ is due to Cline, Parshall and Scott \cite{ClineParshallScott}; see also
\cite{Minors} for some further elaborations.

\smallskip
\noindent
It is obvious that $Q$ is a free left $H^\circ$--module. Next, we have an adjoint pair
of functors
\begin{equation}\label{E:minorsadjfunct}
\xymatrix{B-\mathsf{mod} \ar@/^2ex/[rr]|{\,\sG\,} & & H-\mathsf{mod} \ar@/^2ex/[ll]|{\,\sF\,}
}
\end{equation}
 where
$\sG = \Hom_B(Q, \,-\,)$ and $\sF = Q \otimes_{H} \,-\,$. The functor $\sF$ is exact and has the following explicit description: if $X$ is a left $H$--module then
$$
\sF(X) =
\left(
\begin{array}{c}
X \\
X
\end{array}
\right) \cong X \oplus X,
$$
where for  $b = \left(\begin{array}{cc}
b_1 & b_2 \\
b_3 & b_4
\end{array}
\right) \in B$
  and   $x = \left(\begin{array}{c}
x_1 \\
x_2
\end{array}
\right) \in
\left(
\begin{array}{c}
X \\
X
\end{array}
\right),
$  the element $b \circ x$ is given by the matrix multiplication. Note also that the functor $\sH:= \Hom_{H}\bigl(eB, \,-\,\bigr): H-\mathsf{mod} \lar B-\mathsf{mod}$ is exact and right adjoint to the functor $\sG$.
Finally, we denote
$$
J :=
\left(
\begin{array}{cc}
I & H \\
I & H
\end{array}
\right) \quad \mbox{\rm and} \quad
T:= B/J \cong A/I =\bar{A}.
$$
\begin{theorem}\label{T:skewgentleResolution}
The following results are true.
\begin{enumerate}
\item The algebra $B = B(\vec{\sm}, \simeq)$ is again skew--gentle and its normalization is the algebra
\begin{equation}\label{E:normalization}
R:= R(H) =
\left(
\begin{array}{cc}
H & H \\
I & H
\end{array}
\right).
\end{equation}
Next, $L:= \mathsf{rad}(R) = \mathsf{rad}(B) = \left(
\begin{array}{cc}
I & H \\
I & I
\end{array}
\right)$. Moreover, if the original algebra $A$ is gentle then $B$ is gentle, too.
\item The derived functor $\sD\sF: D^b(H-\mathsf{mod}) \lar D^b(B-\mathsf{mod})$ of the (exact) functor $\sF$ is fully faithful. Next,
we have a semi--orthogonal decomposition
\begin{equation*}
D^b(B-\mathsf{mod}) = \bigl\langle D^b_T(B-\mathsf{mod}), \; \mathsf{Im}(\sD\sF)\bigr\rangle,
\end{equation*}
where $D^b_T(B-\mathsf{mod})$ denotes the full subcategory of $D^b(B-\mathsf{mod})$ consisting of those complexes, whose cohomologies are modules
over the semi--simple algebra $T$. Moreover, the exact functor
$D^b(\bar{A}-\mathsf{mod}) \lar D^b_T(B-\mathsf{mod}), \; \bar{A} \mapsto T$ is an equivalence of triangulated categories, i.e.~the above  semi--orthogonal decomposition of $D^b(B-\mathsf{mod})$ can be informally rewritten as
\begin{equation}\label{E:semiorthogonal}
D^b(B-\mathsf{mod}) = \bigl\langle D^b(\bar{A}-\mathsf{mod}), \; D^b(H-\mathsf{mod})\bigr\rangle.
\end{equation}
\item We have: $\mathsf{gl.dim}(B) = 2$.
\item Let $Z \in \mathrm{Ob}(H-\mathsf{mod})$ be such that $H-\mathsf{mod} = \add(Z)$ (for example, $Z$ is the direct sum of all pairwise non--isomorphic indecomposable $H$--modules) and $\widetilde{Z} := \sF(Z)$. Then  we have: $D^b(B-\mathsf{mod}) = \bigl\langle T \oplus \widetilde{Z}\bigr\rangle_2$ (here, we follow the  notation of Rouquier's work \cite{Rouquier}). In particular, the Rouquier dimension
    $\mathsf{der.dim}(B)$ of the derived category $D^b(B-\mathsf{mod})$ is at most one.
\end{enumerate}
\end{theorem}

\begin{proof} (1) Assume first that $H = T_m$. Then  a straightforward computation shows that  $R(H) \cong T_{2m}$. Moreover, if
$H \cong  H_1 \times \dots \times H_t$ then $R(H) \cong R\bigl(H_1\bigr) \times \dots \times R\bigl(H_t\bigr)$. It is  ease to see  that
\begin{itemize}
\item $L := \left(
\begin{array}{cc}
I & H \\
I & I
\end{array}
\right)$ is a two--sided ideal in both algebras  $B$ and $R$,
\item the quotient algebra $R/L \cong \bar{H} \times \bar{H}$ is semi--simple with the same number  of simple factors as $R/\mathsf{rad}(R)$.
\end{itemize}
Therefore, $L = \mathsf{rad}(R)$. Next, the algebra $B/L \cong \bar{A} \times \bar{H}$ is semi--simple, too. Moreover, $L$ is a nilpotent ideal (both in $R$ and $B$). Indeed,
$L^{2m} = 0$, where $m := \mathsf{max}\{m_1, \dots, m_t\}$. Therefore, we have: $L = \mathsf{rad}(R)$; see for instance \cite[Proposition 3.1.13]{dk}.

\smallskip
\noindent
Next, we have a commutative diagram in the category of $\kk$--algebras:
\begin{equation}\label{E:pullbackgentle2}
\begin{array}{cc}
\xymatrix{
B \ar[rr] \ar@{_{(}->}[d] & &  \bar{B}
\ar@{^{(}->}[d] \ar[rr]^-\cong & & \bar{A} \times \bar{H}  \ar@{^{(}->}[d] \\
R \ar[rr]  & &  \bar{R} \ar[rr]^-\cong & & \bar{H} \times \bar{H}.
}
\end{array}
\end{equation}
It follows from Remark \ref{R:skewgentleviagluing} that the algebra $B$ can be obtained from the hereditary algebra $R$ by the process of gluing/blowing--up idempotents described in  Definition \ref{D:skewgentle}. Therefore, $B$ is skew--gentle, as asserted. Moreover, it is gentle if and only if $A$ is gentle.

\smallskip
\noindent
(2) Recall that $\sG: B-\mathsf{mod} \lar H-\mathsf{mod}$ is a localization functor, whose kernel is the category of $B/J$--modules, where
\begin{equation}\label{E:bimoduleJ}
J:= \mathsf{Im}\bigl(Be \otimes_{eBe} eB \xrightarrow{\mathsf{mult}} B\bigr) =
\left(\begin{array}{cc}
I & H \\
I & H
\end{array}
\right),
\end{equation}
and the unit of adjunction $\mathsf{Id}_{H-\mathsf{mod}} \stackrel{\eta}\lar \sG \circ \sF$ is an isomorphism of functors; see for instance \cite[Theorem 4.3]{Minors}. Since both functors
$\sG$ and $\sF$ are exact, the unit $\mathsf{Id}_{D^b(H-\mathsf{mod})} \stackrel{\eta}\lar \sD\sG \circ \sD\sF$  of the derived adjoint pair $(\sD\sF, \sD\sG)$ is an isomorphism of functors, too. Hence, the functor $\sD\sF$ is fully faithful; see for instance  \cite[Theorem 4.5]{Minors}.

\smallskip
\noindent
Let $ \sD\sF \circ \sD\sG  \stackrel{\xi}\lar \mathsf{Id}_{D^b(H-\mathsf{mod})}$ be the adjunction counit. Then for any object $X^{\bu}$ of
$D^b(B-\mathsf{mod})$, we have a distinguished triangle
\begin{equation}\label{E:disttriangle}
\sD\sF \circ \sD\sG(X^\bu) \xrightarrow{\xi_{X^\bu}} X^\bu \lar \mathsf{Cone}\bigl(\xi_{X^\bu}\bigr) \lar \sD\sF \circ \sD\sG(X^\bu)[1].
\end{equation}
Since the morphism $\sD\sG\bigl(\xi_{X^\bu}\bigr): \sD\sG \circ \sD\sF \circ \sD\sG(X^\bu) \lar \sD\sG(X^\bu)$ is an isomorphism, $\mathsf{Cone}\bigl(\xi_{X^\bu}\bigr)$
belongs to the kernel of $\sD\sG$, which is the triangulated category $D^b_T(B-\mathsf{mod})$.
For any object $Y^\bu$ of $D^b(H-\mathsf{mod})$ and $Z^\bu$ of $D^b_T(B-\mathsf{mod})$, we have:
$$
\Hom_{D^b(B)}\bigl(\sD\sF(Y^\bu), \, Z^\bu\bigr) \cong \Hom_{D^b(H)}\bigl(Y^\bu, \, \sD\sG(Z^\bu)\bigr) \cong 0.
$$
Therefore, we indeed have a semi--orthogonal decomposition (\ref{E:semiorthogonal}), as asserted.
Moreover, even a stronger property is true: we have a \emph{recollement diagram} (see \cite{BBD,ClineParshallScott}):
\begin{equation}\label{E:recollement}
\xymatrix{D^b_T(B-\mathsf{mod}) \ar[rr]|{\,\sI\,} && D^b(B-\mathsf{mod}) \ar@/^2ex/[ll] \ar@/_2ex/[ll] \ar[rr]|{\,\sD\sG\,}
  && D^b(H-\mathsf{mod}) \ar@/^2ex/[ll]^{\,\sD\sH\,} \ar@/_2ex/[ll]_{\,\sD\sF\,}}.
\end{equation}
Here, $\sI$ is the natural inclusion functor, admitting both left and right adjoint functors; see \cite{ClineParshallScott}. Since the $B$--bimodule $J$ is projective viewed as a right $B$--module (we have: $J \cong eB \oplus eB$), the natural exact functor given as the
composition
$$
D^b(\bar{A}-\mathsf{mod}) \stackrel{\cong}\lar D^b(B/J-\mathsf{mod}) \lar D^b_T(B-\mathsf{mod}), \quad \bar{A} \mapsto T
$$
is an equivalence of triangulated categories; see for instance \cite[Theorem 4.6]{Minors}.

\smallskip
\noindent
(3) Applying  \cite[Lemma 5.1]{Minors} to our setting, we obtain:  $\mathsf{gl.dim}(B) \le 2$. Since by our assumptions, the algebra $A$ is not hereditary, the
algebra $B$ is not hereditary, too. Therefore, we may conclude that $\mathsf{gl.dim}(B) = 2$.

\smallskip
\noindent
(4) Since the algebra $H$ is Morita equivalent to $T_{m_1} \times \dots  \times T_{m_t}$, it has finite (derived) representation type.
Let $Z$  be a representation generator of  $H-\mathsf{mod}$, $\widetilde{Z} := \sF(Z)$, $Y := T \oplus \widetilde{Z}$ and
$X^\bu$ be any object of $D^b\bigl(B-\mathsf{mod}\bigr)$. According to the exact triangle (\ref{E:disttriangle}), there exist objects $X^\bu_1, X^\bu_2 \in
\bigl\langle Y\bigr\rangle_1$ (i.e.~$X^\bu_1$ and  $X^\bu_2$ are direct sums of shift of direct summands of $Y$; see \cite{Rouquier} for the notation)
fitting into an exact triangle
$$
X^\bu_1 \lar X^\bu \lar X^\bu_2 \lar X^\bu_1[1].
$$
But this precisely  means that $D^b(B-\mathsf{mod}) = \bigl\langle Y\bigr\rangle_2$, hence  $\mathsf{der.dim}(B) \le 1$.
\end{proof}

\begin{example}
Let $A$ be the gentle algebra from Example \ref{Ex:MyFavorite}. Then the corresponding gentle algebra $B$ is the path algebra of the quiver
$$
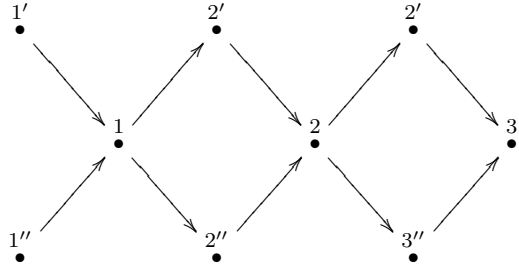
\xymatrix{
\stackrel{1'}\bul \ar[rd]    &                        & \stackrel{2'}\bul \ar[rd] & & \stackrel{2'}\bul \ar[rd] &  \\
 & \stackrel{1}\bul \ar[ru] \ar[rd] &  & \stackrel{2}\bul \ar[ru] \ar[rd] & & \stackrel{3}\bul  \\
\stackrel{1''}\bul \ar[ru] & & \stackrel{2''}\bul \ar[ru] & & \stackrel{3''}\bul \ar[ru] &
}
$$
subject to the following set of relations: any path starting  at a vertex from from the set $\{1', 2', 3'\}$ and ending at a vertex of the set $\{1'', 2'', 3''\}$ is zero, and  conversely, any path starting  at a vertex from  the set $\{1'', 2'', 3''\}$ and ending at a vertex of the set $\{1', 2', 3'\}$ is zero.
\end{example}

\begin{remark}\label{R:GPRevisited}
Let $N:= \kk\llbracket u, v\rrbracket/(uv)$ and
$$
C:= \End_N\bigl(N \oplus N/(v)\bigr)^\circ \cong
\left(
\begin{array}{cc}
\kk\llbracket u, v\rrbracket/(uv) & \kk\llbracket u\rrbracket \\
u \kk\llbracket u\rrbracket  & \kk\llbracket u\rrbracket \\
\end{array}
\right).
$$
Then $C$ is isomorphic to the completed path algebra from Introduction, studied by Gelfand and Ponomarev in \cite{GP}. It is not difficult to show that
$\mathsf{gl.dim}(C) = 3$. In fact, $C$ is the so--called \emph{cluster--tilting} resolution of singularities of $N$. The algebra $N$ is a minor of the algebra $C$ and the module categories of $N$ and $C$ as well as their derived categories are related in a similar way as in the case of the (skew--)gentle algebras $A$ and $B$; see
\cite{ClineParshallScott, Minors}.
\end{remark}

\smallskip
\noindent
As a consequence of Theorem \ref{T:skewgentleResolution}, we get the following result  for the Rouquier dimension of skew--gentle algebras.

\begin{theorem}\label{T:RouquierDimSkewGentle}
Let $A$ be an arbitrary skew--gentle algebra, $H$ be its normalization and $Z$ be a representation--generator of $H-\mathsf{mod}$, viewed as an $A$--module. Then we have:
$D^b(A-\mathsf{mod}) = \langle Z\rangle_2$, i.e.~for any object $X^\bu$ of $D^b(A-\mathsf{mod})$, there exists an exact triangle
 $$
 X^\bu_1 \lar X^\bu \lar X^\bu_2 \lar X^\bu_1[1],
 $$
in which  the complexes $X_1^\bu$ and $X_2^\bu$ are appropriate direct sums of certain shifts of some direct summands of $Z$.
 In particular, $\mathsf{der.dim}(A) \le 1$.
\end{theorem}

\begin{proof}
We apply again the construction of minors, but now to another  natural idempotent
$
f = \left(
\begin{array}{cc}
1 & 0 \\
0 & 0
\end{array}
\right)$ of the algebra $B$. Let  $
P := B f =
\left(
\begin{array}{c}
A \\
I
\end{array}
\right).
$
Then we have: $fBf = A \cong \End_B(P)^\circ$. Next, $\widetilde\sG:= \Hom_B(P, \,-\,): B-\mathsf{mod} \lar A-\mathsf{mod}$ is a bilocalization functor (between abelian categories), and its derived functor $$\sD\widetilde\sG: D^b(B-\mathsf{mod}) \lar D^b(A-\mathsf{mod})$$ is a bilocalization functor (between triangulated categories). Since  $\sD\tilde\sG$ is essentially surjective, we have:
$D^b(A-\mathsf{mod}) =  \bigl\langle \sD\widetilde\sG(Y)\bigr\rangle_2$, where $Y = T \oplus \widetilde{Z}$ is the left $B$--module introduced in Theorem \ref{T:skewgentleResolution}. Next, observe that the composition $\widetilde\sG \circ \sF: H-\mathsf{mod} \lar A-\mathsf{mod}$ is isomorphic to the forgetful functor. Indeed, for any left $H$--module $X$ we have:
$$
\widetilde\sG \circ \sF(X) = \Hom_B\left(Bf,
\left(
\begin{array}{c}
X \\
X
\end{array}
\right)
 \right) \cong f \cdot
 \left(\begin{array}{c}
X \\
X
\end{array}\right) \cong X.
$$
Similarly,
$
\widetilde\sG(T) = \widetilde\sG
\left(
\begin{array}{c}
\bar{A} \\
0
\end{array}
\right) \cong \bar{A}.
$
Since both functors $\widetilde\sG$ and $\sF$ are exact, we have:
$$\sD\widetilde\sG \circ \sD\sF(Y) \cong \widetilde\sG \circ \sF(Y) \cong \bar{A} \oplus Z.$$
Since we have an injective homomorphism of semi--simple algebras $\bar{A} \lar \bar{H}$, the semi--simple left $A$--module $\bar{A}$ is a direct summand of $Z$, what finishes the proof.
\end{proof}

\begin{corollary}\label{C:RouquierDimSkewGentle}
Let $A$ be a skew--gentle  algebra of infinite derived representation type. Then we have: $\mathsf{der.dim}(A) =  1$.
\end{corollary}

\begin{proposition}\label{C:RouquierDimReprDiscr}
Let $A$ be an (indecomposable)  $\kk$--algebra of discrete derived representations type. Then we have:
\begin{itemize}
\item $\mathsf{der.dim}(A) =  0$ if $A$ is derived equivalent to a hereditary algebra of Dynkin type.
\item $\mathsf{der.dim}(A) =  1$ otherwise.
\end{itemize}
\end{proposition}

\begin{proof} This is a consequence of Corollary \ref{C:RouquierDimSkewGentle} combined with Vossieck's classification of finite dimensional algebras of discrete derived representation type \cite{Vossieck} (see also \cite{BobinskiGeissSkowronski}). The latter classification asserts that
 an indecomposable algebra $A$ has discrete derived representation type if and only if  it is either derived equivalent to the path algebra of a Dynkin quiver
or to a gentle algebra of a very special form. In the first case, there are only finitely many (up to shifts) indecomposable objects in
$D^b(A-\mathsf{mod})$, hence $\mathsf{der.dim}(A) =  0$. Otherwise, there are infinitely many (again, up to shifts)  indecomposable objects
of $D^b(A-\mathsf{mod})$. Therefore, we have: $\mathsf{der.dim}(A) =  1$.
\end{proof}

\begin{example}\label{Ex:DualNumbers}
Consider the gentle algebra $A = \kk[\varepsilon]/(\varepsilon^2)$. Then we have: $\mathsf{der.dim}(A) =  1$ and
$D^b(A-\mathsf{mod}) = \langle \kk\rangle_2$. To show this, we actually do not need Proposition  \ref{C:RouquierDimReprDiscr}. Indeed, it is well--known that any indecomposable object of $D^b(A-\mathsf{mod})$ is  isomorphic (up to a shift)
either to $\kk$ or to a complex of the form
\begin{equation}\label{E:indeccomplex} X_n := \bigl( \dots \lar 0 \lar \underbrace{A \stackrel{\varepsilon}\lar A \stackrel{\varepsilon}\lar \dots \stackrel{\varepsilon}\lar A}_{n \; \mathrm{times}} \lar 0 \lar \dots\bigr).
\end{equation}
For any $n \in \NN$, let $w_n$ be a generator of the $\kk$--vector space $\Ext_A^n(\kk, \kk) \cong \kk$. Since
$X_n \cong \mathsf{Cone}\bigl(\kk \stackrel{w_n}\lar \kk[n]\bigr)$, we have the statement.
\end{example}

\begin{remark}\label{R:MoreGeneralSetting} The proofs of all results of this section can be generalized on the following setting. Consider an algebra  $H = H_1 \times \dots \times H_t$, where
$H_i$ is derived--equivalent to the path algebra of a Dynkin quiver for any $1 \le i \le t$. Let $A \subset H$ be a subalgebra such that $I:= \mathsf{rad}(H) = \mathsf{rad}(A)$. Then the algebra $B :=
\left(
\begin{array}{cc}
A & H \\
I & H
\end{array}
\right)$ has finite global dimension. Moreover, let $Z$ be the direct sum of all (up to a shift)  indecomposable objects of $D^b\bigl(H-\mathsf{mod})$. Viewing $Z$ as an object of  $D^b\bigl(A-\mathsf{mod})$, we have:
$D^b(A-\mathsf{mod}) = \langle Z\rangle_2$. As a consequence, we get: $\mathsf{der.dim}(A) \le 1$.
\end{remark}

\begin{example}\label{Ex:fatpointRouqDim}
For any $n \in \NN$, let $A = \kk[\varepsilon_1, \dots, \varepsilon_n]/(\varepsilon_1, \dots, \varepsilon_n)^2$. Then we have: $\mathsf{der.dim}(A) = 1$.
Indeed, since $A$ has infinite derived representation type, we have: $\mathsf{der.dim}(A) \ge 1$. On the other hand, we have an algebra extension
$
A \subset H:= \underbrace{T_2 \times \dots \times T_2}_{n \; \mathrm{times}},
$
satisfying assumptions of Remark \ref{R:MoreGeneralSetting}. For any $1 \le i \le n$, let $$R_i:= A/(\varepsilon_1, \dots, \hat{\varepsilon}_i, \dots, \varepsilon_n) \cong \kk[\varepsilon_i]/(\varepsilon_i^2).$$ Then we have:
$
D^b(A-\mathsf{mod}) = \bigl\langle \kk \oplus R_1 \oplus \dots \oplus R_n\bigr\rangle_2,
$
implying the statement.
\end{example}

\begin{conjecture} Let $A$ be a finite dimensional algebra of tame derived representation type. Then we have: $\mathsf{der.dim}(A) =  1$.
\end{conjecture}

\subsection{Rouquier dimension of tame projective curves}
Let $X$ be a projective scheme over the field $\kk$ and $\mathsf{der.dim}(X)$ be the Rouquier dimension of the derived category $D^b\bigl(\Coh(X)\bigr)$ of coherent sheaves on $X$. It was shown by Rouquier \cite{Rouquier} that
$
\mathsf{dim}(X) \leq \mathsf{der.dim}(X) < \infty.
$
Moreover, for a smooth projective curve  $X$, Orlov \cite{Orlov} has shown that $\mathsf{der.dim}(X) = 1$. In the case of singular curves,
some upper bounds for $\mathsf{der.dim}(X)$ were obtained in  \cite[Theorem 10]{bd} and  \cite[Corollary 7.2]{BurbanDrozdGavranIMRN}. However, to our best knowledge, apart of the algebra $\kk[\varepsilon]/(\varepsilon^2)$ (see Example \ref{Ex:DualNumbers}), no other precise values for the Rouquier dimension of  a singular projective scheme were known so far.

\smallskip
\noindent
 Now, let $X$ be a reduced projective curve over $\kk$. According to \cite{DGVB,Duke}, the derived category $D^b\bigl(\Coh(X)\bigr)$ has tame representation type if and only if
\begin{itemize}
\item $X$ is smooth and the genus of $X$ is at most one (the latter case follows from Atiyah's work \cite{At}).
\item $X$ is either a chain  or a cycle of several copies of $\PP^1$; see \cite{DGVB,Duke}.
\end{itemize}

\begin{theorem}\label{T:RoquierDimSingCurves}
Let $X$ be a tame projective curve. Then we have: $\mathsf{der.dim}(X) = 1$.
\end{theorem}

\begin{proof} For $X$ smooth (i.e. $X = \PP^1$ or an elliptic curve), this is a special case of Orlov's result \cite{Orlov}. In \cite[Section 6]{bd}
it was shown that for a singular tame curve $X$
there exists a \emph{gentle} algebra $\Lambda_X$ and an essentially surjective functor
\begin{equation}\label{E:essentiallysurjective}
D^b\bigl(\Lambda_X-\mathsf{mod}\bigr) \stackrel{\sP}\lar D^b\bigl(\Coh(X)\bigr).
\end{equation}
Let $Z$ be a representation generator of the normalization of the algebra $\Lambda_X$. According to  Theorem \ref{T:RouquierDimSkewGentle} we have:  $D^b\bigl(\Coh(X)\bigr) = \bigl\langle \sP(Z)\bigr\rangle_2$, implying the result.
\end{proof}

\smallskip
\noindent
It is natural to ask about an \emph{explicit} description of the  generator $\sP(Z)$ of the derived category $D^b\bigl(\Coh(X)\bigr)$. To answer this question, we  need a better understanding of the functor $\sP$ from the proof
of Theorem \ref{T:RoquierDimSingCurves}. In the following example, we give  an answer to a question, posed to the first--named author by Rouquier in 2004.

\begin{example}\label{Ex:RouqDimWeierCubic}
 Let $E := \overline{V(y^2-x^3-x^2)} \subset \PP^2$ be a nodal Weierstra\ss{} cubic, $s = (0, 0) \in E$ its unique singular point and
$\PP^1 \stackrel{\nu}\lar E$ be its  normalization map. Let $\kO_E$ and $\kO_{\PP^1}$ be the structure sheaves of $E$ and $\PP^1$, respectively.
We denote $\kO := \kO_E$ and put $\widetilde{\kO}(n) := \nu_*\bigl(\kO_{\PP^1}(n)\bigr)$ for any $n \in \ZZ$.
Then we claim that
\begin{equation}
D^b\bigl(\Coh(E)\bigr) = \bigl\langle \kk_s \oplus \widetilde\kO \oplus \widetilde{\kO}(1)\bigr\rangle_2.
\end{equation}
In other words, for any object $\kF^\bu$ of $D^b\bigl(\Coh(E)\bigr)$ there exists an exact triangle
$$
\kF^\bu_1 \lar \kF^\bu \lar \kF^\bu_2 \lar \kF^\bu_1[1],
$$
where $\kF^\bu_1$ and $\kF^\bu_2$ are direct sums of appropriate shifts of coherent sheaves $\kk_s$,  $\widetilde\kO$ and $\widetilde{\kO}(1)$.
The explicit description of the functor $\sP$ from Theorem \ref{T:RoquierDimSingCurves} includes the following key intermediate ingredient. Let
$\kI := \mathit{Ann}_E\bigl(\widetilde\kO/\kO\bigr) \cong \mathit{Hom}_E\bigl(\widetilde\kO, \kO)$ be the conductor ideal. Then we have a sheaf of $\kO$--orders
$
\kB :=
\left(
\begin{array}{cc}
\kO & \widetilde\kO \\
\kI & \widetilde\kO
\end{array}
\right)
$
(called in \cite{bd} \emph{Auslander order} of $E$)
as well as  the associated non--commutative nodal curve $\mathbbm{E} = (E, \kB)$. The categories of coherent sheaves $\Coh(\mathbbm{E})$ and $\Coh(E)$ stand in the same relation
as (skew)--gentle algebras $A$ and $B$ from Theorem \ref{T:skewgentleResolution}.
In particular, consider the following idempotent sections
$f:=
\left(
\begin{array}{cc}
1 & 0 \\
0 & 0
\end{array}
\right) \; \mbox{\rm and} \; e:=
\left(
\begin{array}{cc}
0 & 0 \\
0 & 1
\end{array}
\right)
\in \Gamma(E, \kB)$. Then $$\kP := \kB f \cong
\left(
\begin{array}{c}
\kO \\
\kI
\end{array}
\right) \; \mbox{\rm and} \; \kQ := \kB e \cong
\left(
\begin{array}{c}
\widetilde{\kO} \\
\widetilde{\kO}
\end{array}
\right)
$$
 are  locally projective sheaves of left $\kB$--modules and  $f \kB f \cong \kO \cong \mathit{End}_\kB(\kP)$, whereas
$e \kB e \cong \widetilde{\kO} \cong \mathit{End}_\kB(\kQ)$. Then we have two exact  bilocalization functors
$$
\widetilde\sG:=
\mathit{Hom}_\kB(\kP, \,-\,) \; \mbox{\rm and}\; \sG:=
\mathit{Hom}_\kB(\kQ, \,-\,): \;  \Coh(\mathbbm{E}) \lar \Coh(E).
$$
The kernels of $\widetilde\sG$ and $\sG$ can be identified with the  categories of finite--dimensional $\kk$--modules over the algebras $\Gamma(E, \kB/\kL)$ and $\Gamma(E, \kB/\kJ)$ respectively, where
$$
\kL:= \mathsf{Im}\bigl(\kB f \otimes_\kO f \kB \xrightarrow{\mathsf{mult}} \kB\bigr) =
\left(
\begin{array}{cc}
\kO & \widetilde\kO \\
\kI & \kI
\end{array}
\right) \; \mbox{\rm and} \; \kJ:= \mathsf{Im}\bigl(\kB e \otimes_\kO e \kB \xrightarrow{\mathsf{mult}} \kB\bigr) =
\left(
\begin{array}{cc}
\kI & \widetilde\kO \\
\kI & \widetilde\kO
\end{array}
\right).
$$
It follows that $\kB/\kL \cong \widetilde\kO/\kI \cong (\kk \times \kk)_s$ and $\kB/\kJ \cong \kO/\kI \cong \kk_s$. Moreover, the  functor
$\sG$ (respectively, $\widetilde\sG$) admits left and right adjoint functors $\sF$ and $\sH$ (respectively, $\widetilde\sF$ and $\widetilde\sH$).

Since the sheaf of two--sided $\kB$--modules $\kJ$ is locally projective, viewed as a  right $\kB$--module,
analogously to (\ref{E:recollement})
we get a recollement diagram
\begin{equation}\label{E:recollementsheaves}
\xymatrix{D^b\bigl(\mathsf{Vect}(\kk)\bigr) \ar[rr]|{\,\sI\,} && D^b\bigl(\Coh(\mathbbm{E})\bigr) \ar@/^2ex/[ll]  \ar@/_2ex/[ll] \ar[rr]|{\,\sD\sG\,}
  && D^b\bigl(\Coh(\PP^1)\bigr), \ar@/^2ex/[ll]^{\,\sD\sH\,} \ar@/_2ex/[ll]_{\,\sD\sF\,}}
\end{equation}
as well as the associated  semi--orthogonal decomposition $$
D^b\bigl(\Coh(\mathbbm{E})\bigr) = \bigl\langle D^b\bigl(\mathsf{Vect}(\kk), D^b\bigl(\Coh(\PP^1)\bigr)\bigr\rangle.$$
Let $\kS_* := \sI(\kk)$. Then  the complex  $\kH^\bu:= \kS_\ast \oplus \kQ(-1) \oplus \kQ$ is a tilting object in the category $D^b\bigl(\Coh(\mathbbm{E})\bigr)$
$\bigl($i.e.~$\Hom_{D^b(\mathbbm{E})}\bigl(\kH^\bu, \kH^\bu[n]\bigr) = 0$ for $n \ne 0$ and $\kH^\bu$ generates $D^b\bigl(\Coh(\mathbbm{E})\bigr)\bigr)$ and its
endomorphism algebra $\End_{D^b(\mathbbm{E})}\bigl(\kH^\bu\bigr)$ is isomorphic to the gentle algebra $A$ from Example \ref{Ex:MyFavorite}. As a consequence,
we obtain  an exact equivalence of triangulated categories
$
D^b\bigl(A-\mathsf{mod}\bigr) \stackrel{\sT}\lar  D^b\bigl(\Coh(\mathbbm{E})\bigr).
$
Composing  $\sT$ with  $\sD\widetilde\sG$, we get an essentially surjective functor
$D^b(A-\mathsf{mod}) \stackrel{\sP}\lar D^b\bigl(\Coh(E)\bigr)$, which is moreover a bilocalization functor.
Everything can be summarized by the following diagram of categories and functors:
$$
\xymatrix{
\Perf(E) \ar@{^{(}->}[rr]^-{\sL \widetilde\sF} \ar@{_{(}->}[d] & & D^b\bigl(\Coh(\mathbbm{E})\bigr)  \ar[dll]_-{\sD\widetilde\sG}\\
D^b\bigl(\Coh(E)\bigr) & & D^b(A-\mathsf{mod}). \ar[ll]_{\sP} \ar[u]_-{\sT}
}
$$
For any vertex  $1 \le i \le 3$, let $S_i$ be the corresponding simple module, $P_i$ the corresponding indecomposable projective module and
$I_i$ the corresponding indecomposable injective  module over the algebra $A$. Obviously, $S_1 = I_1$ and $S_3 = P_3$. To get a desired generator
of the derived category $D^b\bigl(\Coh(E)\bigr)$, we have to compute the images of the following  objects of the category $A-\mathsf{mod}$ under the functor
$\sP$:
\begin{itemize}
\item of the  simple modules $S_1$, $S_2$ and $S_3$;
\item of the following two--dimensional representations $$U_+ := \xymatrix{
\kk   \ar@/^/[r]^{1} \ar@/_/[r]_{0}  & \kk \ar@/^/[r] \ar@/_/[r] & 0
}
\; \mbox{\rm and} \; V_+:= \xymatrix{
0  \ar@/^/[r] \ar@/_/[r]    & \kk   \ar@/^/[r]^{1} \ar@/_/[r]_{0}& \kk
}
$$
as well as of their flips $U_-$ and $V_-$ with respect to the natural involution of the algebra $A$;
\item of the following three--dimensional representation
$W_+:= \xymatrix{
\kk   \ar@/^/[r]^{1} \ar@/_/[r]_{0}  & \kk \ar@/^/[r]^{1} \ar@/_/[r]_{0}  & \kk
}
$
as well as of its flip $W_-$.
\end{itemize}
Now, let us proceed with the computations.
\begin{itemize}
\item It follows from the description of the tilting equivalence $\sT$ that
$$\sT(S_3) = \sT(P_3) = \kS_\ast[-1].$$
 Therefore, $\sP(S_3) = \sD\sG\bigl(\kS_*[-1]\bigr) \cong  \sG\bigl(\kS_*\bigr)\bigl[-1\bigr] = \kk_s[-1]$.  Next, since $\sT$ is an
  equivalence of categories, we have:
  $
  \sT \circ \sS_A \cong \sS_{\mathbbm{E}} \circ \sT,
  $
  where $\sS_A$ and $\sS_{\mathbbm{E}}$ are Serre functors of $D^b(A-\mathsf{mod})$ and $D^b\bigl(\Coh(\mathbbm{E})\bigr)$, respectively.
  We have: $\sS_A(P_i) \cong I_i$ for any $1 \le i \le 3$. It follows from an explicit description of the Serre functor $\sS_\mathbbm{E}$ given by
  \cite[Theorem 3]{bd} that
  $$
  \sT(S_1) = \sT(I_1) \cong \sS_{\mathbbm{E}}(\kQ) \cong
  \left(
  \begin{array}{c}
  \widetilde\kO\\
  \kI
  \end{array}
  \right)[1].
  $$
 As a consequence, $\sP(S_1) \cong \sD\sG\left(\left(
  \begin{array}{c}
  \widetilde\kO\\
  \kI
  \end{array}
  \right)[1]\right) \cong \sG\left(\left(
  \begin{array}{c}
  \widetilde\kO\\
  \kI
  \end{array}
  \right)\right)[1] \cong \widetilde{\kO}[1].
  $
  \item It follows from \cite[Proposition 12]{bd} that the objects $\sT(W_\pm)$ belong to the kernel of the functor $\sD\sG$. Therefore,
  $\sP(W_\pm) = 0$.
 \item The remaining computations can be performed in an analogous way as in \cite[Section 7]{bd}. Let us just state the final answer:
$$
\sP(V_\pm) \cong \widetilde\kO,\; \sP(U_\pm) \cong \kk_s \; \mbox{\rm and}\; \sP(S_2) \cong \widetilde\kO(1),
$$
which imply the claim. \qed
\end{itemize}
\end{example}

\section{Skew--gentle algebras and matrix problems}\label{S:matrixproblems}
Let $A \subset H$ be an extension of finite--dimensional  algebras such that $I:= \mathsf{rad}(A) = \mathsf{rad}(H)$.
We denote $\bar{A} = A/I$ and $\bar{H} = H/I$.
For $C \in \bigl\{A, H, \bar{A}, \bar{H}\bigr\}$, let $C-\mathsf{pro}$ be the category of finite dimensional projective $C$--modules. Next, for
any $\ast \in \{+, -, b, \, \emptyset\}$, let
$\mathsf{Hot}^\ast\bigl(C-\mathsf{pro}\bigr)$ be the corresponding homotopy category.

\subsection{Category of triples}
Obviously, we have a pair of $\kk$--linear functors:
\begin{equation}\label{E:CommaCat}
\mathsf{Hot}^\ast\bigl(H-\mathsf{pro}\bigr)
 \xrightarrow{\bar{H}\otimes_{H}\,-\,}  \mathsf{Hot}^\ast\bigl(\bar{H}-\mathsf{pro}\bigr)  \xleftarrow{\bar{H}\otimes_{\bar{A}}\,-\,}
 \mathsf{Hot}^\ast\bigl(\bar{A}-\mathsf{pro}\bigr).
\end{equation}
\begin{definition}\label{D:triples}
The
\emph{category of triples} $\Tri^\ast(A)$ is  a full subcategory of the comma category associated with the pair of functors
$\bigl(\bar{H}\otimes_{\bar{A}}\,-\,, \bar{H}\otimes_{H}\,-\,\bigr)$. Namely, its objects
are  triples $(Y^\bu, V^\bu, \theta)$, where
\begin{itemize}
\item $Y^\bu$
is an object of $\mathsf{Hot}^\ast\bigl(H-\mathsf{pro}\bigr)$.
\item $V^\bu$ is an object of $\mathsf{Hot}^\ast\bigl(\bar{A}-\mathsf{pro}\bigr)$.
\item
$\bar{H} \otimes_{\bar{A}} V^\bu \stackrel{\theta}\lar \bar{H} \otimes_{H}
Y^\bu$ is an \emph{isomorphism} in $\mathsf{Hot}^\ast\bigl(\bar{H}-\mathsf{pro}\bigr)$ (called \emph{gluing map}).
\end{itemize}
A morphism $(Y_1^\bu, V_1^\bu, \theta_1) \lar (Y_2^\bu, V_2^\bu, \theta_2)$ in $\Tri^\ast(A)$ is given by a pair $(g, h)$, where
\begin{itemize}
\item $Y_1^\bu \stackrel{g}\lar Y_2^\bu$ is a morphism in  $\mathsf{Hot}^\ast\bigl(H-\mathsf{pro}\bigr)$  and
\item $V_1^\bu \stackrel{h}\lar V_2^\bu$ is a morphism in $\mathsf{Hot}^\ast\bigl(\bar{A}-\mathsf{pro}\bigr)$
\end{itemize}
such that
the following diagram
\begin{equation}\label{E:morphismstriples}
\begin{array}{c}
\xymatrix
{
\bar{H} \otimes_{\bar{A}} V_1^\bu \ar[rr]^-{\theta_1} \ar[d]_-{\mathsf{id} \otimes h} & &
\bar{H} \otimes_{H} Y_1^\bu \ar[d]^-{\mathsf{id} \otimes g}\\
\bar{H} \otimes_{\bar{A}} V_2^\bu \ar[rr]_-{\theta_2} & &
\bar{H} \otimes_{H} Y_2^\bu
}
\end{array}
\end{equation}
is commutative in $\mathsf{Hot}^\ast\bigl(\bar{H}-\mathsf{pro}\bigr)$.
\end{definition}

\begin{theorem}\label{T:mainconstr}
The functor $\mathsf{Hot}^\ast\bigl(A-\mathsf{pro}\bigr)
 \stackrel{\EE}\lar \Tri^\ast(A), \; X^\bu \stackrel{\EE}\longmapsto
\bigl(H \otimes_A X^\bu, \bar{A}\otimes_A X^\bu, \theta_{X^\bu}\bigr)
$, where  $\bar{H}\otimes_{\bar{A}}  ( \bar{A} \otimes_A X^\bu)
\stackrel{\theta_{X^\bu}}\lar \bar{H}\otimes_{H} (H \otimes_A X^\bu)$ is  the canonical isomorphism in the category $\mathsf{Hot}^\ast\bigl(\bar{H}-\mathsf{pro}\bigr)$, has the following properties:
\begin{itemize}
\item $\EE$ is full and essentially surjective.
\item $\EE(X_1^\bu) \cong \EE(X_2^\bu)$ if and only if $X_1^\bu \cong X_2^\bu$.
\item $\EE(X^\bu)$ is indecomposable if and only if $X^\bu$ is indecomposable.
\end{itemize}
\end{theorem}

\begin{proof} We give an update of the proof of \cite[Theorem 2.4]{Nodal}, where this result was stated and proven for $* \in \{b, -\}$.

\smallskip
\noindent
\underline{Step 1}.
Consider first the category of triples $\mathsf{tri}(A)$ attached to the pair of $\kk$--linear functors
\begin{equation}\label{E:CommaCatTwo}
H-\mathsf{pro}
 \xrightarrow{\bar{H}\otimes_{H}\,-\,}  \bar{H}-\mathsf{pro} \xleftarrow{\bar{H}\otimes_{\bar{A}}\,-\,}
 \bar{A}-\mathsf{pro}.
\end{equation}
We claim that the functor $$A-\mathsf{pro} \stackrel{\EE}\lar \mathsf{tri}(A), \; P \mapsto \bigl(H \otimes_A P, \bar{A} \otimes_A P, \theta_{P}\bigr)$$
is an equivalence of categories. Indeed, let $(Q, V, \theta)$ be on object of $\mathsf{tri}(A)$. Then we define an $A$--module $X$ via the follwing pull--back diagram in the category of $A$--modules:
\begin{equation}\label{E:reconstrPrelim}
\begin{array}{c}
\xymatrix
{ 0 \ar[r] & I Q  \ar[r] \ar@{=}[d] & X \ar[r]^p \ar[d]_\imath   & V
\ar[d]^-{\widetilde{\theta}} \ar[r] & 0 \\
0 \ar[r] & I Q  \ar[r] & Q
\ar[r]^{\pi} & \bar{Q}
\ar[r] & 0,
}
\end{array}
\end{equation}
where $\widetilde\theta$ is the morphism corresponding to the gluing map $\theta$ via the adjunction isomorphism
$\Hom_{\bar{H}}\bigl(\bar{H}\otimes_{\bar{A}} V, \bar{Q}\bigr) \cong \Hom_{\bar{A}}\bigl(V, \bar{Q}\bigr)$. Equivalently, we may define $X$ via the short exact sequence
\begin{equation}\label{E:exactseqreconstr}
0 \lar X \lar Q  \oplus V  \xrightarrow{(-\pi \; \tilde\theta)} \bar{Q} \lar 0.
\end{equation}
The universal property of kernel (or pull--back) implies that the assignment
$$
\mathsf{tri}(A) \stackrel{\sfK}\lar A-\mathsf{mod}, \; (Q, V, \theta) \mapsto X
$$
is in fact functorial. Consider the triple $T_A := (H, \bar{A}, \vartheta_A)$, where $\vartheta_A$ is the canonical isomorphism $\bar{H} \otimes_{\bar{A}}
\bar{A} \lar \bar{H} \otimes_H H$. Obviously,
$\EE(A) \cong T_A$ and $\sfK(T_A) \cong A$. In can be shown  that $\mathsf{tri}(A) = \mathsf{add}\bigl(T_A\bigr)$. Hence, $\EE$ and $\sfK$ are quasi--inverse equivalences of categories.

\smallskip
\noindent
\underline{Step 2}.
Let $(Y^\bu, V^\bu, \theta)$
be an object of $\Tri^\ast(A)$.  Without loss of generality we may assume that both complexes $Y^\bu$ and $V^\bu$ are minimal. Since
the algebra $\bar{A}$ is semi--simple,
\begin{equation}\label{E:complexV}
V^\bu = \bigl( \dots \stackrel{0}\lar V^{r-1} \stackrel{0}\lar V^{r} \stackrel{0}\lar V^{r+1} \stackrel{0}\lar \dots \bigr)
\end{equation}
is a complex with zero differentials. Since the differentials of the complex $Y^\bu$ have coefficients in $I$, the complex $\bar{Y}^\bu:= \bar{H}\otimes_H Y^\bu$ has   zero differentials, too. For any $r \in \ZZ$, consider the $r$--th component of the gluing morphism $\theta$
\begin{equation}\label{E:CompnentRofTheta}
\theta^r: \;  \bar{H} \otimes_{\bar{A}} V^r \lar \bar{H} \otimes_H Y^r.
\end{equation}
Because of the minimality assumption on complexes $Y^\bu$ and $V^\bu$, there is precisely one representative in the homotopy class of $\theta$.
In other words, $\theta$ can be identified with the collection of its components $\bigl\{\theta^r \bigr\}_{r \in \ZZ}$, which are isomorphisms
in the category $\bar{H}-\mathsf{mod}$.

\smallskip
\noindent
Consider the following commutative diagram in the \emph{abelian} category
$\mathsf{Com}^\ast(A-\mathsf{mod})$:
\begin{equation}\label{E:reconstr}
\begin{array}{c}
\xymatrix
{ 0 \ar[r] & I Y^\bu  \ar[r] \ar@{=}[d] & X^\bu \ar[r]^p \ar[d]_\imath   & V^\bu
\ar[d]^-{\widetilde{\theta}} \ar[r] & 0 \\
0 \ar[r] & I Y^\bu  \ar[r] & Y^\bu
\ar[r]^{\pi} & \bar{Y}^\bu
\ar[r] & 0,
}
\end{array}
\end{equation}
where $\widetilde\theta$ is the morphism corresponding to $\theta$ under the adjunction
$$
\Hom_{\mathsf{Com}^\ast(\bar{H})}\bigl(\bar{H} \otimes_{\bar{A}} V^\bu, \bar{H} \otimes_H Y^\bu\bigr) \cong \Hom_{\mathsf{Com}^\ast(\bar{A})}\bigl(V^\bu, \bar{H} \otimes_H Y^\bu\bigr).
$$
Equivalently, we have the following  short exact sequence in $\mathsf{Com}^\ast(A-\mathsf{mod})$:
\begin{equation}\label{E:triangle1}
0 \lar X^\bu \lar Y^\bu \oplus V^\bu \xrightarrow{(-\pi \; \tilde\theta)} \bar{Y}^\bu \lar 0.
\end{equation}
In this way we get an assignment (but not a functor!): $$\Ob\bigl(\Tri^\ast(A)\bigr) \stackrel{\sfK}\lar \Ob\bigl(\mathsf{Com}^\ast(A-\mathsf{mod})\bigr),
\quad  (Y^\bu, V^\bu, \theta) \stackrel{\sfK}\mapsto X^\bu.$$

\noindent
\underline{Step 3}. We claim that the following statements are true.
\begin{itemize}
\item $X^\bu$ is a minimal bounded complex of projective $A$--modules.
\item The assignment $\sfK$ maps isomorphic objects of the category  $\Tri^\ast(A)$ into isomorphic objects of $\mathsf{Com}^\ast(A-\mathsf{mod})$.
\item We have an isomorphism $\EE(X^\bu) \cong (Y^\bu, V^\bu, \theta)$, induced by the morphisms $p$ and $\imath$ through the  corresponding adjunctions.
\end{itemize}
Let $(Y_1^\bu, V_1^\bu, \theta_1) \xrightarrow{(g, h)}  (Y_2^\bu, V_2^\bu, \theta_2)$ be   a morphism in $\Tri^\ast(A)$. As above, we may assume that
$Y_i^\bu$ and $V_i^\bu$ are minimal for $i = 1, 2$. Therefore, $V_1^\bu \stackrel{h}\lar V_2^\bu$ has a unique representative in its homotopy class.
Assume now that the morphism $(g, h)$ is an isomorphism. Let $Y_1^\bu \stackrel{\widetilde{g}}\lar Y_2^\bu$ be any representative of $g$ in
$\mathsf{Com}^\ast(H-\mathsf{mod})$. Then both morphisms  $h$ and $\widetilde{g}$ are isomorphisms in the corresponding categories of complexes.
Let $\bar{Y}_i^\bu := \bar{H} \otimes_{\bar{A}} Y_i^\bu$ for $i = 1,2$ and $\bar{Y}_1^\bu \stackrel{\bar{g}}\lar \bar{Y}_2^\bu$ be the morphism of complexes, induced
by $\widetilde{g}$. Because of the minimality, $\bar{g}$ does not depend on a particular choice of the lift $\widetilde{g}$.
By the definition of morphisms in the category $\Tri^\ast(A)$,  the following
diagram
\begin{equation}\label{E:CommDiagHot}
\begin{array}{c}
\xymatrix
{
V_1^\bu \ar[rr]^-{\widetilde{\theta}_1} \ar[d]_-{h} & &
 \bar{Y}_1^\bu \ar[d]^-{\bar{g}}\\
V_2^\bu \ar[rr]^-{\widetilde{\theta}_2} & &
\bar{Y}_2^\bu
}
\end{array}
\end{equation}
is commutative in the homotopy category $\mathsf{Hot}^\ast\bigl(\bar{A}-\mathsf{pro}\bigr)$. Since all complexes in diagram (\ref{E:CommDiagHot})  have zero differentials, we actually have a commutative diagram in the category $\mathsf{Com}^\ast(\bar{A}-\mathsf{mod})$, hence
in the category $\mathsf{Com}^\ast(A-\mathsf{mod})$. Using the universal property of  kernel applied to (\ref{E:triangle1}), we conclude  that the pair $(\widetilde{g}, h)$ induces an isomorphism  $X_1^\bu \stackrel{\widetilde{f}}\lar X_2^\bu$ in the category $\mathsf{Com}^\ast(A-\mathsf{mod})$. The claim follows
now from Step 1, applied to the components of the complexes in question.

\smallskip
\noindent
\underline{Step 4}. We have seen  that $(\sfK \circ \EE) (X^\bu) \cong X^\bu$ for any $X^\bu \in \Ob\bigl(\mathsf{Hot}^\ast(A-\mathsf{pro})\bigr)$ as well as
$(\EE \circ \sfK)  (Y^\bu, V^\bu, \theta) \cong (Y^\bu, V^\bu, \theta)$ for any $(Y^\bu, V^\bu, \theta) \in \Ob\bigl(\Tri^\ast(A)\bigr)$. This implies that
the functor $\EE$ is essentially surjective and reflects indecomposability and isomorphism classes of objects.
To show that $\EE$ is full, consider a morphism $(Y_1^\bu, V_1^\bu, \theta_1) \xrightarrow{(g, h)}  (Y_2^\bu, V_2^\bu, \theta_2)$  as in Step 3.
Let $Y_1^\bu \stackrel{\widetilde{g}}\lar  Y_2^\bu$ be a  representative of the homotopy class $g$. If $X_i^\bu := \sfK(Y_i^\bu, V_i^\bu, \theta_i)$ for $i = 1, 2$, then the universal property of kernel yields a uniquely  morphism of complexes $X_1^\bu \stackrel{\widetilde{f}}\lar X_2^\bu$, induced by the pair
$(\widetilde{g}, h)$. If $f$ denotes  the homotopy class of $\widetilde{f}$ then we get a commutative diagram
$$
\xymatrix{
\EE(X_1^\bu) \ar[r]^{\EE(f)} \ar[d]_{\cong} & \EE(X_2^\bu) \ar[d]^\cong\\
(Y_1^\bu, V_1^\bu, \theta_1) \ar[r]^{(g, h)} & (Y_2^\bu, V_2^\bu, \theta_2)
}
$$
in the category $\Tri^\ast(A)$, implying the statement.
\end{proof}

\subsection{Category of triples and the corresponding matrix problem}\label{SS:TriplesandMP} Now, let us return to the setting of
 Section \ref{S:Generalities}.
Let $(\vec{\sm}, \simeq)$ be   a datum as in Definition \ref{D:skewgentle}, $A = A(\vec{\sm}, \simeq)$ be the corresponding skew--gentle algebra,
$H = H(\vec{\sm}, \simeq)$ its normalization, $I = \mathsf{rad}(A) = \mathsf{rad}(H)$, $\bar{A} = A/I$ and $\bar{H} = H/I$.
Let $\Omega = \Omega(\vec{\sm})$ be the set given by (\ref{E:SetOmega}), whereas the sets $\overline{\Omega}$  and $\widetilde{\Omega}$ are the ones
from Definition \ref{D:Sets}.
In order to make the category $\Tri^\ast(A)$ more accessible, we begin  with the following observations.

\smallskip
\noindent
1.~We have a decomposition
\begin{equation}\label{E:barH}
\bar{H} \cong \prod\limits_{(i, j) \in \Omega} \bar{H}_{(i, j)}, \quad \mbox{\rm where} \quad \bar{H}_{(i, j)} \cong
\left\{
\begin{array}{cl}
\mathsf{Mat}_2(\kk) & \mbox{\rm if} \; (i, j) \simeq (i, j) \\
\kk & \mbox{\rm if} \; (i, j) \not\simeq (i, j).
\end{array}
\right.
\end{equation}
In particular, each  algebra $\bar{H}_{(i, j)}$ is equal or Morita equivalent to the field $\kk$.

\smallskip
\noindent
Primitive idempotents of the algebra $H$ are parameterized by the elements of the set $\overline{\Omega}$. Namely, for each element $(i, j) \in \Omega$ such that
$(i, j) \simeq (i, j)$, we have two \emph{conjugate} idempotents $e_{((i, j), \pm)} \in H$ in the $j$--th diagonal block of the algebra $H_i$, corresponding  to the matrices
$
\left(
\begin{array}{cc}
1 & 0 \\
0 & 0
\end{array}
\right)
$
and $
\left(
\begin{array}{cc}
0 & 0 \\
0 & 1
\end{array}
\right)
$
with respect to the decomposition (\ref{E:2times2block}). The corresponding indecomposable projective $H$--modules
$H \cdot e_{((i, j), +)}$ and $H \cdot e_{((i, j), -)}$ are isomorphic. Summing up, the isomorphism classes of indecomposable projective $H$--modules
are parameterized by the elements of the set $\Omega$. So, for any $(i, j) \in \Omega$, we denote by
\begin{itemize} \item $Q_{(i, j)}$ the corresponding indecomposable projective $H$--module
\item $\bar{Q}_{(i, j)} := \bar{H} \otimes_H Q_{(i, j)}$.
\end{itemize}
Note that
$
\bar{H}_{(i, j)} \cong
\left\{
\begin{array}{cl}
 \bar{Q}_{(i, j)} \oplus \bar{Q}_{(i, j)} & \mbox{\rm if} \; (i, j) \simeq (i, j) \\
\bar{Q}_{(i, j)} & \mbox{\rm if} \; (i, j) \not\simeq (i, j).
\end{array}
\right.
$
\smallskip
\noindent
In what follows, it will be convenient to put $Q_{(i, 0)} = Q_{(i, m_i+1)} = 0$ for any $1 \le i \le t$.

\smallskip
\noindent
2.~Note that the  algebra $A$ is basic. The primitive idempotents of $A$ as well as the corresponding indecomposable projective $A$--modules are parameterized by the elements of the set $\widetilde\Omega$. So, for any $\gamma \in \widetilde\Omega$, let $e_\gamma \in A$ be the corresponding primitive idempotent and  $P_\gamma := A \cdot e_\gamma$ the corresponding indecomposable projective $A$--module. Obviuosly, we have a decomposition
\begin{equation}
\bar{A} \cong \prod\limits_{\gamma \in \widetilde\Omega} \bar{A}_{\gamma}, \quad \mbox{\rm where} \quad \bar{A}_\gamma \cong \kk
\quad \mbox{\rm for any} \quad \gamma \in \widetilde\Omega.
\end{equation}
Moreover, there exists  the following natural isomorphisms of $\bar{H}$--modules:
$$
\bar{H}\otimes_{\bar{A}} \bar{A}_\gamma \cong
\left\{
\begin{array}{cl}
\bar{Q}_{(i, j)} \oplus \bar{Q}_{(k, l)}  & \mbox{\rm if} \; \gamma = \{(i, j), (k,l)\} \; \mbox{\rm is of the first type} \\
\bar{Q}_{(i, j)} & \mbox{\rm if} \; \gamma =
\left\{
\begin{array}{cl}
((i, j), \pm) & \mbox{\rm is of the second type} \\
(i, j) & \mbox{\rm is of the third type}.
\end{array}
\right.
\end{array}
\right.
$$

\smallskip
\noindent
3.~For any $1 \le i \le t$ and $1 \le b < a \le m_i+1$, consider the indecomposable complex of projective $H$--modules
\begin{equation}\label{E:ComplexW}
W^\bu_{(i, (a, b))}:= \bigl(\dots \lar 0 \lar Q_{(i, a)} \xrightarrow{\kappa_{(i, (a, b))}}  Q_{(i, b)} \lar 0 \lar \dots\bigr),
\end{equation}
where its unique cohomology $W_{(i, (a, b))} := Q_{(i, b)}/Q_{(i, a)}$  is located in the zero degree. Within our convention, we have:  $Q_{(i, m_i+1)} = 0$, hence $W^\bu_{(i, (m_i+1, b))} = Q_{(i, b)}[0]$.
Note that for any $1 \le i \le t$ and $1 \le b < a \le m_i$, there exists a \emph{unique} (up to a scalar) non--zero morphism $\kappa_{(i, (a, b))}$ from $Q_{(i, a)}$ to
$Q_{(i, b)}$. Therefore, the notation for the differential in (\ref{E:ComplexW}) will be  omitted in what follows. Note that
\begin{equation}\label{E:splittingcompl}
\bar{H} \otimes_H W^\bu_{(i, (a, b))}:= \bigl(\dots \lar 0 \lar \bar{Q}_{(i, a)} \stackrel{0}\lar \bar{Q}_{(i, b)} \lar 0 \lar \dots\bigr).
\end{equation}

\smallskip
\noindent
4.~Let $Y^\bu  = \bigl(\dots \lar Y^{p-1} \lar Y^p \lar Y^{p+1} \lar \dots \bigr) \in \Ob\bigl(\mathsf{Hot}^\ast(H-\mathsf{pro})\bigr)$ be a minimal complex of projective $H$--modules. Since the algebra $H$ is hereditary,   we have a splitting
\begin{equation}\label{E:ComplexSumOfW}
Y^\bu \cong \bigoplus\limits_{r \in \ZZ} \left(
\bigoplus_{\substack{1 \le i \le t \\ 1 \le b < a \le m_i+1}} \Bigl(W_{(i, (a, b))}^\bu[-r]\Bigr)^{\oplus l(i, (a, b), r)}\right)
\end{equation}
for some uniquely determined multiplicities $l(i, (a, b), r) \in \NN_0$. At this place, we use the classification of indecomposable representations of a quiver of type $\bul \lar \bul \lar  \dots \lar \bul \lar \bul$ as well as a result of Dold \cite[Satz 4.7]{Dold} on the structure of the derived category of a hereditary algebra.
 Of course, for $* = b$, we additionally have:
$l(i, (a, b), r) = 0$ for all $|r| \gg 0$, whereas for $* = \pm$, we have $l(i, (a, b), r) = 0$  for $r \ll 0$ or $r \gg 0$, respectively.

\smallskip
\noindent
Since the complex $Y^\bu$ is minimal, we get:
$$
\bar{Y}^\bu := \bar{H} \otimes_H Y^\bu \cong \bigl(\dots \stackrel{0}\lar \bar{Y}^{r-1} \stackrel{0}\lar \bar{Y}^r \stackrel{0}\lar \bar{Y}^{r+1} \stackrel{0}\lar \dots \bigr).
$$
Next, observe that
for any $r \in \ZZ$, we have an isomorphism of $\bar{H}$--modules
\begin{equation}\label{E:moduleYr}
\bar{Y}^r \cong \bigoplus\limits_{(i, j) \in \Omega} \bar{Q}_{(i, j)}^{m(r, (i, j))},
\end{equation}
induced by the decomposition (\ref{E:ComplexSumOfW}) and isomorphisms (\ref{E:splittingcompl}).

\smallskip
\noindent
5.~Let $V^\bu \in \Ob\bigl(\mathsf{Hot}^\ast(\bar{A}-\mathsf{pro})\bigr)$ be as in (\ref{E:complexV}). For any $r \in \ZZ$, we have decompositions
\begin{equation}\label{E:moduleVr}
V^r \cong \bigoplus\limits_{\gamma \in \widetilde\Omega} \bar{A}_\gamma^{\oplus n(r, \gamma)}
\end{equation}
for some uniquely determined multiplicities $n(r, \gamma) \in \NN_0$.

\smallskip
\noindent
6.~Let $(Y^\bu, V^\bu, \theta)$ be an object of the category of triples $\Tri^\ast(A)$. As in the proof of Theorem \ref{T:mainconstr}, we may assume
both complexes $Y^\bu$ and $V^\bu$ to be minimal. Moreover, we may fix direct sum decompositions (\ref{E:ComplexSumOfW}) and (\ref{E:moduleVr}).
Recall that the  gluing map $\bar{H} \otimes_{\bar{A}} V^\bu \stackrel{\theta}\lar
\bar{H} \otimes_{H} Y^\bu$ can be identified with  the  collection of its components  $\left\{\bar{H} \otimes_{\bar{A}} V^r \stackrel{\theta^r}\lar \bar{Y}^r\right\}_{r \in \ZZ}$. Since $\theta$ is an isomorphism, all morphisms of $\bar{H}$--modules $\theta^r$ are isomorphisms, too. According to the decomposition
(\ref{E:barH}), each isomorphism $\theta^r$ is itself determined  by a collection of its components $\left\{\theta^r_{(i, j)}\right\}_{(i, j) \in \Omega}$, which are isomorphisms of $\bar{H}_{(i, j)}$--modules.
Since each algebra $\bar{H}_{(i, j)}$ is Morita equivalent to $\kk$,  we can represent each morphism $\theta^r_{(i, j)}$ by an invertible  matrix $\Theta_{(i, j)}^r \in \mathsf{Mat}_{m(r, (i, j))}(\kk)$, where $m(r, (i, j))$ is the multiplicity defined by the decomposition (\ref{E:moduleYr}). Moreover, in virtue of  the  decompositions (\ref{E:ComplexSumOfW}),  the matrix
$\Theta_{(i, j)}^r$ has  an induced division into \emph{horizontal} stripes. As we shall see below, in the case $(i, j) \simeq (i, j)$, the matrices
$\Theta_{(i, j)}^r$ have additional natural divisions into two \emph{vertical} stripes.

\smallskip
\noindent
7.~Let $(i, j) \in \Omega$ be such that $(i, j) \simeq (i, j)$. According to (\ref{E:moduleVr}), the $\bar{A}$--module $V^r$ has  two direct summands $\bar{A}_{((i, j), \pm)}^{n(r, ((i, j), \pm))}$ and  we get  an induced isomorphism:
$$
\bar{H} \otimes_{\bar{A}} \left(\bar{A}_{((i, j), +)}^{\oplus n(r, ((i, j), +))} \oplus \bar{A}_{((i, j), -)}^{\oplus n(r, ((i, j), -))}\right) \cong
Q_{(i, j)}^{\oplus m(r, (i, j))}.
$$
Note that since  $\theta^r_{(i, j)}$ is an isomorphism, we automatically have: $$n\bigl(r, ((i, j), +))\bigr) + n\bigl(r, ((i, j), -))\bigr) = m(r, (i, j)).$$
Therefore, the matrix $\Theta^r_{(i, j)}$ is divided into two vertical stripes, labelled by the symbols $+$ and $-$, and having $n\bigl(r, ((i, j), +))\bigr)$ respectively  $n\bigl(r, ((i, j), -))\bigr)$ columns.

\smallskip
\noindent
Moreover, each matrix $\Theta^r_{(i, j)}$ has the following  division into horizontal stripes. First note that the only
complexes $W^\bu_{(i, (a, b))}[r]$,  which contribute to the $\bar{H}_{(i, j)}$--module $\bar{Y}^r_{(i, j)}$, are contained in to following diagram of morphisms
in $\mathsf{Hot}^\ast(H-\mathsf{mod})$:
\begin{equation}\label{E:hierarchyofmorphisms}
\begin{array}{c}
\xymatrix{
& Q_{(i, j)} \ar[r] \ar[d] & Q_{(i, j-1)} \ar[d] \\
& \dots \ar[d] & \dots \ar[d]\\
& Q_{(i, j)} \ar[r] \ar[d]& Q_{(i, 1)} \\
& Q_{(i, j)} \ar[d]& \\
Q_{(i, m_i)} \ar[r]\ar[d] & Q_{(i, j)} \ar[d] & \\
\dots \ar[d] & \dots \ar[d] & \\
Q_{(i, j+1)} \ar[r] & Q_{(i, j)} &
}
\end{array}
\end{equation}
where the component $Q_{(i, j)}$ of each complex is located in the  $r$--th degree.  Moreover, the displayed morphisms are the only ones, whose $(r, (i,j))$--th
component is non--zero after applying the functor $\bar{H} \otimes_H \,-\,$.

\smallskip
\noindent
 Next, for any $r \in \ZZ$, $1 \le i \le t$  and $1 \le b < a \le m_i$ we introduce a \emph{pair} of symbols $q_{((i, b), r)}^{((i, a), r-1)}$ and $q^{((i, b), r)}_{((i, a), r-1)}$, which encode the complex $W^\bu_{(i, (a, b))}[-r]$. Additionally, we introduce a \emph{single} symbol
$q_{((i, b), r)}^{((i, m_i+1), r-1)}$, which encodes the stalk complex $Q_{(i, b)}[-r] = W^\bu_{(i, (m_i+1, b))}[r]$. Then for  any $r \in \ZZ$ and $(i, j) \in \Omega$, we have a totally ordered set $\dF_{((i, j), r)}$, whose elements are
\begin{equation}\label{E:FirstOrder}
q_{((i, j), r)}^{((i, 1), r+1)} < \dots < q_{((i, j), r)}^{((i, j-1), r+1)} < q_{((i, j), r)}^{((i, m_i+1), r-1)} <
q_{((i, j), r)}^{((i, m_i), r-1)} < \dots < q_{((i, j), r)}^{((i, j-1), r-1)}.
\end{equation}
Then the  horizontal stripes  of the matrices $\Theta^r_{(i, j)}$ are labelled by the elements of $\dF_{(r, (i, j))}$.
In the notations of  the decomposition (\ref{E:ComplexSumOfW}), we see that  the horizontal stripe of $\Theta^r_{(i, j)}$  labelled by the symbol $q_{((i, j), r)}^{((i, k), r\pm 1)}$ has
\begin{itemize}
\item $l\bigl((i, (k, j)), r\bigr)$ rows, if $1 \le j < k \le m_{i+1}$,
\item $l\bigl((i, (j, k)), r-1\bigr)$ rows, if $1 \le k < j \le m_{i}$.
\end{itemize}

\subsection{Category $\Tri^\ast(A)$ and the corresponding matrix problem}\label{SS:TriplesMP} Let us make a brief overview  of the results obtained in the previous subsections.

\smallskip
\noindent
1.~According to Theorem \ref{T:mainconstr}, we have a functor $\mathsf{Hot}^\ast(A) \stackrel{\EE}\lar \Tri^\ast(A)$, reflecting indecomposability and isomorphism classes of objects. Therefore, having a classification of the indecomposable objects of the category of triples, we can deduce from it a description
of the indecomposable objects of $\mathsf{Hot}^\ast(A)$.

\smallskip
\noindent
2.~Let $(Y^\bu, V^\bu, \theta)$ be an object of $\Tri^\ast(A)$. Fixing decompositions (\ref{E:ComplexSumOfW}) and (\ref{E:moduleVr}), we can attach
to $(Y^\bu, V^\bu, \theta)$ a collection of square and invertible matrices $\bigl\{\Theta_{(i, j)}^r\bigr\}_{((i, j), r) \in \Omega \times \ZZ}$.
Each matrix $\Theta_{(i, j)}^r$ is divided into horizontal stripes, labelled by the elements of the set $\dF_{((i, j), r)}$. Moreover,
if $(i, j) \in \Omega$ is such that $(i, j) \simeq (i, j)$ then $\Theta_{(i, j)}^r$ is divided into two vertical stripes, labelled by the symbols $\pm$.
Finally, note that
\begin{itemize}
\item Whenever  $(i, j) \ne (k, l) \in \Omega$ are such that $(i, j) \simeq (k, l)$ then the matrices $\Theta_{(i, j)}^r$ and $\Theta_{(k, l)}^r$ have the same number of columns (hence, the same size) for any $r \in \ZZ$.
   \item For any $r \in \ZZ$, $1 \le i \le t$  and $1 \le b < a \le m_i$ the horizontal stripe of the matrix $\Theta_{(i, b)}^r$ labelled  by $q_{((i, b), r)}^{((i, a), r-1)}$  has the same number of rows as the horizontal stripe of the matrix $\Theta_{(i, a)}^{r-1}$ labelled by  $q^{((i, b), r)}_{((i, a), r-1)}$.
\end{itemize}
The described block structure put on the collection  of matrices $\bigl\{\Theta_{(i, j)}^r\bigr\}_{((i, j), r) \in \Omega \times \ZZ}$ will be called \emph{decoration}.

\smallskip
\noindent
3.~Conversely, assume we are given a collection of invertible  matrices $\bigl\{\Theta_{(i, j)}^r\bigr\}_{((i, j), r) \in \Omega \times \ZZ}$ equipped with a decoration, i.e.~a row (and possibly column) division as in the previous paragraph. Then the sizes of the corresponding blocks allow to reconstruct the multiplicities
(\ref{E:ComplexSumOfW}) and (\ref{E:moduleVr}). In this way, we get an object  $(Y^\bu, V^\bu, \theta)$ of the category  $\Tri^\ast(A)$.

\smallskip
\noindent
4.~The description of the isomorphism classes of objects of the category $\Tri^\ast(A)$ leads to a certain problem of linear algebra (matrix problem).
Assume we have two triples $(Y^\bu, V^\bu, \theta)$ and $(\widehat{Y}^\bu, \widehat{V}^\bu, \widehat\theta)$.  If they are isomorphic, then we have:
$Y^\bu \cong \widehat{Y}^\bu$ in $\mathsf{Hot}^\ast(H-\mathsf{pro})$ and $V^\bu \cong \widehat{V}^\bu$ in $\mathsf{Hot}^\ast(\bar{A}-\mathsf{pro})$.
So, we may assume that $Y^\bu = \widehat{Y}^\bu$ and $V^\bu =  \widehat{V}^\bu$ and moreover, the decompositions (\ref{E:ComplexSumOfW}) and (\ref{E:moduleVr}) are fixed. By the definition of morphisms in the category $\Tri^\ast(A)$, we have $(Y^\bu, V^\bu, \theta) \cong (Y^\bu, V^\bu, \widehat\theta)$ if and only if there exist $g \in \Aut_{\mathsf{Hot}^\ast(H-\mathsf{pro})}\bigl(Y^\bu\bigr)$ and $h \in \Aut_{\mathsf{Hot}^\ast(\bar{A}-\mathsf{pro})}\bigl(V^\bu\bigr)$
such that
$
\bar{g} \theta = \widehat{\theta} \bar{h}
$
where $\bar{g} = \mathsf{id} \otimes g$ and $\bar{h} = \mathsf{id} \otimes g$ are the induced maps (see diagram (\ref{E:morphismstriples})).
It is clear, that $\bar{g}$ and $\bar{h}$ are given by a collection of its components $\bigl\{\bar{g}_{(i, j)}^r\bigr\}_{((i,j), r) \in \Omega \times \ZZ}$ and
$\bigl\{\bar{h}_\gamma^r\bigr\}_{(\gamma, r)  \in \widetilde\Omega \times \ZZ}$, which are automorphisms of the $\bar{H}_{(i, j)}$--module $\bar{Y}^r_{(i, j)}$ and
$\bar{A}_\gamma$--module $V^r_{\gamma}$, respectively.

\smallskip
\noindent
5.~Let $\bigl\{\Theta_{(i, j)}^r\bigr\}_{((i, j), r) \in \Omega \times \ZZ}$ be the collection of decorated matrices corresponding to the triple $(Y^\bu, V^\bu, \theta)$ and
$\bigl\{\widehat\Theta_{(i, j)}^r\bigr\}_{((i, j), r) \in \Omega \times \ZZ}$ be the collection  corresponding to $(Y^\bu, V^\bu, \widehat{\theta})$.
Then the relation $\bar{g} \theta = \widehat{\theta} \bar{h}$ can be rewritten as a system of equalities
\begin{equation}\label{E:transfrule}
\Phi^r_{(i, j)} \Theta^r_{(i, j)} = \widehat{\Theta}^r_{(i, j)} \Psi^r_{(i, j)} \quad \mbox{\rm for each}\; ((i, j), r) \in \Omega \times \ZZ,
\end{equation}
where $\Phi^r_{(i, j)}$ and $\Psi^r_{(i, j)}$ are invertible matrices satisfying the following conditions.
\begin{enumerate}
\item If $(i, j) \ne (k, l) \in \Omega$ are such that $(i, j) \simeq (k, l)$ then $\Psi^r_{(i, j)} = \Psi^r_{(k, l)}$ for any $r \in \ZZ$.
\item If $(i, j) \simeq (i, j)$ then $$
\Psi^r_{(i, j)} =
\left(
\begin{array}{cc}
\Psi_{((i, j), +)}^r & 0 \\
0 & \Psi_{((i, j), -)}^r
\end{array}
\right),
$$
where each block  $\Psi_{((i, j), \pm)}^r$ is a square and invertible matrix of the size $n\bigl(r, ((i, j), \pm))$.
\item Each matrix $\Phi_{(i, j)}^r$ is divided into blocks $\left(\Phi_{(i, j)}^r\right)_{uv}$, where $u, v \in \dF_{((i, j), r)}$. Moreover, this blocks
satisfy the following additional constraints.
\begin{enumerate}
\item Let $u = q_{((i, j), r)}^{((i, k), r\pm 1)}$. Then for any $v \in \dF_{((i, j), r)}$, the block $\left(\Phi_{(i, j)}^r\right)_{uv}$ has
\begin{itemize}
\item $l\bigl((i, (k, j)), r\bigr)$ rows, if $1 \le j < k \le m_{i+1}$,
\item $l\bigl((i, (j, k)), r-1\bigr)$ rows, if $1 \le k < j \le m_{i}$.
\end{itemize}
The same rule applies for the number of columns of the block $\left(\Phi_{(i, j)}^r\right)_{uv}$.
\item We have: $\left(\Phi_{(i, j)}^r\right)_{uv} = 0$ if $u < v$ in $\dF_{((i, j), r)}$.
\item For any $r \in \ZZ$, $1 \le i \le t$ and $1 \le a < b \le m_i$, let $u:= q_{((i, b), r)}^{((i, a), r-1)} \in \dF_{((i, b), r)}$ and
 $v := q^{((i, b), r)}_{((i, a), r-1)} \in \dF_{((i, a), r-1)}$. Then we have:
 $$
 \left(\Phi_{(i, b)}^r\right)_{uu} = \left(\Phi_{(i, a)}^{r-1}\right)_{vv}
 $$
\end{enumerate}
\end{enumerate}
Conversely, any tuples   of matrices $\left\{\Phi^r_{(i, j)}\right\}_{((i, j), r) \in \Omega \times \ZZ}$ and
$\left\{\Psi^r_{(i, j)}\right\}_{((i, j), r) \in \Omega \times \ZZ}$ satisfying the above constraints arise as the images of appropriate automorphisms of $Y^\bu$ and $V^\bu$.

\smallskip
\noindent
\textbf{Summary}. The transformation rule (\ref{E:transfrule}) can be rephrased as follows: two collections of decorated matrices
$\bigl\{\widehat\Theta_{(i, j)}^r\bigr\}_{((i, j), r) \in \Omega \times \ZZ}$ and  $\bigl\{\widehat\Theta_{(i, j)}^r\bigr\}_{((i, j), r) \in \Omega \times \ZZ}$
are equivalent, if the second collection can be obtained from the first one using the following transformations.

\smallskip
\noindent
1.~\textsl{Transformations of columns}. Let $r \in \ZZ$ be arbitrary.
\begin{enumerate}
\item If $(i, j) \ne (k, l) \in \Omega$ are such that $(i, j) \simeq (k, l)$  then  we may perform arbitrary invertible \emph{simultaneous} transformations
with the columns of  $\Theta_{(i, j)}^r$ and $\Theta_{(k, l)}^r$.
\item If $(i, j) \simeq (i, j)$ then for any $r \in \ZZ$, we may perform arbitrary invertible transformations \emph{within} each block $\pm$
of the matrix $\Theta_{(i, j)}^r$. However, we can not add columns from the block labelled by $+$ to columns from the block labelled by $-$.
\item If $(i, j) \not\simeq  (k, l)$ for any $(k, l) \in \Omega$ then we may perform arbitrary invertible transformations of columns $\Theta_{(i, j)}^r$.
\end{enumerate}

\smallskip
\noindent
2.~\textsl{Transformations of rows}. Let  $((i, j), r) \in \Omega \times \ZZ$.
\begin{enumerate}
\item We can perform arbitrary invertible transformations of rows within the horizontal block of the matrix $\Theta_{(i, j)}^r$ labelled by $q_{((i, j), r)}^{((i, m_i+1), r-1)}$.
\item For any $u < v \in \dF_{((i, j), r)}$, we may add a scalar multiple of any row of weight $u$ to any row of weight $v$.
\item For any element $(i, \bar{j}) \in \Omega$ with $\bar{j} \ne j$, we may perform arbitrary simultaneous transformations of horizontal blocks of the matrices $\Theta_{(i, j)}^r$ and $\Theta_{(i, \bar{j})}^{r\pm 1}$, labelled by the symbols  $q_{((i, j), r)}^{((i, \bar{j}), r\pm 1)}$ and
 $q^{((i, j), r)}_{((i, \bar{j}), r\pm 1)}$, respectively.
\end{enumerate}

\smallskip
\noindent
\textbf{Conclusion}. The classification of objects of the category $\Tri^\ast(A)$ (and as a conclusion, of the homotopy category $\mathsf{Hot}^\ast(A-\mathsf{pro})$) reduces to a certain \emph{matrix problem}. As we shall see below, this
 matrix problem turns out to be \emph{tame}. Moreover,
there are explicit combinatorial rules to write canonical forms for the   decorated matrices describing indecomposable objects of this matrix problem.

\section{Bimodule problems and representations of bunches of (semi--)chains}
The notion   of a bimodule problem has been introduced by the second--named author  in \cite{DrozdLOMI} in an attempt  to give  a  formal definition  of a matrix problem. See also \cite{dr,BurbanDrozdMemoirs} for further examples and elaborations. In this section, we follow very closely the exposition of \cite[Chapter 8]{BurbanDrozdMemoirs}.

Let $\kk$ be a field, $\catC$ be an $\kk$--linear category and $\catB$ be an $\catA$--bimodule. Recall that this   means that
for  any pair of objects $A, B$ of $\catC$, we have  a $\kk$--vector space $\catB(A, B)$ and for any further pair of objects
$A', B'$ of $\catC$ there are left and right multiplication maps
$$
\catA(B, B') \times \catB(A, B) \times \catA(A', A) \lar \catB(A', B'),
$$
which are $\kk$--multilinear and associative.
\begin{definition}
The   \emph{bimodule category} $\el(\catA, \catB)$  (or  \emph{category of elements of the $\catA$--bimodule $\catB$}) is defined as follows. Its objects are pairs $(A, \vartheta)$, where
$A$ is an object of the category $\catA$ and $\vartheta \in \catB(A, A)$, whereas the morphisms  are defined as follows:
$$\el(\catA, \catB)\bigl((A, \vartheta), (A', \vartheta')\bigr) = \bigl\{f \in \catA(A, A') \, | \, f \vartheta = \vartheta' f \bigr\}.$$
The composition of morphisms in $\el(\catA, \catB)$  is the same as  in $\catA$.
\end{definition}

\begin{remark}
 The category $\el(\catA, \catB)$  is additive, $\kk$--linear  and idempotent complete provided the category $\catA$ is additive and idempotent complete. In applications, one usually begins
with such a category $\catA$, for which the endomorphism algebra of any of its objects is local
(obviously, in this case, $\catA$ can not be additive). Then one takes the \emph{additive closure}
$\catA^{\omega}$ of $\catA$ and extends $\catB$ to an $\catA^{\omega}$--bimodule $\catB^{\omega}$ by additivity. Abusing the notation,
we write  $\el(\catA, \catB)$, having actually the category $\el(\catA^{\omega}, \catB^{\omega})$ in mind.
\end{remark}

\begin{example}\label{Ex:bunchesofChains} Consider the following   category $\catA$ and $\catA$--bimodule $\catB$.
\begin{itemize}
\item $\Ob(\catA) = \{a, c, d\}$.
\item The non--zero morphism spaces of $\catA$ are the following ones:
\begin{itemize}
\item $\catA(a, a) = \kk 1_{a}$, $\catA(c, c) = \kk 1_{c}$ and $\catA(d, d) = \kk 1_{d}$, whereas
\item $\catA(c, d) = \langle \nu_1, \nu_2\rangle_{\kk} \cong \kk^2$.
\end{itemize}
\item The bimodule $\catB$ is defined as follows:
\begin{itemize}
\item $\catB(a, c) = \langle \phi_1, \phi_2\rangle_{\kk} \cong \kk^2 \cong \catB(a, d) = \langle \psi_1, \psi_2\rangle_{\kk}$.
    \item For $(x, y) \notin \bigl\{(a, c), (a, d)\bigr\}$ we have: $\catB(x, y) = 0$.
\item The action of $\catA$ on $\catB$ is given by the  rule: $\nu_i \circ \phi_j = \delta_{ij}\psi_i$ for $i, j = 1, 2$.
\end{itemize}
\end{itemize}
\noindent
The entire  data can be visualized by the following picture.
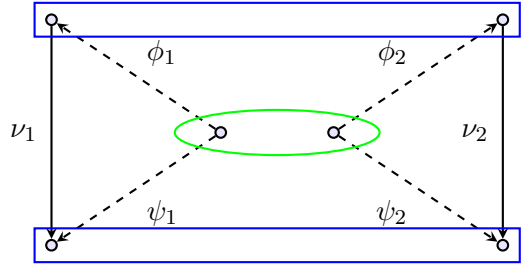
\begin{center}
\begin{tikzpicture}[scale=0.75,
    thick,
    dot/.style={fill=blue!10,circle,draw, inner sep=1pt, minimum size=4pt}]
 \draw (0,2) node[dot] (a1){} (0,-2)node[dot](b1){} (3,0)node[dot](c1){}
       (8,2) node[dot] (a2){} (8,-2)node[dot](b2){} (5,0)node[dot](c2){}
 ;

\draw[-stealth, thick, dashed](c1)-- node [above right]{$\phi_1$}(a1);
\draw[-stealth, thick, dashed](c1)--node [below right]{$\psi_1$}(b1);
 \draw[-stealth,thick] (a1) to node[left=0.5mm]{$\nu_1$} (b1);

\draw[-stealth, thick, dashed](c2)-- node [above left]{$\phi_2$}(a2);
\draw[-stealth, thick, dashed](c2)--node [below left]{$\psi_2$}(b2);
 \draw[-stealth,thick] (a2) to node[left=0.5mm]{$\nu_2$} (b2);

\draw[blue] ($(a1)-(3mm,3mm)$ ) rectangle  ($(a2)+(3mm,3mm)$ );
\draw[blue] ($(b1)-(3mm,3mm)$ ) rectangle  ($(b2)+(3mm,3mm)$ );

 \node[draw= green, ellipse,  fit=(c1) (c2)]{};

\end{tikzpicture}
\end{center}
The encircled points represent the three objects of the category $\catA$, the solid arrows denote non--trivial morphisms in $\catA$, whereas
the dotted arrows describe the generators of the bimodule $\catB$. This example is analogous to \cite[Example 8.3]{BurbanDrozdMemoirs}.
Hence, we skip  the computational details and state the matrix problem, underlying the description of the isomorphism classes of objects of the bimodule category $\el(\catA, \catB)$.
\begin{itemize}
\item  We have two matrices $\Theta_1$ and $\Theta_2$ of the same size over the field $\kk$, each of which consists of two horizontal blocks  $\Phi_1, \Psi_1$, respectively  $\Phi_2, \Psi_2$. These blocks  correspond to the generators $\phi_1, \phi_2, \psi_1$ and $\psi_2$ of the bimodule $\catB$.
   The blocks $\Phi_1$ and $\Phi_2$ (respectively, $\Psi_1$ and $\Psi_2$) have the same size.
\item
The admissible  transformations of  columns and rows of  the matrices $\Theta_1$ and $\Theta_2$,
are compositions of the following  ones.
\begin{itemize}
\item
We can perform  any \emph{simultaneous}  elementary transformations
\begin{itemize}
\item   of columns of  $\Theta_1$ and $\Theta_2$.
\item   of rows  of  $\Phi_1$ and  $\Phi_2$ (respectively, $\Psi_1$ and $\Psi_2$).
\end{itemize}
\item    For $i = 1, 2$, we may add   an arbitrary multiple of any row of the block $\Phi_i$ to any row
of the block $\Psi_i$ for $i = 1, 2$. These transformations of rows of the matrices $\Theta_1$ and $\Theta_2$ can be performed \emph{independently}.
\end{itemize}
\end{itemize}
The obtained matrix problem can be visualized by the following picture.
   
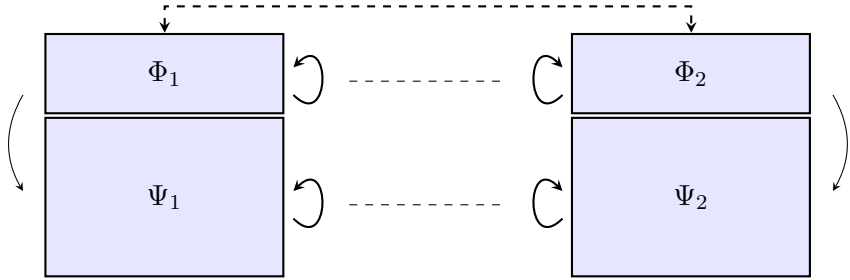
\begin{center}
 \begin{tikzpicture}
  \matrix[nodes={draw, thick, fill=blue!10},
    row sep=1pt, ]  {
  \node[ minimum height=30pt, minimum width=90pt](h1){$\Phi_1$}; &&\\
  \node[ minimum height=60pt, minimum width=90pt](s1){$\Psi_1$};&& \\
  }  ;

  \matrix[nodes={draw,thick, fill=blue!10},
    row sep=1pt, ]  at (7,0) {
  \node[ minimum height=30pt, minimum width=90pt](h2){$\Phi_2$}; &&\\
  \node[ minimum height=60pt, minimum width=90pt](s2){$\Psi_2$};&& \\
  }
  ;

\node (hr1) [base right=-0.25 cm of h1, inner sep=0pt, minimum size=10pt ]{};
\node (sr1) [base right=-0.25 cm of s1, inner sep=0pt, minimum size=10pt ]{};

\node (hl2) [base left=-0.25 of h2,inner sep=0pt, minimum size=10pt ]{};
\node (sl2) [base left=-0.25 cm of s2,inner sep=0pt, minimum size=10pt  ]{};

\node (hl1) [base left=0 cm of h1, inner sep=0pt, minimum size=10pt ]{};
\node (sl1) [base left=0 cm of s1, inner sep=0pt, minimum size=10pt ]{};

\node (hr2) [base right=0 cm of h2,inner sep=0pt, minimum size=10pt  ]{};
\node (sr2) [base right=0 cm of s2,inner sep=0pt, minimum size=10pt  ]{};

\tikzset{every loop/.style={min distance=5mm,in=135,out=225,looseness=5}}
\draw[-stealth,thick] (hl2)to[loop ] node[minimum size=20pt,inner sep=10pt](hh2){}   ();
\draw[-stealth,thick] (sl2)to[loop ] node[minimum size=20pt,inner sep=10pt](ss2){}   ();

\tikzset{every loop/.style={min distance=5mm,in=45,out=315,looseness=5}}
\draw[-stealth,thick] (hr1)to[loop ] node(hh1)[minimum size=20pt,inner sep=10pt]{}   ();
\draw[-stealth,thick] (sr1)to[loop ] node(ss1)[minimum size=20pt,inner sep=10pt]{}   ();

 \draw[dashed](hh1) to node[midway,above]{} (hh2);
 \draw[dashed](ss1) to node[midway,above]{} (ss2);

\draw[-stealth](hl1) to[bend right]node[midway,left]{} (sl1);
\draw[-stealth](hr2) to[bend left]node[midway,right]{} (sr2) ;

 \coordinate[above=10pt of h2](ha2){};

 \coordinate[above=10pt of  h1](ha1){};

 \draw[stealth-stealth, thick, dashed](h1) -- (ha1)-- (ha2)node[midway,above]{} -- (h2);

\end{tikzpicture}
\end{center}
This matrix problem arises in the course  of a classification of those vector bundles $\kF$ on the nodal Weirstra\ss{} cubic $E = \overline{V(y^2- x^3 - x^2})$, for which $\nu^*(\kF) \in \add\bigr(\kO_{\PP^1} \oplus \kO_{\PP^1}(1)\bigr)$, where $\PP^1 \stackrel{\nu}\lar E$ is the normalization map; see
\cite{DGVB, SurveyBBDG}.
\end{example}

\begin{example}\label{Ex:bunchesofSemiChains} Consider the following   category $\catA$ and bimodule $\catB$.
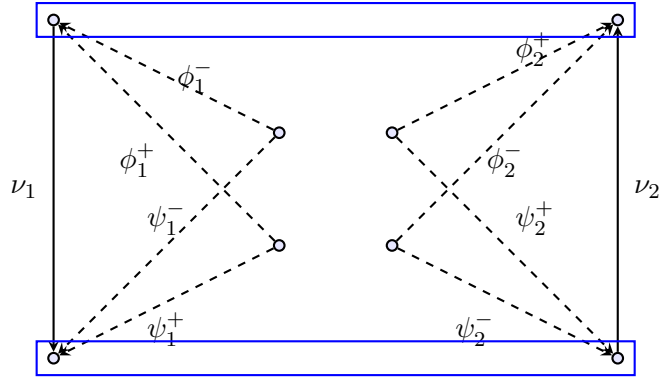
\begin{center}
\begin{tikzpicture}[scale=0.75,
    thick,
    dot/.style={fill=blue!10,circle,draw, inner sep=1pt, minimum size=4pt}]
 \draw (0,3) node[dot] (a1){} (0,-3)node[dot](b1){} (4,-1)node[dot](c1){} (4,1)node[dot](d1){}
       (10,3) node[dot] (a2){} (10,-3)node[dot](b2){} (6,-1)node[dot](c2){} (6,1)node[dot](d2){}
 ;

 \draw[-stealth,thick] (a1) to node[left=0.75mm]{$\nu_1$} (b1);
\draw[-stealth,thick] (b2) to node[right=0.75mm]{$\nu_2$} (a2);

\draw[-stealth, thick, dashed](c1)-- node [below left]{$\phi_1^+$}(a1);
\draw[-stealth, thick, dashed](c1)--node [below ]{$\psi_1^+$}(b1);

\draw[-stealth, thick, dashed](d1)-- node [right]{$\phi_1^-$}(a1);
\draw[-stealth, thick, dashed](d1)--node [above ]{$\psi_1^-$}(b1);

\draw[-stealth, thick, dashed](c2)-- node [below]{$\phi_2^-$}(a2);
\draw[-stealth, thick, dashed](c2)--node [below left]{$\psi_2^-$}(b2);

 \draw[-stealth, thick, dashed](d2)-- node [above right ]{$\phi_2^+$}(a2);
\draw[-stealth, thick, dashed](d2)--node [above right]{$\psi_2^+$}(b2);

\draw[blue] ($(a1)-(3mm,3mm)$ ) rectangle  ($(a2)+(3mm,3mm)$ );
\draw[blue] ($(b1)-(3mm,3mm)$ ) rectangle  ($(b2)+(3mm,3mm)$ );
\end{tikzpicture}
\end{center}
\begin{itemize}
\item $\Ob(\catA) = \bigl\{a_+, a_-, b_+, b_-, c, d\bigr\}$.
\item The only non--zero morphism spaces of the category $\catA$ are:
\begin{itemize}
\item $\catA(a_\pm, a_\pm) = \kk 1_{a_\pm}$, $\catA(b_\pm, b_\pm) = \kk 1_{b_\pm}$ $\catA(c, c) = \kk 1_{c}$ and $\catA(d, d) = \kk 1_{d}$.
\item $\catA(c, d) = \kk\nu_1$ and $\catA(d, c) = \kk\nu_2$ (hence, $\nu_1 \nu_2 = 0$ and $\nu_2 \nu_1 = 0$).
\end{itemize}
\item The bimodule $\catB$ is defined as follows:
\begin{itemize}
\item $\catB(a_\pm, c) = \langle \phi_1^\pm, \phi_2^\pm\rangle_{\kk} \cong \kk^2 \cong \catB(a_\pm, d) = \langle \psi_1^\pm, \psi_2^\pm\rangle_{\kk}$.
    \item For $(x, y) \notin \bigl\{(a_\pm, c), (a_\pm, d)\bigr\}$ we have: $\catB(x, y) = 0$.
\item The action of  $\catA$ on  $\catB$ is given by the following rules:
\begin{itemize}
\item $\nu_1 \circ \phi_1^\pm  = \psi_1^\pm$ and $\nu_2 \circ \psi_2^\pm  = \phi_2^\pm$.
\item $\nu_1 \circ \phi_2^\pm = 0$ and $\nu_2 \circ \psi_1^\pm = 0$.
\end{itemize}
\end{itemize}
\end{itemize}
\noindent
The matrix problem, underlying the description of the isomorphism classes of objects of the bimodule category $\el(\catA, \catB)$ is the following.
\begin{itemize}
\item  We have two matrices $\Theta_1$ and $\Theta_2$ over the field $\kk$. Each $\Theta_i$ is divided into four blocks
$\Phi_i^\pm, \Psi_i^\pm$ for $i = 1, 2$.
   The  blocks $\Phi_1 = \left(\Phi_1^+ \big| \Phi_1^-\right)$ and $\Phi_2 = \left(\Phi_2^+ \big| \Phi_2^-\right)$   $\bigl($respectively,
$\Psi_1 = \left(\Psi_1^+ \big| \Psi_1^-\right)$ and $\Psi_2 = \left(\Psi_2^+ \big| \Psi_2^-\right)$$\bigr)$ of the matrices $\Theta_1$ and $\Theta_2$ have the same number of rows.
\item
The admissible  transformations of  columns and rows of  the matrices $\Theta_1$ and $\Theta_2$,
are compositions of the following  ones.
\begin{itemize}
\item
We can perform  any \emph{simultaneous}  elementary transformations of rows  of  the blocks $\Phi_1$ and $\Phi_2$ (respectively,  $\Psi_1$ and $\Psi_2$)
of the matrices $\Theta_1$ and $\Theta_2$.
\item    We can perform  \emph{independent}  elementary transformations
\begin{itemize}
\item of columns  of each block $\Theta_i^\pm = \left(\begin{array}{c} \Phi_i^\pm \\ \hline \Psi_i^\pm\end{array} \right)$ of the matrix $\Theta_i$ for $i = 1, 2$.
\item we may add   an arbitrary multiple of any row of the block $\Phi_1$ to any row
of the block $\Psi_1$ of the matrix $\Theta_1$.
\item similarly, we may add   an arbitrary multiple of any row of the block $\Psi_2$ to any row
of the block $\Phi_2$ of the matrix $\Theta_2$.
\end{itemize}
\end{itemize}
\end{itemize}
This matrix problem can be visualized by the following picture.
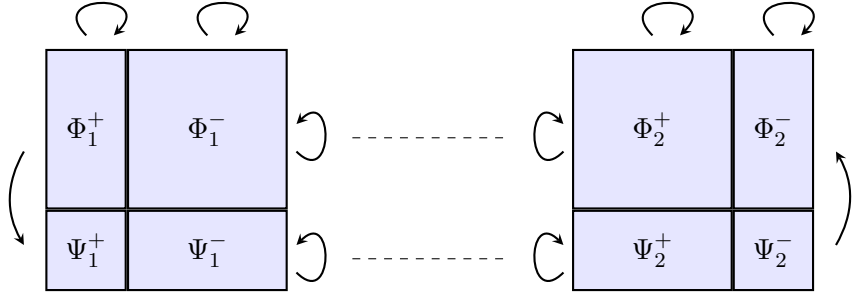
\begin{center}
 \begin{tikzpicture}
  \matrix[nodes={draw, thick, fill=blue!10},
    row sep=0pt, ]  {
  \node[ minimum height=60pt, minimum width=30pt](h01){$\Phi_1^+$}; &\node[ minimum height=60pt, minimum width=60pt](h1){$\Phi_1^-$};&\\
  \node[ minimum height=30pt, minimum width=30pt](s01){$\Psi_1^+$};&\node[ minimum height=30pt, minimum width=60pt](s1){$\Psi_1^-$};& \\
    }  ;

  \matrix[nodes={draw,thick, fill=blue!10},
    row sep=0pt, ]  at (7,0) {
  \node[ minimum height=60pt, minimum width=60pt](h02){$\Phi_2^+$}; &\node[ minimum height=60pt, minimum width=30pt](h2){$\Phi_2^-$};&\\
  \node[ minimum height=30pt, minimum width=60pt](s02){$\Psi_2^+$};&\node[ minimum height=30pt, minimum width=30pt](s2){$\Psi_2^-$};& \\
    }
  ;

\node (hr1) [base right=-0.25 cm of h1, inner sep=0pt, minimum size=10pt ]{};
\node (sr1) [base right=-0.25 cm of s1, inner sep=0pt, minimum size=10pt ]{};

\node (hl2) [base left=-0.25 of h02,inner sep=0pt, minimum size=10pt ]{};
\node (sl2) [base left=-0.25 cm of s02,inner sep=0pt, minimum size=10pt  ]{};

\node (hl1) [base left=0 cm of h01, inner sep=0pt, minimum size=10pt ]{};
\node (sl1) [base left=0 cm of s01, inner sep=0pt, minimum size=10pt ]{};

\node (hr2) [base right=0 cm of h2,inner sep=0pt, minimum size=10pt  ]{};
\node (sr2) [base right=0 cm of s2,inner sep=0pt, minimum size=10pt  ]{};

\tikzset{every loop/.style={min distance=5mm,in=135,out=225,looseness=5}}
\draw[-stealth,thick] (hl2)to[loop ] node[minimum size=20pt,inner sep=10pt](hh2){}   ();
\draw[-stealth,thick] (sl2)to[loop ] node[minimum size=20pt,inner sep=10pt](ss2){}   ();

\tikzset{every loop/.style={min distance=5mm,in=45,out=315,looseness=5}}
\draw[-stealth,thick] (hr1)to[loop ] node(hh1)[minimum size=20pt,inner sep=10pt]{}   ();
\draw[-stealth,thick] (sr1)to[loop ] node(ss1)[minimum size=20pt,inner sep=10pt]{}   ();

 \draw[dashed](hh1) to node[midway,above]{} (hh2);
 \draw[dashed](ss1) to node[midway,above]{} (ss2);

\draw[-stealth,thick](hl1) to[bend right]node[midway,left]{}  (sl1);
\draw[-stealth,thick](sr2) to[bend right]node[midway,right]{} (hr2) ;

\node(ha1)[base  right =0 cm of h1.north, inner sep=0pt, minimum size=10pt  ]{};
\node(ha01)[base  right =0 cm of h01.north, inner sep=0pt, minimum size=10pt  ]{};

\node(ha2)[base  right =0 cm of h2.north, inner sep=0pt, minimum size=10pt  ]{};
\node(ha02)[base  right =0 cm of h02.north, inner sep=0pt, minimum size=10pt  ]{};


 \tikzset{every loop/.style={min distance=8 mm,in=45,out=135,looseness=5}}

\draw[-stealth,thick] (ha01)to[loop ] node [minimum size=20pt,inner sep=10pt]{}   ();
\draw[-stealth,thick] (ha1) to[loop ] node [minimum size=20pt,inner sep=10pt]{}   ();

\draw[-stealth,thick] (ha2) to[loop ] node [minimum size=20pt,inner sep=10pt]{}   ();
\draw[-stealth,thick] (ha02) to[loop ] node [minimum size=20pt,inner sep=10pt]{}   ();

\end{tikzpicture}
\end{center}
\end{example}

\begin{example}[Chessboard problem]\label{Ex:ChessboardProblem} For any  $n \in \NN$, we have   the following   category $\catA$ and $\catA$--bimodule $\catB$.
\begin{itemize}
\item $\Ob(\catA) = \bigl\{x_1, \dots, x_n, y_1, \dots, y_n\bigr\}$.
\item For any $z \in \Ob(\catA)$ we have: $\catA(z, z) = \kk 1_z$. Next, for any $1 \le i \le j \le n$ we have:
$$
\catA(x_i, x_j) = \kk \cdot p_{ji}, \; \mbox{\rm and} \; \catA(y_j, y_i)=  \kk \cdot q_{ij},
$$
whereas all the remaining morphism spaces are zero.
In this notation, $p_{ii}$ and $q_{ii}$ are the identity morphisms for any $1 \le i \le n$. Moreover, for any
$1 \le i \le j \le  k \le n$, we have the composition rules:
$$
p_{kj} p_{ji} = p_{ki} \; \mbox{\rm and} \; q_{ij} q_{jk}  = q_{ik}.
$$
\item For any $1 \le k, l \le n$ we have: $\catB(y_l, x_k) = \kk \cdot \omega_{kl}$, whereas $\catB(z', z'') = 0$ unless $z'= y_l$ and $z'' = x_k$ for some
$1 \le k, l \le n$.
\item The action of the category $\catA$ on the bimodule $\catB$ is specified by the formula
$$
q_{lk} \circ \omega_{kj} \circ p_{ij} = \omega_{li}
$$
for any $1 \le i \le j \le n$ and $1 \le l \le k \le n$.
\end{itemize}
The matrix problem, underlying the description of the isomorphism classes of objects of the bimodule category $\el(\catA, \catB)$ is the following.
\begin{itemize}
\item We have a square matrix $\Theta$ over the field $\kk$, divided into $n$ horizontal and $n$ vertical stripes, labelled
by the symbols $x_1, \dots, x_n$, respectively $y_1, \dots, y_n$. Moreover, for any $1 \le i \le n$, the number of rows of the horizontal block
$x_i$ is equal to the number of columns of the vertical block $y_i$.
\item For any $1 \le i \le n$ one can perform an arbitrary elementary transformation of  rows of the $x_i$-th stripe \emph{simultaneously} with the inverse elementary transformation of columns of the $y_i$-th stripe of the matrix  $\Theta$.
\item For any $1 \le i \le j \le n$ one can add any multiple of any row of $x_i$-th stripe to any row
of $x_j$-th stripe. Similarly, one can add any multiple of any column of $y_i$-th stripe to any  column
of $y_j$-th stripe.
\end{itemize}
This matrix problem can be visualized by the following picture.

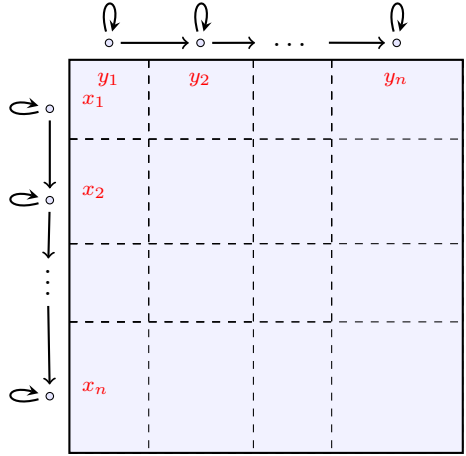
\newcommand{\sfrm}[3]{
\node[draw,solid, thick, fit=(#1-1-1)(#1-#2-#3), inner sep=0pt]{};}
\begin{center}
\begin{tikzpicture}
[ dot/.style={fill=blue!10,circle,draw, inner sep=1pt, minimum size=1pt},
str/.style={inner sep=1pt, minimum size=0pt}
]

\matrix (first) [tbl5,  name=tbl,
row 1/.style={minimum height=30pt},
row 2/.style={minimum height=40pt},
row 3/.style={minimum height=30pt},
row 4/.style={minimum height=50pt},
column 1/.style={text width=30pt},
column 2/.style={text width=40pt},
column 3/.style={text width=30pt},
column 4/.style={text width=50pt},
] at (0,0)
{
~ & ~ & ~ & ~                \\
~ & ~ & ~ & ~                \\~ & ~ & ~ & ~                \\~ & ~ & ~ & ~                \\%
};
\sfrm{tbl}{4}{4};

\foreach \x in {1,..., 4}{\foreach \y in {1,..., 4}{
\hdline{tbl}{\x}{\y}; \vdline{tbl}{\x}{\y};
}

\ifnum\x =3
\node[above=1pt of tbl-1-\x, outer sep=3pt, text depth= -0.5ex, text height=0pt](a\x){$\dots$};
\node[left=0pt of tbl-\x-1, outer sep=3pt, text depth= -0.5ex, text height=7pt](l\x){$\vdots$};
 \else
{
\node[above=2pt of tbl-1-\x, dot, outer sep=3pt](a\x){};
\node[base left=3pt of tbl-\x-1, outer sep=3pt, dot](l\x){};
}
\fi


\ifnum\x=1
\else
 \pgfmathsetmacro\i{\x-1};
 \draw[thick] (a\i)[ ->] to node[above =2pt]{$$}  (a\x);
  \fi
}

\node[below= 1pt of tbl-1-1.north,red ]{$_{y_1}$};
\node[below= 1pt of tbl-1-2.north,red ]{$_{y_2}$};
\node[below= 1pt of tbl-1-4.north,red ]{$_{y_n}$};
\node[base right=1pt of tbl-1-1.west,red]{$_{x_1}$};
\node[base right=1pt of tbl-2-1.west,red]{$_{x_2}$};
\node[base right=1pt of tbl-4-1.west,red]{$_{x_n}$};

\draw[-stealth,thick] (l1)to[loop left]  node{} ();
\draw[-stealth,thick] (a1)to[loop above]  node{} ();
\draw[-stealth,thick] (l2)to[loop left]  node{} ();
\draw[-stealth,thick] (a2)to[loop above]  node{} ();
\draw[-stealth,thick] (l4)to[loop left]  node{} ();
\draw[-stealth,thick] (a4)to[loop above]  node{} ();

\draw[thick,->] (l1)--node[left=2pt]{$$} (l2); \draw[thick,->] (l2)--node[left=2pt]{$$}(l3);
\draw[thick,->] (l3)--node[left=2pt]{$$}(l4);
%

\end{tikzpicture}
\end{center}
\end{example}

\begin{definition}[Bunches of semi--chains] Let $\Sigma$ be a finite or countable set. Assume that for any $i \in \Sigma$, we have totally ordered
sets (chains) $\dE_i$ and $\dF_i$ such that
$\dE_i \cap \dE_j  = \dF_i \cap \dF_j = \emptyset$ for all $i \ne j \in \Sigma$ and $\dE_i \cap \dF_j = \emptyset$ for all $i, j \in \Sigma$.
The sets $\dE_i$ and $\dF_i$ are usually assumed to be  finite or countable. We put
$$
\dE := \bigcup\limits_{i \in \Sigma} \dE_i, \quad
\dF := \bigcup\limits_{i \in \Sigma} \dF_i \quad \mbox{\rm and} \quad
\ddX := \dE \cup \dF.
$$
Note that $\ddX$
is a \emph{partially ordered} set with a partial order $\le$ arising from the total order on the sets $\dE_i$ and $\dF_i$. Let  $\sim$ be a symmetric but not necessarily reflexive relation on $\ddX$ such that for any $x \in \ddX$ there exists a\emph{t most one}
$y \in \ddX$ such that $x \sim y$. Then the datum
\begin{equation}\label{E:bunchofsemichains}
\dX := \bigl(\Sigma, \{\dE_i\}_{i \in \Sigma}, \{\dF_i\}_{i \in \Sigma}, \le, \sim\bigr)
\end{equation} is called
\emph{bunch of semi--chains}. In the case $x \not\sim x$ for any $x \in \ddX$, we shall say that $\dX$ is a \emph{bunch of chains}.
Note that we automatically get another symmetric (but not reflexive) relation $-$ on the set $\ddX$. Namely, for any $x, y \in \ddX$ we say that
$x-y$ if there exists $i \in \Sigma$ such that $(x, y)  \in \bigl(\dE_i \times  \dF_i\bigr) \cup \bigl(\dF_i \times  \dE_i\bigr)$.
\end{definition}

\begin{definition}[Bimodule problem associated with a bunch of semi--chains]
Let $\dX$ be a bunch of semi--chains.

\smallskip
\noindent
1.~In the first step, we define the following $\kk$--linear category $\catC$.
\begin{itemize}
\item The objects of $\catC$ are elements of the set $\ddX$.
\item For any $x \in \Ob(\catC)$ we have: $\catC(x, x) = \kk \cdot 1_x$. Next, for any $x, y \in \Ob(\catC)$ we have
 $$
 \catC(x, y) = \left\{
 \begin{array}{cl} \kk \cdot p_{yx} & \mbox{\rm if} \; x \le y \\
 0 & \mbox{\rm otherwise}
 \end{array}
 \right.
 $$
 In particular, we put $p_{xx} = 1_x$ for any $x \in \ddX$.
 \item The composition of morphisms in $\catC$ is determined by the rule $p_{zy} \cdot p_{yx} = p_{zx}$ for any elements $x \le y \le z$ in $\ddX$.
\end{itemize}
Note that the relation $\sim$ does not play any role in the definition of the category $\catC$.

\smallskip
\noindent
2.~In the second  step, we define a new category $\catA$ based on the category $\catC$.
\begin{itemize}
\item Firstly, we construct a new partially ordered set $\overline{\ddX}$, replacing each element $x \in \ddX$ such that $x \sim x$ by a pair of  new elements $x_+$ and $x_-$, which are no longer self--equivalent. The element $x \in \ddX$ is called \emph{predecessor} of the elements $x_\pm \in \overline{\ddX}$.
    In this way, we get new (partially ordered) sets $\overline{\dE}_i$ and $\overline{\dF}_i$ for each $i \in \Sigma$.
    Namely, for any $y, z \in \overline{\dE}_i$ (respectively, $\dF_i$) we say that $y < z$ if it is the case for the corresponding predecessors.
    However, the elements $x_+$ and $x_-$ are incomparable in the new set $\overline{\ddX}$.
\item Next, we construct yet another    set $\widetilde{\ddX}$ based on  $\overline{\ddX}$, replacing any  pair of elements  $x', x'' \in \overline{\ddX}$ such that $x' \sim x''$,  by a single element  $x \in \widetilde{\ddX}$. The elements $x', x'' \in \ddX$ are called \emph{predecessors} of the element $x \in \widetilde{\ddX}$. We say  that $x'$ and $x''$  are \emph{tied}.
\item An element $\tilde{x} \in \widetilde{\ddX}$ with two predecessors will be called \emph{element of the first type}. Elements of $\widetilde{X}$
of the form $x_\pm$ are of the \emph{second type}, whereas the remaining ones are of  the \emph{third type}.
    For $(x, \tilde{x}) \in \ddX \times \widetilde{\ddX}$  we shall write $x \triangleleft \tilde{x}$, provided  $x$ is a predecessor of  $\tilde{x}$.
\item We put: $\Ob(\catA) := \widetilde{\ddX}$. Next,  $\catA(x_\pm, x_\mp) := 0$  for any pair of objects $(x_+, x_-) \in \widetilde{\ddX} \times \widetilde{\ddX}$ of the second type. For any
 pair $(\tilde{u}, \tilde{v}) \in \widetilde{\ddX} \times \widetilde{\ddX}$,  which is  not of the form $(x_\pm, x_\mp)$, we put:
$
\catA(\tilde{u}, \tilde{v}) := \bigoplus_{\substack{u \triangleleft \tilde{u}\\
v \triangleleft \tilde{v}}} \catC(u, v).
$
\item The composition of morphisms
$
\catA(\tilde{v}, \tilde{w}) \times \catA(\tilde{u}, \tilde{v}) \lar \catA(\tilde{u}, \tilde{w})
$
is specified by the rule: $p_{wu} \cdot p_{vu} = p_{wu}$ provided $u \le v \le w$, whereas all other compositions are defined to be zero.
\end{itemize}

\smallskip
\noindent
3.~Finally,  we define a bimodule $\catB$ over the category $\catA$.
\begin{itemize}
\item For any $i \in \Sigma$,  $x \in \dE_i$ and  $y\in \dF_i$,  we introduce a symbol
$\omega_{yx}$.
\item For any objects $\tilde{x}, \tilde{y}$ of the category $\catA$, we put:
$$
\catB(\tilde{x}, \tilde{y}) := \bigoplus\limits_{\substack{y \triangleleft \tilde{y}, \tilde{y} \in \dF\\ x \triangleleft \tilde{x}, \tilde{x} \in \dE \\ y-x}} \kk \cdot \omega_{yx}.
$$
\item Let $\tilde{x}, \tilde{y}, \tilde{u}$ and  $\tilde{v}$ be objects of $\catA$. Then the  action
$$
\catA(\tilde{v}, \tilde{y}) \times \catB(\tilde{u}, \tilde{v}) \times \catA(\tilde{x}, \tilde{u}) \lar \catB(\tilde{x}, \tilde{y})
$$
of the category $\catA$ on the bimodule $\catB$ is specified by the rules
$$
p_{yv} \circ \omega_{vu} \circ p_{ux} = \omega_{yx} \; {\rm for} \;
u \triangleleft \tilde{u}, v \triangleleft \tilde{v}, x \triangleleft \tilde{x} \; {\rm and} \; y \triangleleft \tilde{y},
$$
whereas all other compositions are put to be zero.
\end{itemize}
We call the bimodule category  $\Rep(\dX) := \el(\catA, \catB)$ the \emph{category of representations} of a bunch of semi--chains $\dX$. \qed
\end{definition}

\begin{remark}
In explicit terms, the description of the isomorphism classes of objects of the category $\Rep(\dX)$ is given by the following matrix problem.
\begin{itemize}
\item For any $i \in \Sigma$, we have a matrix $\Theta_i$ (of finite size) over the field $\kk$, which is divided into horizontal and vertical stripes, labelled by the elements of the sets $\overline{\dE}_i$ and $\overline{\dF}_i$, respectively (since the sets $\overline{\dE}_i$ and $\overline{\dF}_i$ are partially ordered, the corresponding labels of rows and columns of $\Theta_i$ will be also called \emph{weights}).  Note that all but finitely many matrices $\Theta_i$ are equal to zero.
    \item Let $x, y \in \overline{\ddX}$ be such that $x \sim y$ and $x \in \overline{\dE}_i \cup \overline{\dF}_i$, whereas $y \in \overline{\dE}_j \cup \overline{\dF}_j$. Then the stripes  of the matrices $\Theta_i$ and $\Theta_j$ labelled by $x$ and $y$ have the same number of rows/columns and called
        \emph{conjugate}.
\end{itemize}
We are allowed to perform   compositions of the following (elementary) transformations with the collection of matrices $\bigl\{\Theta_i\bigr\}_{i \in \Sigma}$.
\begin{itemize}
\item For any $x < x' \in \overline{\dE}_i$, we may add a multiple of any row of the matrix $\Theta_i$ of weight $x$ to any row of weight $x'$.
\item Similarly, for any $y < y' \in \overline{\dF}_i$, we may add a multiple of any column  of the matrix $\Theta_i$ of weight $y'$ to any column of weight $y$.
\item Assume that $x \in \overline{\dE}_i \cup \overline{\dF}_i$ and $y \in \overline{\dE}_j \cup \overline{\dF}_j$ are such that $x \sim y$.
If $(x, y) \in \bigl(\overline{\dE}_i  \times \overline{\dE}_j\bigr) \cup \bigl(\overline{\dF}_i  \times \overline{\dF}_j\bigr)$ then we may do any  simultaneous transformations with stripe  of $\Theta_i$ labelled by $x$ and the conjugated stripe  of $\Theta_j$ labelled by $y$. Otherwise (i.e.~if  $(x, y) \in \bigl(\overline{\dE}_i  \times \overline{\dF}_j\bigr) \cup \bigl(\overline{\dF}_i  \times \overline{\dE}_j\bigr)$), we may perform  any elementary transformation within the block $x$ of $\Theta_i$ simultaneously with the inverse elementary transformation within the block $y$ of $\Theta_j$.
\end{itemize}
\end{remark}

\begin{example}\label{Ex:bunchesofChains2} Let $\Sigma = \{1, 2\}$, $\dE_i = \{c_i, d_i\}$, $\dF_i = \{a_i\}$  and $c_i < d_i$ for $i = 1, 2$. Furthermore, assume that
$a_1 \sim a_2$, $c_1 \sim c_2$ and $d_1 \sim d_2$. Denoting $\nu_i = p_{d_i c_i}$, $\phi_i:= \omega_{a_i c_i}$ and $\psi_i:= \omega_{a_i d_i}$ for $i = 1, 2$, we get the matrix problem from  Example \ref{Ex:bunchesofChains}.
\end{example}

\begin{example}\label{Ex:bunchesofSemiChains2} Let $\Sigma = \{1, 2\}$, $\dE_i = \{c_i, d_i\}$, $\dF_1 = \{a\}$, $\dF_2 = \{b\}$  and $c_i < d_i$ for $i = 1, 2$. Furthermore, assume that
$a \sim a$, $b \sim b$, $c_1 \sim c_2$ and $d_1 \sim d_2$.  This gives the matrix problem from  Example \ref{Ex:bunchesofSemiChains}.
\end{example}

\begin{example}\label{Ex:ChessboardProblem2} Let $\Sigma = \{*\}$, $\dE_* = \{x_1, \dots, x_n\}$, $\dF_* = \{y_1, \dots, y_n\}$, $x_1 < \dots < y_n$, $y_n < \dots < y_1$ and
$x_i \sim y_i$ for all $1 \le i \le n$. Then the corresponding matrix problem is the one from Example \ref{Ex:ChessboardProblem} (chessboard problem).
\end{example}

\begin{theorem}\label{T:BunchOfSemiChains}
Let $\dX$ be a bunch of semi--chains. Then the corresponding bimodule category $\Rep(\dX)$ has tame representation type (in some cases it can have discrete, or even finite representation type).
\end{theorem}

\smallskip
\noindent
\emph{Comment to the proof}. A  bunch--of--chains type matrix problems appeared  for the first time in a work of Nazarova and Roiter \cite{NazarovaRoiterDyad} in the context of classification of the indecomposable finite length modules over algebra $\kk\llbracket x, y\rrbracket/(xy)$ (or more generally, over the dyad of two discrete valuation rings). In that work, the canonical forms describing the underlying  indecomposable objects were also constructed. The  chessboard problem from Example \ref{Ex:ChessboardProblem} appeared  for the first time in \cite{NazarovaRoiterDyad}.

Matrix problems of  bunch--of--semi--chains type were introduced in a subsequent  work of the same authors \cite{NazarovaRoiter}. Nazarova and Roiter gave a correct proof of tameness of these  matrix problems. However, the precise combinatorial
pattern of the underlying  indecomposable objects stated in \cite{NazarovaRoiter} turned out to be wrong. The right combinatorics of the indecomposable objects for bunches of semi--chains was established by Bondarenko in \cite{BondarenkoMain}. An elaboration of the reduction algorithm of Nazarova and Roiter for  bunch--of--chains type matrix problems was carried out by Klingler and Levy \cite{KlinglerLevy}, whereas the general bunches of semi--chains were treated  in a work of Deng \cite{Deng}.

\smallskip
In \cite{Nodal}, the authors considered a slight generalization of representations of bunches of semi--chains, which included the case when we have
\emph{countably many}  matrices $\Theta_i$ of finite size. It turned out that the reduction algorithm of \cite{NazarovaRoiter, BondarenkoMain} extends literally
on this case as well, leading to a similar combinatorial pattern for indecomposable objects; see also \cite{BondarenkoBlaBla}. Finally, in \cite{BurbanDrozdMemoirs}, the authors discovered a new type of tame matrix problems called representations of \emph{decorated} bunches of chains.
An interested reader might consult  \cite[Chapter 13]{BurbanDrozdMemoirs} for a detailed treatment of  a class of tame matrix problems, which include representations of a bunch of chains. \qed

\smallskip
Now, we state the classification of the indecomposable objects of the category $\Rep(\dX)$, where
$\dX$ is a bunch of \emph{chains}. Note that in this case we have: $\ddX = \overline{\ddX}$. For any $u \in \ddX$ we denote by $|u|$ its class in
$\widetilde{\ddX}$ and by $Z_{|u|}$ the corresponding object of the category $\catA^\omega$ (which is the additive closure of $\catA$).

\medskip
\noindent
1.~An $\dX$\emph{--word}  is  a sequence
$\underline{v}=x_1 \rho_1 \dots
x_{l-1} \rho_{l-1} x_l$, where $x_i\in \ddX$, $\rho_i\in \bigl\{\sim,-\bigr\}$ and the following conditions are satisfied:
 \begin{itemize}
\item   $x_i \rho_i x_{i+1}$ in $\ddX$ for each $i\in \bigl\{1,2,\dots,l-1\bigr\}$.

\item   $\rho_i\ne \rho_{i+1}$ for each $i\in \bigl\{1,2,\dots,l-2\bigr\}$.

\item  If $x_1$ is tied, then $\rho_1=\sim$, and if $x_l$ is tied, then $\rho_{l-1}=\sim$.
\end{itemize}
 We call $l$ the \emph{length} of the word $\underline{v}$ and denote it by $l(\underline{v})$. Next, we denote by $[\underline{v}]$ the sequence of elements of $\widetilde{\ddX}$ obtained
 from $\underline{v}$ by replacing any subword of the form $x'\sim x''$ by the corresponding element of $\widetilde{\ddX}$ and deleting all symbols $-$.
 For any such subword $x'\sim x''$ we put: $\sigma(x', x'') := 0$ if $(x', x'') \in (\dE \times \dE) \cup (\dF \times \dF)$ and $\sigma(x', x'') := 1$
 otherwise.

\medskip
\noindent
2.~For an $\dX$--word $\underline{v}$ we denote by $\underline{v}^\circ$  its \emph{inverse word} to $\underline{v}$, i.e.~$\underline{v}^\circ=x_l \rho_{l-1}x_{l-1}\dots \rho_2 x_2 \rho_1x_1$.

\medskip
\noindent
3.~An $\dX$--word $\underline{w}$ of length $l$ is called \emph{cyclic} if $\rho_1 = \sim = \rho_{l-1}$ and
we have a relation $x_l-x_1$ in $\ddX$.
 For such a cyclic word we set $\rho_l=\!-$ and define $x_i,\rho_i$ for all $i\in\ZZ$ by the formula
$x_{i+ql}=x_i,\,\rho_{i+ql}=\rho_i$ for any $q\in\ZZ$. In particular, $x_0=x_l$ and $\rho_0=\!-$.
Note that the length of a cyclic word is always even: $l = 2s$ for some $s \in \NN$.

\medskip
\noindent
4.~If $\underline{w}$ is a cyclic word, then  its \emph{$k$--shift} is the word
$
\underline{w}^{(k)}=x_{2k+1} \rho_{2k+1} \rho_{2k+2}\dots x_{2k-1} \rho_{2k-1}x_{2k}.
$
We put: $\sigma(k,\underline{w})=\sum_{j=1}^k\sigma(x_{2j-1},x_{2j})$.
Next, we call  a cyclic word $\underline{w}$ of length $l$   \emph{periodic} if $\underline{w}^{(k)}=\underline{w}$  for
some $k<l$. In what follows, we shall also use the following notation for a cyclic word $\underline{w}$:
$$
\underline{w} := \lha x_1 \sim  x_2 - x_2 \sim x_4 - \dots - x_{2s-1} \sim x_{2s} \rha
$$

\medskip
\noindent
 Now we introduce  the following objects of the bimodule category $\Rep(\dX)$.

\smallskip
\noindent
\underline{\textsl{Strings}}.
 Let $\underline{v}$ be an $\dX$--word.
 Then the  corresponding  \emph{string representation} $S(\underline{v}) = (Z, \vartheta)$ is defined as follows:
\begin{itemize}
\item  $Z=\bop_{|u|\in[\underline{v}]} Z_{|u|}$  and $\vartheta\in \catB(Z, Z)$.

\item  Let
$u_i-u_{i+1}$ be a subword of $\underline{v}$.
Then $\catB(Z, Z)$ has a direct summand $$\catB\bigl(Z_{|u_{i+1}|},Z_{|u_i|}\bigr)
\oplus \catB\bigl(Z_{|u_{i}|}, Z_{|u_{i+1}|}\bigr)$$ and the corresponding component of $\vartheta$ is
\begin{itemize}
\item $ \omega_{u_{i+1} u_i} \in \catB\bigl(Z_{|u_{i}|}, Z_{|u_{i+1}|}\bigr)$ if $u_{i+1} \in \dE$
and $
u_i \in \dF$,
\item
$\omega_{u_{i} u_{i+1}} \in \catB\bigl(Z_{|u_{i+1}|}, Z_{|u_{i}|}\bigr)$ if $u_{i} \in \dE$
and $
u_{i+1} \in \dF$.
\end{itemize}
\item  All other components of $\vartheta$ are set to be zero.
\end{itemize}

\smallskip
\noindent
\underline{\textsl{Bands}}.
 Let $\underline{w}$ be a cyclic $\dX$--word of length $l=2n$, $m \in \NN$ and $\pi \in \kk^\ast$. Let $I$ be the identity matrix of size $m$, whereas $J$ is  a Jordan block
 of size $m$ with eigenvalue $\pi$.
 The \emph{band representation}
$B\bigl(\underline{w}, m,\pi\bigr) = (Z, \vartheta)$ is defined as follows:
\begin{itemize}
\item  $Z=\bop_{|v|\in[\underline{w}]} Z_{|v|}^{\oplus m}$ and  $\vartheta \in\catB(Z,Z)$.

\item  Suppose
$u_i-u_{i+1}$ is a subword of $\underline{w}$. Then $\catB(Z, Z)$ has a
direct summand $$\catB\bigl(Z_{|u_{i+1}|}^{\oplus m}, Z_{|u_i|}^{\oplus m}\bigr) \oplus
\catB\bigl(Z_{|u_{i}|}^{\oplus m}, Z_{|u_{i+1}|}^{\oplus m}\bigr)$$ and we define the
corresponding component of $\vartheta$ as follows:
\begin{itemize}
\item $
\omega_{u_{i+1} u_i} I \in \catB\bigl(Z_{|u_{i}|}^{\oplus m}, Z_{|u_{i+1}|}^{\oplus m}\bigr)$
if $ u_{i} \in \dF$ and $
u_{i+1} \in \dE$,
\item
$
\omega_{u_{i} u_{i+1}} I\in \catB\bigl(Z_{|u_{i+1}|}^{\oplus m}, Z_{|u_{i}|}^{\oplus m}\bigr)$ if $u_{i+1} \in \dF$ and $
u_{i} \in \dE$,
\end{itemize}
 where $I$ is the identity $m \times m$ matrix.

\item  The component of $\vartheta$ corresponding to the direct summand
$$\catB\bigl(Z_{|u_1|}^{\oplus m}, Z_{|u_n|}^{\oplus m}\bigr) \oplus
\catB\bigl(Z_{|u_n|}^{\oplus m}, Z_{|u_1|}^{\oplus m}\bigr)
$$ of $\catB(Z, Z)$
 is defined as
 \begin{itemize}
\item $
 \omega_{u_{1} u_n} J \in \catB\bigl(Z_{|u_{n}|}^{\oplus m}, Z_{|u_{1}|}^{\oplus m}\bigr)$  if $ u_{n} \in \dF$ and $
u_{1} \in \dE$,
\item
$
 \omega_{u_{n} u_{1}} J \in \catB\bigl(Z_{|u_{1}|}^{\oplus m}, Z_{|u_{n}|}^{\oplus m}\bigr)$ if $u_{1} \in \dF$ and $
u_{n} \in \dE$.
\end{itemize}

\item  All other components of $\vartheta$ are zero.
\end{itemize}

\begin{theorem}[see \cite{NazarovaRoiter, BondarenkoMain, BurbanDrozdMemoirs}]\label{T:stringsandbands} Let $\dX$ be a bunch of chains.
Then the  description of the indecomposable objects of $\Rep(\dX)$ is the following.
 \begin{itemize}
\item   Every string or band representation  is indecomposable and every indecomposable
object  of $\Rep(\dX)$  is isomorphic to some string or band representation.

\item  Any string representation is not isomorphic to any band representation.

\item  Two string representations $S(\underline{v})$ and $S(\underline{v}')$ are isomorphic  if and only if $\underline{v}' = \underline{v}$ or $\underline{v}' = \underline{v}^\circ$.

\item  Two band representations $B(\underline{w}, m, \pi)$ and $B(\widetilde{\underline{w}}, \widetilde{m}, \widetilde{\pi})$ are isomorphic if and only if $\widetilde{m} = m$ and
\begin{itemize}
\item $\widetilde{\underline{w}} = \underline{w}^{(k)}$ for some $k \in \ZZ$ and $\tilde{\pi} = \pi$, or
\item $\widetilde{\underline{w}} = {\underline{w}^\circ}^{(k)}$ for some $k \in \ZZ$ and $\widetilde{\pi} = \pi^{\overline{\sigma(k, \underline{w})}}$, where $\overline{\sigma(k, \underline{w})} = (-1)^{\sigma(k, \underline{w})}$. \qed
\end{itemize}
\end{itemize}
\end{theorem}

\begin{example} Consider the  bunch of chains $\dX$ introduced in Example \ref{Ex:bunchesofChains2}.
Let $a \in \widetilde{\ddX}$ (respectively $c, d \in \widetilde{\ddX}$) be the equivalences class of $a_1, a_2 \in \ddX$ (respectively,
$c_1, c_2; d_1, d_2 \in \ddX$). Consider the band datum $(\underline{w}, m, \pi)$, where  $m \in \NN$, $\pi \in \kk^*$ and
$$
\underline{w} := \lha a_1 \sim a_2 - c_2 \sim c_1 - a_1 \sim a_2 - d_2 \sim d_1 -a_1
\sim a_2 - c_2 \sim c_1 \rha
$$
Then $[\underline{w}] = a c a d a c$ and the band object $B(\underline{w}, m, \pi)$ is given by by the pair  $(Z, \vartheta)$, where
$$Z = Z_{a}^{\oplus m} \oplus Z_c^{\oplus m} \oplus Z_{a}^{\oplus m} \oplus Z_{d}^{\oplus m} \oplus Z_{a}^{\oplus m} \oplus
Z_{c}^{\oplus m}$$
and $\vartheta \in \kB(Z, Z)$ is given by the following matrix

\begin{center}
\begin{tikzpicture}
\matrix (first) [tbl5,text width=20pt,  name=table]
{
0               & 0          &0             &0       & 0                &0\\
\phi_2 I & 0       &  \phi_1 I &0       &0                 &0\\
0               & 0          &0             &0       & 0                &0\\
0               &0          & \psi_2 I & 0       & \psi_1 I &0       \\
0               & 0          &0             &0       & 0                &0\\
 \phi_1 J& 0        &0              &0    & \phi_2I      &0\\
};

\node[draw,solid, thick, fit=(table-1-1)(table-6-6), inner sep=0pt]{};

\node[above=3pt of table-1-1]{$a$};
\node[above=3pt of table-1-2]{$c$};
\node[above=3pt of table-1-3]{$a$};
\node[above=3pt of table-1-4]{$d$};
\node[above=3pt of table-1-5]{$a$};
\node[above=3pt of table-1-6]{$c$};


\node[base left=3pt of table-1-1]{$a$};
\node[base left=3pt of table-2-1]{$c$};
\node[base left=3pt of table-3-1]{$a$};
\node[base left=3pt of table-4-1]{$d$};
\node[base left=3pt of table-5-1]{$a$};
\node[base left=3pt of table-6-1]{$c$};

\end{tikzpicture}
\end{center}
In the explicit terms of Example \ref{Ex:bunchesofChains}, it corresponds to the following pair of  matrices:
\newcommand{\sfrm}[3]{
\node[draw,solid, thick, fit=(#1-1-1)(#1-#2-#3), inner sep=0pt]{};}
\begin{center}
  \begin{tikzpicture}

\matrix (first) [tbl5,text width=30pt, minimum height=30pt,  name=table]
{
0   &  I  & 0 \\
  J    &  0 & 0  \\[0.5ex]
   0 & 0 &  I      \\
};
\sfrm{table}{2}{3}; \node[draw, thick, fit=(table-3-1)(table-3-3), inner sep=0pt]{};

\node[above=3pt of table-1-2]{$a_1$};
\draw(table-1-3) to node[midway,base right=23pt]{$c_1$}  (table-2-3);
\node[base right=3pt of table-3-3]{$d_1$};

\draw(table-1-1) to node[midway,base left=25pt]{$\Phi_1\;=$}  (table-2-1);
\node[base left=10pt of table-3-1]{$\Psi_1\;=$};
\foreach \x in {1,2,3}{
\foreach \y in {1,2,3}{
\hdline{table}{\x}{\y};
\vdline{table}{\x}{\y};
}}

\matrix (second) [tbl5,text width=30pt, minimum height=30pt,  name=table]
at(6,0)
{
 I  & 0&0 \\
 0 & 0 & I    \\[0.5ex]
   0 &  I  &0   \\
};
\sfrm{table}{2}{3};
\node[draw, thick, fit=(table-3-1)(table-3-3), inner sep=0pt]{};

\node[above=3pt of table-1-2]{$a_2$};
\draw(table-1-1) to node[midway,base left=23pt]{$c_2$}  (table-2-1);
\node[base left=3pt of table-3-1]{$d_2$};

\draw(table-1-3) to node[midway,base right=25pt, text depth=5pt,]{$=\;\Phi_2$}  (table-2-3);
\node[base right=10pt of table-3-3]{$=\;\Psi_2$};

\foreach \x in {1,2,3}{
\foreach \y in {1,2,3}{
\hdline{table}{\x}{\y};
\vdline{table}{\x}{\y};
}}
\end{tikzpicture}
\end{center}
\end{example}

\begin{remark}
In many  concrete  situations arising in applications, the combinatorics of string and band representations as well as the rules to write the canonical forms, can be significantly simplified; see for instance \cite{Nodal, SurveyBBDG,BurbanGnedin}. As we shall see below, this also happens for the derived categories of a gentle algebra.
\end{remark}

\section{Indecomposables in the  derived category of a gentle algebra}\label{S:GentleCombinatorics}

\noindent
In this section, we shall give a complete classification of the indecomposable objects of the homotopy categories $\mathsf{Hot}^\ast(A)$, where
$* \in \left\{b, +, -, \emptyset\right\}$ and $A$ is a
gentle algebra.

\subsection{Bunch of semi--chains attached to a skew--gentle algebra}
\begin{definition}\label{D:bunchsemichainsMP}
Let $(\vec{\sm}, \simeq)$ be   a datum as in Definition \ref{D:skewgentle}. Then we attach to that the following bunch of semi--chains
$\dX = \dX(\vec{\sm}, \simeq)$.
\begin{itemize}
\item The index set $\Sigma := \Omega \times \ZZ$.
\item For any $((i, j), r) \in \Sigma$, let  $\dE_{((i,j), r)} := \{u_{((i, j), r)}\}$ and $\bigl(\dF_{((i,j), r)}, <\bigr)$ be the totally ordered set defined by (\ref{E:FirstOrder}).
\item The relation $\sim$ is defined by the following rules.
\begin{itemize}
\item If $(i, j) \simeq (i, j)$ in $\Omega$ then we have: $u_{((i, j), r)} \sim u_{((i, j), r)}$ for any $r \in \ZZ$.
\item For any $r \in \ZZ$, $1 \le i \le t$  and $1 \le b < a \le m_i$ we have:
 $q_{((i, b), r)}^{((i, a), r-1)} \sim q^{((i, b), r)}_{((i, a), r-1)}$.
\end{itemize}
\end{itemize}
\end{definition}

\begin{theorem}\label{T:BurbanDrozdTriples} Let $(\vec{\sm}, \simeq)$ be   a datum as in Definition \ref{D:skewgentle}, $A = A(\vec{\sm}, \simeq)$ be the corresponding skew--gentle
 algebra and $\dX = \dX(\vec{\sm}, \simeq)$ be the bunch of semi--chains introduced in Definition \ref{D:bunchsemichainsMP}. Then we have a full functor
\begin{equation}\label{E:Reduction}
\Tri^b(A) \stackrel{\MM}\lar \mathsf{Rep}(\dX),\; \bigl(Y^\bu, V^\bu, \theta\bigr) \lar
\bigl\{\Theta^r_{(i, j)}\bigr\}_{((i, j), r) \in I}
\end{equation}
satisfying the following properties.
\begin{itemize}
\item For any $T_1, T_2 \in \Ob\bigl(\Tri^b(A)\bigr)$ we have: $\MM(T_1) \cong \MM(T_2)$ if and only if $T_1 \cong T_2$.
\item For any $T\in \Ob\bigl(\Tri^b(A)\bigr)$, the object $\MM(T)$ is indecomposable if and only if $T$ is indecomposable.
\end{itemize}
Moreover, the essential image of the functor $\MM$ is given by those collections of decorated matrices  $\bigl\{\Theta^r_{(i, j)}\bigr\}_{((i, j), r) \in \Sigma}$ for which
each entry $\Theta^r_{(i, j)}$ is a square and non--degenerated.
\end{theorem}

\begin{proof} In Subsection \ref{SS:TriplesMP}, we essentially gave a construction of the  functor $\MM$.
It is clear that $\MM$ is full,  the statement about the description of the essential image of  $\MM$ is  obvious.

\smallskip
\noindent
It follows from the description
of morphisms in the category $\mathsf{Hot}^b(H-\mathsf{mod})$ (see again diagram (\ref{E:hierarchyofmorphisms})) that for any object
$T$ of the category $\Tri^b(A)$, all elements of $
\End_{\Tri^b(A)}(T)$  lying in  the kernel of the map
$
\End_{\Tri^b(A)}(T) \xrightarrow{\MM} \End_{\Rep(\dX)}\bigl(\MM(T)\bigr)
$
are nilpotent. Assume that $\MM(T) \cong \MM(S)$ in $\Rep(\dX)$. Then we have mutually inverse isomorphisms $\MM(T) \stackrel{u}\lar \MM(S)$
and $\MM(S) \stackrel{v}\lar \MM(T)$. Since $\MM$ is full, there exists morphisms $T \stackrel{\widetilde{u}}\lar S$
and $S \stackrel{\widetilde{v}}\lar T$ such that $u = \MM(\widetilde{u})$ and $v = \MM(\widetilde{u})$. Since
$\MM\bigl(1_{T} - \widetilde{v} \widetilde{u}\bigr) = 0$, the endomorphism $1_{T} - \widetilde{v} \widetilde{u}$ is nilpotent.
As a consequence, the morphism
$
\widetilde{v} \widetilde{u} = 1_{T} - \bigl(1_{T} - \widetilde{v} \widetilde{u}\bigr)
$
is invertible. Analogously, $\widetilde{u} \widetilde{v}$ is invertible as well, hence both morphisms $\widetilde{u}$ and  $\widetilde{u}$ are isomorphisms. Therefore, the functor $\MM$ reflects isomorphism classes of objects. Since the essential image of $\MM$ is closed under taking direct summands, $\MM$ also reflects indecomposability of objects.
\end{proof}

\medskip
\noindent
\textbf{Summary}. We have constructed a pair of functors $\EE$ and $\MM$:
\begin{equation}\label{E:Summary}
\mathsf{Hot}^b(A-\mathsf{pro}) \stackrel{\EE}\lar \Tri^b(A) \stackrel{\MM}\lar \Rep(\dX),
\end{equation}
both reflecting  the isomorphism classes and indecomposability of objects; see Theorem \ref{T:mainconstr} and Theorem \ref{T:BurbanDrozdTriples}. Moreover, the functor $\EE$ even induces a bijections between the isomorphism classes of objects of the corresponding categories.  Hence, starting with  a classification of indecomposable objects of the category  $\Rep(\dX)$, we can deduce from it  a description of the indecomposable objects of the triangulated category $\mathsf{Hot}^b(A-\mathsf{pro})$.

\subsection{Indecomposable objects of the perfect derived category of a gentle algebra}
From now on, assume that the datum $(\vec{\sm}, \simeq)$ is  such that $(i, j) \not\simeq (i, j)$ for any $(i, j) \in \Omega = \Omega(\vec{\sm})$.
We shall say that an element $(i, j) \in \Omega$ is \emph{tied} if there exists $(k, l) \in
\Omega$ such that $(i, j) \simeq (k, l)$.
Recall that the  algebra $A = A(\vec{\sm}, \simeq)$ is \emph{gentle} and $\dX = \dX(\vec{\sm}, \simeq)$ is a bunch of \emph{chains}.

In this subsection, we are going to explain a classification of indecomposable objects of the perfect derived   category $\mathsf{Hot}^b(A-\mathsf{pro})$.

As it was explained in course of the discussion of Theorem \ref{T:BunchOfSemiChains}, there are two types of indecomposable objects of the
category $\Rep(\dX)$: strings $S(\underline{v})$ and bands $B(\underline{w}, m, \pi)$. A special feature of $\dX$ is the following: the only untied elements of $\ddX$ are of the form $q_{((i, b), r)}^{((i, m_i+1), r-1)}$ for any $((i, b), r) \in \Sigma$, as well as $u_{((j, p),r)}$, where $(j, p) \in \Omega$ is untied.

\begin{lemma}\label{L:reductionOne} In the notation of Definition \ref{D:bunchsemichainsMP}, the following results are true.

\smallskip
\noindent
1.~For any non--periodic cyclic word $\underline{w}$ in $\dX$, $m \in \NN$ and $\pi \in \kk$, the corresponding band object $B(\underline{w}, m, \pi)$ belongs
to the essential image of the functor $\MM$.

\smallskip
\noindent
2.~Let $\underline{v}$ be a $\dX$ word.  Then the corresponding string object $S(\underline{v})$ belongs to the essential image of the functor $\MM$ if and only if
$\underline{v}$ contains elements from both sets $\dE$ and $\dF$ and begins and ends with an \emph{untied} element.
\end{lemma}

\begin{proof} An object of $\Rep(\dX)$ given by a set  of decorated  matrices $\left\{\Theta^r_{(i, j)}\right\}_{((i, j), r) \in \Sigma}$
belongs to the essential image of the functor $\MM$ if and only if all $\Theta^r_{(i, j)}$ are square and non--degenerate.
The statement follows
from the explicit descriptions of $B(\underline{w}, m, \pi)$ and $S(\underline{v})$, given in Theorem \ref{T:BunchOfSemiChains}.
\end{proof}

\begin{remark}
It turns out that the combinatorics of bands $(\underline{w}, m, \pi)$ and strings $\underline{v}$ in the bunch of chains $\dX$,
which satisfy  the conditions of  Lemma \ref{L:reductionOne}, admits the following  significant simplification: having a cyclic word
$\underline{w}$ or a word $\underline{v}$, we can cross all elements from the set $\dE$. Let $\overline{w}$ and $\overline{v}$ be the obtained words.
Then no information is lost under this operation: the words $\underline{w}$ and $\underline{v}$ can be reconstructed   from $\overline{w}$ and $\overline{v}$.
\end{remark}

\begin{definition}\label{D:BandStringsReduced} We introduce the following reduced version of strings and bands, introduced in Theorem \ref{T:BunchOfSemiChains}.

\smallskip
\noindent
A \emph{band datum} is a tuple  $(\overline{w}, m, \pi)$, where $m \in \NN$, $\pi \in \kk^*$ and $\overline{w}$ is a non--periodic sequence
$$
\overline{w} = \lha x_1 \sim x_2 - x_3 \sim x_4 - \dots - x_{2s-1} \sim x_{2s} \rha
$$
such that
\begin{itemize}
\item For any $1 \le e \le s$, each subword $x_{2e-1} \sim x_{2e}$ of $\underline{w}$ is of the form $q_{((i, k), r)}^{((i, l), r-1)}\sim q^{((i, k), r)}_{((i, l), r-1)}$
or $q^{((i, k), r)}_{((i, l), r-1)} \sim q_{((i, k), r)}^{((i, l), r-1)}$ for some $r \in \ZZ$, $1 \le i \le t$ and $1 \le k < l \le m_i$.
\item For any $1 \le e \le s$, each subword $x_{2e}-x_{2e+1}$ of $\overline{w}$  satisfies the following condition: if $x_{2e} = q_{((i, k), r)}^{((i, \bar{k}), r\pm 1)}$
then $x_{2e+1} = q_{((j, l), r)}^{((j, \bar{l}), r\pm 1)}$ for some $1 \le k \ne \bar{k} \le m_i$, $1 \le l \ne \bar{l} \le m_j$ and
$(i, k) \simeq (j, l)$ in $\Omega$. Here, we put $x_{2s+1} := x_1$.
\end{itemize}

\smallskip
\noindent
A \emph{string datum} $\overline{v}$ can be  of the following four types.

\smallskip
\noindent
1.~$\overline{v} = x_1 \sim x_2 - x_3 \sim x_4 - \dots - x_{2s-1} \sim x_{2s}$, where all pairs $x_{2e-1} \sim x_{2e}$ and $x_{2e} - x_{2e+1}$ satisfy the same constraints as for bands, but additionally $x_1 = q_{((i, k), r)}^{((i, \bar{k}), r \pm 1)}$ and $x_{2s} =
    q_{((j, l), d)}^{((j, \bar{l}), d \pm 1)}$ are such that $(i, k)$ and $(j, l)$ are untied in $\Omega$.

\smallskip
\noindent
2.~$\overline{v} = x_0 - x_1 \sim x_2 - x_3 \sim x_4 - \dots - x_{2s-1} \sim x_{2s}$ with  $x_0 = q_{((i, k), r)}^{((i, m_i+1), r - 1)}$ for some
$((i, k), r) \in I$, whereas
$x_{2s}$ as well as all $\sim$ and $-$ satisfy the constraints from  the previous paragraph. If $(i, k) \in \Omega$ is tied then necessarily $s \ge 1$, otherwise $s = 0$.

\smallskip
\noindent
3.~$\overline{v} = \overline{u}^\circ$, where $\overline{u}$ is as in the previous paragraph.

\smallskip
\noindent
4.~Finally, we may have $\overline{v} = x_0 - x_1 \sim x_2 - x_3 \sim x_4 - \dots - x_{2s-1} \sim x_{2s}-x_{2s+1}$ with  $x_0 = q_{((i, k), r)}^{((i, m_i+1), r - 1)}$ and $x_{2s+1} = q_{((j, l), d)}^{((j, m_j+1), d - 1)}$ for some
$((i, k), r), ((j, l), d) \in \Sigma$.

\smallskip
\noindent
The notions of the shift $\overline{w}^{(k)}$ of a cyclic word $\overline{w}$ for $\kk \in \ZZ$ as well as of the opposite words $\overline{w}^\circ$ and
$\overline{v}^\circ$ are  straightforward. \qed
\end{definition}

\begin{definition}[Gluing diagrams]\label{D:gluingdiagrams} Let $(\overline{w}, m, \pi)$ be a band datum and $\overline{v}$ be a string datum, as in above Definition
\ref{D:BandStringsReduced}.

\smallskip
\noindent
1.~We attach to $\overline{v}$  a so--called  \emph{gluing diagram} of indecomposable complexes $W^\bu_{(i, (a, b))}[-r]$, constructed by the following algorithm.
\begin{itemize}
\item For any $r \in \ZZ$, $1 \le i \le t$ and $1 \le a < b \le m_i$ we replace each subword of $\overline{v}$ of  the form
$q_{(((i, b), r)}^{((i, a), r-1)} \sim q^{(((i, b), r)}_{((i, a), r-1)}$ or $q^{(((i, b), r)}_{((i, a), r-1)}\sim q_{(((i, b), r)}^{((i, a), r-1)}$ by
the complex $W^\bu_{(i, (a, b))}[-r]$, written in a separate line.
\item Analogously, any subword of $\overline{v}$ of the form $q_{(((i, b), r)}^{(i, m_i+1), r-1)}$ is replaced by the complex $W^\bu_{(i, (m_i+1, b))}[-r]$, written in a separate line.
\item For each subword of $\overline{v}$ of the form
$$
q_{((i, j), r)}^{((i, \bar{j}), r \pm 1)} -  q_{((k, l), r)}^{((k, \bar{l}), r \pm 1)}
$$
for some $1 \le j \ne \bar{j} \le m_i+1$ and $1 \le l \ne \bar{l} \le m_{k}+1$, we connect  the corresponding terms  $Q_{(i, j)}$ and $Q_{(k, l)}$ by a dotted line.
\end{itemize}

\smallskip
\noindent
2.~Let
$\overline{w} := \lha x_1 \sim x_2 - \dots - x_{2n-1} \sim x_{2n}  \rha
$
be  a non--periodic cyclic word,
then $x_1 = q_{((i, j), r)}^{((i, \bar{j}), r \pm 1)}$ and $x_{2n} = q_{((k, l), r)}^{((k, \bar{l}), r \pm 1)}$ for some $r \in \ZZ$,
$(i, j) \simeq (k, l)$ in $\Omega$ as well as  appropriate $1 \le j \ne \bar{j} \le m_i$ and $1 \le l \ne \bar{l} \le m_{k}$.
\begin{itemize}
\item In the first step, we  apply the procedure described in the first part to the word $x_1 \sim x_2 - \dots - x_{2n-1} \sim x_{2n}$, viewed as a string parameter.
\item Let $x_1 = q_{((i, j), r)}^{((i, \bar{j}), r \pm 1)}$ for some $r \in \ZZ$ and $1 \le j \ne \bar{j} \le m_i$. Then we have:
$x_{2n} = q_{((k, l), r)}^{((k, \bar{l}), r\pm 1)}$, where $(i, j) \simeq (k, l)$ in $\Omega$ and $1 \le l \ne \bar{l} \le m_k$.
As the next step,   we connect
the terms  $Q_{(i, j)}^{\oplus m}$ and $Q_{(k, l)}^{\oplus m}$ of the corresponding indecomposable complexes by a dotted line.
\item We replace each complex $W^\bu_{(e, (a, b))}[-r]$ appearing in the gluing diagram by the complex $\bigl(W^\bu_{(e, (a, b))}[-r]\bigr)^{\oplus m}$.
\item Finally, we  replace the differential $\kappa I_m$ of the bottom complex (arising from the subword $x_{2n-1} \sim x_{2n}$) by the differential $\kappa J_m(\pi)$.
\end{itemize}
\end{definition}

\begin{example}\label{Ex:Myfavorite2} Let $A$ be the gentle algebra from Example \ref{Ex:MyFavorite}.

\smallskip
\noindent
1.~Consider the band datum $(\underline{w}, m, \pi)$, where $\underline{w} = $
\begin{align*}
&  \lha q_{((2, 3), -2)}^{((2,2),-1)} \sim q^{((2, 3), -2)}_{((2,2),-1)} - u_{((2,2),-1)} \sim u_{((1,2),-1)}  \\
&       -  q_{((1, 2), -1)}^{((1,1),0)} \sim  q^{((1, 2), -1)}_{((1,1),0)}  - u_{((1,1),0)} \sim u_{((2,1),0)}  \\
&     -  q_{((2, 1), 0)}^{((2,2),-1)} \sim  q^{((2, 1), 0)}_{((2,2),-1)} - u_{((2,2),-1)} \sim u_{((2,1),-1)}  \\
&     - q_{((1, 2), -1)}^{((1,3),-2)} \sim q^{((1, 2), -1)}_{((1,3),-2)}  - u_{((1,3),-2)} \sim u_{((2,3),-2)} \rha  \\
\end{align*}
Without loss of information, we may replace $\underline{w}$ by its reduced version $\overline{w} = $
$$
 \lha q_{((2, 3), -2)}^{((2,2),-1)} \sim q^{((2, 3), -2)}_{((2,2),-1)}
       -  q_{((1, 2), -1)}^{((1,1),0)} \sim  q^{((1, 2), -1)}_{((1,1),0)}  -
       q_{((2, 1), 0)}^{((2,2),-1)} \sim  q^{((2, 1), 0)}_{((2,2),-1)}     - q_{((1, 2), -1)}^{((1,3),-2)} \sim q^{((1, 2), -1)}_{((1,3),-2)}   \rha  \\
$$
Then we get the following  gluing diagram:
\begin{equation}\label{E:bandExGluing}
\begin{array}{c}
\xymatrix{
Q_{(2, 3)}^{\oplus m} \ar[r]^-{d I}  \ar@{.}[ddd] & Q_{(2, 3)}^{\oplus m} \ar@{.}[d]&  \\
& Q_{(1, 2)}^{\oplus m} \ar[r]^-{a I} & Q_{(1, 1)}^{\oplus m} \ar@{.}[d]\\
& Q_{(2, 2)}^{\oplus m} \ar[r]^-{c I} \ar@{.}[d] & Q_{(2, 1)}^{\oplus m} \\
Q_{(1, 3)}^{\oplus m} \ar[r]^-{b J} & Q_{(1, 2)}^{\oplus m} &  \\
}
\end{array}
\end{equation}
Here, $I$ is the identity matrix of size $m$, whereas $J$ is the Jordan block of size $m$ with eigenvalue $\pi$.

\smallskip
\noindent
2.~Consider the string datum $\underline{v}$ given by the full word
\begin{align*}
&  \quad \quad \quad \quad \;\; \;  q^{((2, 3), -1)}_{((2,4),-2)} - \; u_{((2,3),-1)} \sim  u_{((2,3),-1)} - \\
&  q_{((1, 3), -1)}^{((1,1),0)} \sim  \, q^{((1, 3), -1)}_{((1,1),0)} - u_{((1,1),0)} \sim \; \; u_{((2,1),0)} \; - \\
&  q_{((2, 1), 0)}^{((2,2),-1)} \sim  q^{((2, 1), 0)}_{((2,2),-1)}  - u_{((2,2),-1)} \sim u_{((1,2),-1)} - \\
&  q_{((1, 2), -1)}^{((1,3),-2)} \sim  q^{((1, 2), -1)}_{((1,3),-2)}  - u_{((1,3),-2)} \sim u_{((2,3),-2)} - \\
& q_{((2, 3),-2)}^{((2,4),-3)}
\end{align*}
The corresponding reduced version $\overline{v}$ is
$$
q^{((2, 3), -1)}_{((2,4),-2)} -
  q_{((1, 3), -1)}^{((1,1),0)} \sim  \, q^{((1, 3), -1)}_{((1,1),0)} -
  q_{((2, 1), 0)}^{((2,2),-1)} \sim  q^{((2, 1), 0)}_{((2,2),-1)}  -
  q_{((1, 2), -1)}^{((1,3),-2)} \sim  q^{((1, 2), -1)}_{((1,3),-2)}  -
q_{((2, 3),-2)}^{((2,4),-3)}
$$
It defines the following gluing diagram:
\begin{equation}\label{E:stringExGluing}
\begin{array}{c}
\xymatrix{
 & Q_{(2, 3)}  \ar@{.}[d]           &                       \\
 & Q_{(1, 3)} \ar[r]^{ba} & Q_{(1, 1)} \ar@{.}[d] \\
 & Q_{(2, 2)} \ar[r]^c \ar@{.}[d] & Q_{(2, 1)}  \\
 Q_{(1, 3)} \ar[r]^b \ar@{.}[d] & Q_{(1, 2)} & \\
 Q_{(2, 3)}
}
\end{array}
\end{equation}
\end{example}

\begin{definition}\label{D:FromGluingToComplexes} Let $(\overline{w}, m, \pi)$ be a band datum and $\overline{v}$ be a string datum as in Definition \ref{D:BandStringsReduced}.
We define the corresponding band complexes $B^\bu(\overline{w}, m ,\pi)$ and string complexes $S^\bu(\overline{v})$ in the homotopy category
$\mathsf{Hot}^b(A-\mathsf{pro})$  by the following procedure.
\begin{itemize}
\item We construct gluing diagrams corresponding to these band or string data.
\item Any time we have projective $H$--modules $Q_{(i, j)}$ and $Q_{(k, l)}$ connected by a dotted arrow, we merge them into the projective
$A$--module $P_\gamma$, where $$\gamma = \overline{\{(i, j), (k, l)\}} \in \widetilde{\Omega}.$$
\item Any projective $H$--module $Q_{(i, j)}$ which is not linked by a dotted arrow with another projective module, has to be replaced by
$P_{\gamma}$, where $\gamma = (i, j) \in \widetilde{\Omega}$
 (in this case, $\gamma$ is automatically an element of the third type).
 \item Taking the direct sums of all projective $A$--modules lying in the same degree, and $A$--linear maps, inherited from the gluing diagram, we get
 a sequence a projective modules and $A$--linear homomorphisms between them.
 \end{itemize}
\end{definition}

\begin{remark} Let $(\overline{w}, m, \pi)$ (respectively $\overline{v}$) be a band (respectively, string) datum. Then for $\overline{u} \in \left\{ \overline{w}, \overline{v}\right\}$, we naturally get a minimal complex of projective $H$--modules
$$
Y^\bu = Y^\bu(\overline{u}) = \left(\dots \lar Y^{r-1} \stackrel{d^{r-1}}\lar Y^r \stackrel{d^r}\lar Y^{r+1} \lar \dots \right)
$$
obtained by taking the direct sum of all indecomposable complexes occurring in the corresponding gluing diagram.
Let $X^\bu$ be the sequence
of projective $A$--modules
and $A$--linear homomorphisms, constructed in Definition \ref{D:FromGluingToComplexes}. Then we have the following commutative diagram in the category of $A$--modules:
$$
\xymatrix{
\dots \ar[r] & X^{r-1} \ar[r]^-{d^{r-1}} \ar@{^{(}->}[d] &  X^r \ar[r]^-{d^r} \ar@{^{(}->}[d] & X^{r+1} \ar[r] \ar@{^{(}->}[d]& \dots \\
\dots \ar[r] & Y^{r-1} \ar[r]^-{d^{r-1}} &  Y^r \ar[r]^-{d^r} & Y^{r+1} \ar[r] & \dots
}
$$
Since all vertical maps in this diagram are embeddings, $X^\bu$ is indeed a complex. Moreover, the embedding $X^\bu \hookrightarrow Y^\bu$ induced an isomorphism $H \otimes_A X^\bu \lar Y^\bu$.
\end{remark}

\begin{example}\label{Ex:Myfavorite3} Let $A$ be the gentle algebra from Example \ref{Ex:MyFavorite}.  Let us construct the band complex
and the string complex, corresponding to the gluing diagrams from Example \ref{Ex:Myfavorite2}.

\smallskip
\noindent
1.~The gluing diagram (\ref{E:bandExGluing}) defines the following complex of projective $A$--modules
$$
\xymatrix{
 & P_2^{\oplus m} \ar[rd]^-{aI} & \\
P_3^{\oplus m} \ar[ru]^-{dI} \ar[rd]_-{bJ} & \bigoplus & P_1^{\oplus m} \\
& P_2^{\oplus m} \ar[ru]_-{cI} &
}
$$
or, in the conventional terms, the complex
\begin{equation}
\dots \lar 0 \lar P_3^{\oplus m} \xrightarrow{\left(\begin{smallmatrix}dI\\ bJ \end{smallmatrix}\right)} P_2^{\oplus 2m} \xrightarrow{\left(\begin{smallmatrix}aI &  cI \end{smallmatrix}\right)} \underline{P_1^{\oplus m}} \lar 0 \lar \dots,
\end{equation}
where the underlined component   is located in the zero degree.

\smallskip
\noindent
2.~The gluing diagram (\ref{E:stringExGluing}) defines the following complex of projective $A$--modules
$$
\xymatrix{
 & P_3  \ar[rd]^-{ba} & \\
 & \bigoplus & P_1 \\
 & P_2  \ar[ru]^-{c} & \\
 P_3 \ar[ru]^-{b} & &
}
$$
or, in the conventional notation, the complex
\begin{equation}
\dots \lar 0 \lar P_3  \xrightarrow{\left(\begin{smallmatrix}0 \\ b \end{smallmatrix}\right)} P_3 \oplus P_2 \xrightarrow{\left(\begin{smallmatrix}ba &  c \end{smallmatrix}\right)} \underline{P_1} \lar 0 \lar \dots
\end{equation}
where the underlined term is again located in the zero degree.
\end{example}

\begin{theorem}\label{T:IndecGentle} Let $A = A(\vec{\sm}, \simeq)$ be a gentle algebra. Then the following results are true.
\begin{itemize}
\item Let  $X^\bu$ be an indecomposable object of $\mathsf{Hot}^b(A-\mathsf{pro})$. Then $X^\bu$ is isomorphic either to some  band complex
$B^\bu(\overline{w}, m, \pi)$ or to some string complex $S^\bu(\overline{v})$ (see Definition \ref{D:BandStringsReduced} and Definition \ref{D:FromGluingToComplexes}).
\item $B^\bu(\overline{w}, m, \pi) \not\cong S^\bu(\overline{v})$ for any band datum $(\overline{w}, m, \pi)$ and string datum $\overline{v}$.
\item We have: $S^\bu(\overline{v}) \cong S^\bu(\overline{v}')$ if and only if $\overline{v}' = \overline{v}$ or $\overline{v}' = \overline{v}^\circ$.
\item Finally, $B^\bu(\overline{w}, m, \pi) \cong B^\bu(\overline{w}', m', \pi)$ if and only if $m = m'$ and
$$
(\overline{w}', \pi') = \left\{
\begin{array}{l}
(\overline{w}^{(k)}, \pi) \\
({\overline{w}^{\circ}}^{(k)}, \pi^{-1})
\end{array}
\right.
$$
for some shift $k \in \ZZ$.
\end{itemize}
\end{theorem}
\begin{proof} For a band datum  $(\overline{w}, m, \pi)$, let $B(\overline{w}, m, \pi)$ be the corresponding indecomposable object of $\Rep(\dX)$. Let
 $\bigl(Y^\bu(\overline{w}), V^\bu(\overline{w}), \theta(\overline{w}, m, \pi)\bigr)$ be the corresponding object of $\Tri^b(A)$
and $X^\bu = B^\bu(\overline{w}, m, \pi)$ be the corresponding complex, constructed in Definition \ref{D:FromGluingToComplexes}.
It follows from the construction that  we have isomorphisms $H\otimes_A X^\bu \stackrel{g}\lar Y^\bu(\overline{w})$ and $\bar{A}\otimes_A X^\bu \stackrel{h}\lar V^\bu(\overline{w})$,
making the following diagram commutative:
$$
\xymatrix{
\bar{H} \otimes_{\bar{A}} \bigl(\bar{A}\otimes_A X^\bu\bigr) \ar[d]_-{\mathsf{id} \otimes h}\ar[r]^{\theta_{X^\bu}} & \bar{H} \otimes_{H} \bigl(H\otimes_A X^\bu\bigr) \ar[d]^-{\mathsf{id} \otimes g} \\
\bar{H} \otimes_{\bar{A}} V^\bu(\overline{w}) \ar[r]^-{\theta(\overline{w}, m, \pi)} & \bar{H} \otimes_{H} Y^\bu(\overline{w})
}
$$
A similar statement is true for a string datum $\overline{v}$.
In other words,  we have:
$$
(\MM \circ \EE) \bigl(B^\bu(\overline{w}, m, \pi)\bigr) \cong B(\overline{w}, m, \pi) \quad \mbox{\rm and} \quad (\MM \circ \EE) \bigl(S^\bu(\overline{v})\bigr) \cong S(\overline{v})
$$
(here, we follow the notation of diagram  (\ref{E:Summary})).
The result is a consequence of Theorem \ref{T:stringsandbands}, Theorem \ref{T:BurbanDrozdTriples} and Lemma \ref{L:reductionOne}.
\end{proof}

\subsection{Unbounded homotopy categories} In the case $\mathsf{gl.dim}(A) < \infty$, Theorem \ref{T:IndecGentle} gives a complete classification of  indecomposable objects of the derived category
$D^b(A-\mathsf{mod})$. However, the case $\mathsf{gl.dim}(A) = \infty$ requires some additional  efforts.

\begin{theorem}\label{T:UnboundedIndecom} Let $A$ be a gentle algebra. Then the indecomposable objects of the unbounded homotopy category $\mathsf{Hot}(A-\mathsf{pro})$ are the following ones.

\smallskip
\noindent
1.~Band complexes $B^\bu(\overline{w}, m, \pi)$ and string complexes  $S^\bu(\overline{v})$ as in Theorem \ref{T:IndecGentle} (i.e.~the indecomposable objects of $\mathsf{Hot}^b(A-\mathsf{pro})$).

\smallskip
\noindent
2.~Infinite string complexes $S^\bu(\overline{u})$, where $\overline{u}$ is either

\begin{itemize}
\item a semi--infinite word  of the form
$x_1 \sim x_2 - x_3 \sim x_4 - \dots,$ where $x_1 = q_{((i, k), r)}^{((i, \bar{k}), r \pm 1)}$ for some $r \in \ZZ$, $(i, k), (i, \bar{k}) \in \Omega$ is such that  $(i, k) \in \Omega$ is untied
or of the form $y_0 - y_1 \sim y_2 - \dots,$ where $y_0 = q_{((i, k), r)}^{((i, m_i+1), r - 1)}$ for some $((i, k), r) \in \Omega$;

\item
or an infinite word $\dots -z_{-2} \sim z_{-1}- z_1 \sim z_{2} - \dots$.
\end{itemize}
In all these cases, the word $\overline{u}$ has to satisfy the following constraint: for any $r \in \ZZ$,  $\overline{u}$ contains only finitely many
elements of the set $\dF_{r}$. The infinite complex $S^\bu(\overline{u})$ is defined by the same rules as the finite ones in the case of
 $\mathsf{Hot}^b(A-\mathsf{pro})$; see Definition \ref{D:gluingdiagrams} and Definition \ref{D:FromGluingToComplexes}.

\smallskip
\noindent
Next, an infinite string complexes is never  isomorphic to a band complex or a finite string complex. Moreover,
\begin{itemize}
\item Let $\overline{u}$ be a semi--infinite string parameter and $\overline{u}'$ an infinite string parameter. Then we have: $S^\bu(\overline{u}) \not\cong S^\bu(\overline{u}')$.
\item Let $\overline{u}$  and $\overline{u}'$ be semi--infinite string parameters. Then $S^\bu(\overline{u}) \cong S^\bu(\overline{u}')$ if and only if
$\overline{u} = \overline{u}'$.
\item Let $\overline{u}$  and $\overline{u}'$ be infinite string parameters. Then $S^\bu(\overline{u}) \cong S^\bu(\overline{u}')$ if and only if
$\overline{u}' = \overline{u}$ or $\overline{u}' = \overline{u}^\circ$, where $\overline{u}^\circ$ is the opposite word to $\overline{u}$.
\end{itemize}

\smallskip
\noindent
Finally, an infinite string complex $S^\bu(\overline{u})$ belongs to $\mathsf{Hot}^{-}(A-\mathsf{pro})$ if and only if there exists $r \in \ZZ$ such that
$\overline{u}$ does not contain any elements of the set $\dF_{i}$ for $i \ge r$. A similar condition characterizes infinite string complexes, which belong to
$\mathsf{Hot}^{+}(A-\mathsf{pro})$.
\end{theorem}

\begin{proof}
As it was explained in Subsection \ref{SS:TriplesMP}, also in the unbounded cases, we can reduce the description of indecomposable objects
of the category of triples $\Tri^\ast(A)$ to a matrix problem. Namely, we have a full functor $\Tri(A) \stackrel{\MM}\lar \Rep^\infty(\dX)$ (respectively,
$\Tri^\pm(A) \stackrel{\MM}\lar \Rep^\pm(\dX)$), reflecting isomorphism classes and indecomposability of objects, where $\Rep^\infty(\dX)$ (respectively,
$\Rep^\pm(\dX)$) is the category of locally finite dimensional representations of the bunch of chains $\dX$. The indecomposable objects of $\Rep^\infty(\dX)$
can be classified in the same way as for $\Rep(\dX)$ (the only difference between $\Rep^\infty(\dX)$ and $\Rep(\dX)$ in that we may have infinitely many
non--zero matrices in the matrix problem (\ref{E:transfrule})). An interested reader might consult for \cite[Subsection 7.6]{BurbanDrozdMemoirs} for an overview. The key point is that the reduction procedure, described in \cite[Subsection 12.6]{BurbanDrozdMemoirs} can be  applied to the infinite case literally. The outcome is that the only indecomposable objects of $\Rep^\infty(\dX)$ which do not belong to $\Rep(\dX)$, are the infinite string representations given by an infinite or semi--infinite parameter as in the statement of the theorem. See also \cite[Theorem C.5]{Nodal} and \cite{BondarenkoBlaBla}.
\end{proof}

\begin{remark}
Note the the Krull--Remak--Schmidt property is also true in the unbounded homotopy category $\mathsf{Hot}(A-\mathsf{pro})$; see \cite[Proposition A2]{Nodal}.
\end{remark}

\begin{example} Let $A$ be the gentle algebra from Example \ref{Ex:MyFavorite}. Consider the  semi--infinite string $\overline{u}$ corresponding to  the following gluing diagram (going up by periodicity):
$$
\xymatrix{
\dots  \ar@{.}[d]           &  & & \\
                    Q_{(1, 2)} \ar[r]^-{a}               & Q_{(1, 1)} \ar@{.}[d] & & \\
Q_{(2, 3)} \ar[r]^-{dc} \ar@{.}[d] & Q_{(2, 1)}  & & \\
Q_{(1, 3)} \ar[r]^-{b} & Q_{(1, 2)} \ar@{.}[d] & &    \\
                   & Q_{(2, 2)} \ar[r]^-{c} & Q_{(2, 1)} \ar@{.}[d] & \\
                   & Q_{(1, 3)} \ar[r]^-{ba} \ar@{.}[d] & Q_{(1, 1)}  & \\
 & Q_{(2, 3)} \ar[r]^-{d}  & Q_{(2, 2)} \ar@{.}[d]  & \\
 & & Q_{(1, 2)} \ar[r]^-{a}  & Q_{(1, 1)} \ar@{.}[d]   \\
 & & & Q_{(2,1)}
}
$$
This diagram defines  the following infinite string complex  $S^\bu(\overline{u})$  of the homotopy category $\mathsf{Hot}^-(A-\mathsf{pro})$:
$$
\dots \lar
P_1 \oplus P_2 \oplus P_3 \xrightarrow{\left(\begin{smallmatrix} 0 & c & ba \\ 0 & 0 & d \\ 0 & 0 & 0\end{smallmatrix}\right)}
P_1 \oplus P_2 \oplus P_3 \xrightarrow{\left(\begin{smallmatrix} 0 & a & dc \\ 0 & 0 & b \\ 0 & 0 & 0\end{smallmatrix}\right)} P_1 \oplus P_2 \oplus P_3
\xrightarrow{\left(\begin{smallmatrix} 0 & c & ba \\ 0 & 0 & d \end{smallmatrix}\right)} P_1 \oplus P_2 \xrightarrow{\left(\begin{smallmatrix} 0 & a &
\end{smallmatrix}\right)} P_1.
$$
In the pictorial form, $S^\bu(\overline{u})$ can be represented as follows:
\begin{equation*}
\xymatrix{
&        \dots                                                        & & & &  \\
P_3 \ar[ru]^-{ba} \ar[r]^-{d} & P_2 \ar[r]^-{a}                      & P_1 & &  & \\
& P_3 \ar[ru]^-{dc} \ar[r]^-{b} & P_2 \ar[r]^-{c} & P_1 & \\
&                             &P_3 \ar[ru]^-{ba} \ar[r]^-{d} & P_2 \ar[r]^-{a} & P_1
}
\end{equation*}
\end{example}

\smallskip
\noindent
Our next goal is to describe indecomposable objects of the derived category $D^b(A-\mathsf{mod})$ in the case $\mathsf{gl.dim}(A) = \infty$.
Of course, we have a fully faithful functor $D^b(A-\mathsf{mod}) \lar \mathsf{Hot}^-(A-\mathsf{pro})$, assigning to a complex its projective resolution.
So, our goal is to describe the essential image of this functor.
We begin with the following observation.

\begin{lemma} Let $(Y^\bu, V^\bu, \theta)$ be an object of (possibly unbounded)  category of triples $\Tri^\ast(A)$ (as usual, we assume that
both complexes  $Y^\bu$ and $V^\bu$ are minimal) and $X^\bu$ be the corresponding
object of the homotopy category $\mathsf{Hot}^\ast(A-\mathsf{pro})$. For any $r \in \ZZ$, consider
the morphism of $A$--modules $H^r(V^\bu) \stackrel{\psi^r}\lar H^{r+1}(I Y^\bu)$, defined as the composition
$$
H^r(V^\bu) \stackrel{\widetilde{\theta}^r}\lar H^{r}(\bar{Y}^\bu) \stackrel{\delta^r}\lar H^{r+1}(I Y^\bu),
$$
where $\delta^r$ is the boundary map, associated with the short exact sequence of complexes
$
0 \lar IY^\bu \lar Y^\bu \lar \bar{Y}^\bu \lar 0.
$
Then the following statement are true.
\begin{itemize}
\item Assume that all maps $\psi^i$ are isomorphism for any $i \le r$. Then we have: $H^i(X^\bu) = 0$ for all $i \le r$.
\item Conversely, assume that $H^i(X^\bu) = 0$ for all $i \le r$. Then we have: $\psi^i$ is an isomorphism for all $i < r$.
\end{itemize}
\end{lemma}

\begin{proof}
According to (\ref{E:reconstr}), we have a short exact sequence
\begin{equation}\label{E:shortExact}
0 \lar IY^\bu \lar X^\bu \lar \bar{Y}^\bu \lar 0
\end{equation}
in the category of complexes $\mathsf{Com}^\ast(A-\mathsf{mod})$. It remains  to observe that the morphism $\psi^r$ can be identified with the $r$-th boundary map  in the long exact  cohomology sequence of the short exact sequence (\ref{E:shortExact}).
\end{proof}

\smallskip
\noindent
It is clear that the map  $\psi^r$ is an isomorphism only if the connecting homomorphism $H^{r}(\bar{Y}^\bu) \stackrel{\delta^r}\lar H^{r+1}(I Y^\bu)$
is an epimorphism. By naturality of $\psi^r$, this  is the case if and only if it is so  for each indecomposable direct summand of $Y^\bu$.

\begin{lemma}\label{L:NormalizationComplexCohom}
Let $1 \le i \le t$, $1 \le b < a \le m_i+1$, $l \in \ZZ$ and $Y^\bu := W^\bu_{(i, (a, b))}[-l]$. Then the following statements are true.
\begin{itemize}
\item $H^{i}(I Y^\bu) \ne 0$ if and only if $i = l$.
\item The boundary map  $\delta^{l-1}$ is surjective if and only if $b = a+1 \le m_i$.
\end{itemize}
\end{lemma}

\begin{proof}
Recall that $I$ is the radical of the algebra $H$. Therefore, $I Q_{(i, j)} = \mathsf{rad}\bigl(Q_{(i, j)}\bigr) = Q_{(i, j+1)},$ where  $Q_{(i, m_i+1)} = 0$. As a consequence,  $I W^\bu_{(i, (a, b))} = W^\bu_{(i, (a+1, b+1))}$, implying the first statement.
Next, the connecting homomorphism $\delta^{l-1}$ can be identified with a map
$$
\bar{Q}_{(i, a)} \stackrel{\widetilde{\delta}_{(a, b)}}\lar Q_{(i, b+1)}/Q_{(i, a+1)} = W_{i, (a+1, b+1))}.
$$
The module $\bar{Q}_{(i, a)}$ is either a simple module (if $a \le m_i$) or zero (if $a = m_i +1$). The module $W_{(i, (a+1, b+1))}$ is has length $b-a$.
Therefore, the morphism $\widetilde{\delta}_{(a, b)}$ can be surjective only if $a = b +1 \le m_i$. A straightforward computation shows, that in this case,
$\widetilde{\delta}_{(a, b)}$ is even an isomorphism.
\end{proof}

\begin{definition} Let $(\vec\sm, \simeq)$ be a datum defining a gentle algebra $A$ and $\Omega = \Omega(\vec\sm)$. For an auxiliary  symbol $\dagger$, we define the map $\Omega \stackrel{\tau}\lar \Omega \cup \{\ddagger\}$ by the following rules. For any  $(i, j) \in \Omega$ we put:
\begin{itemize}
\item $\tau(i, j) = (k, l)$, if $j < m_i$ and there exists $(k, l) \in \Omega$ such that $(i, j+1) \simeq (k, l)$.
\item $\tau(i, j) = \ddagger$ otherwise.
\end{itemize}
A \emph{special cycle} in $\Omega$ is a (non--periodic) sequence of elements $(i_1, j_1), \dots, (i_k, j_k)$ in $\Omega$ such that
$(i_{r+1}, j_{r+1}) = \tau(i_r, j_r)$ for any $r \in \NN$, where $(i_{k+1}, j_{k+1}) := (i_1, j_1)$.
\end{definition}

\begin{remark}
Let $A = \kk[\stackrel{\rightarrow}{Q}]/\langle L\rangle$ be a presentation of a gentle algebra $A$ in terms of the   path algebra of a quiver with relations. It follows, that a special cycle
$(i_1, j_1), \dots, (i_k, j_k) \in \Omega$ defines an oriented cycle  $\varpi = \varpi_k \dots \varpi_1$ in the path algebra $\kk[\stackrel{\rightarrow}{Q}]$:
\begin{equation}\label{E:specialcycle}
(i_1, j_1) \stackrel{\varpi_1}\lar (i_1, j_1 +1) = (i_2, j_2) \stackrel{\varpi_2}\lar (i_2, j_2+1) \lar \dots \stackrel{\varpi_k}\lar (i_1, j_1)
\end{equation}
such that $\varpi_{r+1} \varpi_r = 0$ in $A$ for all $1 \le r \le k$. Conversely, it is not difficult to see, that any such path $\varpi$ in the path algebra
$\kk[\stackrel{\rightarrow}{Q}]$ as in (\ref{E:specialcycle})
defines
a special cycle in $\Omega$.
\end{remark}

\smallskip
\noindent
 For any special cycle  $\varpi:= (i_1, j_1), \dots, (i_k, j_k)$  in $\Omega$, consider the following sequence $\widetilde{u}(\varpi) = $
 \begin{align*}
& q_{((i_1, j_1), 0)}^{((i_1, j_1 +1), -1)} \sim q^{((i_1, j_1), 0)}_{((i_1, j_1 +1), -1)}- q_{((i_2, j_2), -1)}^{((i_2, j_2 +1), -2)} \sim q_{((i_2, j_2), -2)}^{((i_2, j_2 +1), -1)} - \dots - \\
& q_{((i_1, j_1), -k)}^{((i_1, j_1 +1), -k-1)} \sim q^{((i_1, j_1), -k)}_{((i_1, j_1 +1), -k-1)} -\dots
\end{align*}

\begin{lemma}\label{L:resolutiongentle} Let $\varpi:= (i_1, j_1), \dots, (i_k, j_k)$ be a special cycle in $\Omega$. Then
$$
\overline{u}(\varpi) := q_{((i_k, j_k), 0)}^{((i_k, m_k +1), -1)} - \widetilde{u}(\varpi)
$$
is a semi--infinite string parameter in the sense of Theorem \ref{T:UnboundedIndecom} and  the corresponding string complex $S^\bu(\overline{u}\bigl(\varphi)\bigr)$ is a minimal projective resolution
of an $A$--module.
\end{lemma}

\begin{proof}
For any $1 \le l \le k$, let $\gamma_l$ be the vertex  of the quiver $\stackrel{\rightarrow}{Q}$ corresponding to the element $(i_l, j_l) \in \Omega$.
Then $S^\bu(\overline{u}\bigl(\varphi)\bigr)$ is the complex
$$
   \dots  \lar  P_{\gamma_1} \stackrel{\varpi_k}\lar P_{\gamma_k} \stackrel{\varpi_{k-1}}\lar   \dots \stackrel{\varpi_2}\lar  P_{\gamma_2} \stackrel{\varpi_1}\lar  P_{\gamma_1} \lar 0.
$$
It follows from the definition of gentle algebra that this complex is acyclic apart of the zero degree.
\end{proof}

\begin{theorem}\label{T:boundedInfGlobDim} Let $A = A(\vec\sm, \simeq)$ be a gentle algebra. Then the following results are true.

\smallskip
\noindent
1.~$\mathsf{gl.dim}(A) = \infty$  if and only if there exists a special cycle in $\Omega$.

\smallskip
\noindent
2.~Assume that $\mathsf{gl.dim}(A) = \infty$. Then the only indecomposable objects of the homotopy category $\mathsf{Hot}^-(A-\mathsf{pro})$ with bounded cohomology, which do not belong to $\mathsf{Hot}^-(A-\mathsf{pro})$,   are
the infinite string complexes $S^\bu(\overline{u})$ satisfying the following additional conditions:
\begin{itemize}
\item if $\overline{u}$ is a semi--infinite word than after a finite part, it is equal to a shift of the sequence $\widetilde{u}(\varpi)$, attached
to a special cycle $\varpi$ in $\Omega$.
\item Similarly, if $\overline{u}$ is an infinite word than asymptotically at $\pm \infty$ it is given by the sequences $\widetilde{u}(\varpi_1)$
and $\widetilde{u}(\varpi_2^\circ)$ attached to some special cycles $\varpi_1$ and $\varpi_2$ in $\Omega$.
\end{itemize}
\end{theorem}

\begin{proof} The stated criterion  for a gentle algebra to have infinite global dimension is not new. It follows for instance from Kalck's description
\cite{Kalck} of the singularity category $D_{sg}(A) = D^b(A-\mathsf{mod})/\mathsf{Hot}^b(A-\mathsf{pro})$ of a gentle algebra $A$. We give now another proof, which also gives a new insight on the description of the category $D_{sg}(A)$.

Assume that  $\Omega$ admits a special cycle.  Then  $\mathsf{gl.dim}(A) = \infty$ by Lemma \ref{L:resolutiongentle}.
Conversely, assume that $\mathsf{gl.dim}(A) = \infty$. Then there exists an infinite indecomposable complex in $\mathsf{Hot}^-(A-\mathsf{pro})$ with bounded
cohomology. According to Theorem \ref{T:UnboundedIndecom}, it is some infinite string complex $S^\bu(\overline{u})$, attached to some infinite or semi--infinite string parameter $\overline{u}$.

\smallskip
\noindent
We say that the gluing diagram of $\overline{u}$ has a \emph{right turn} if it contains a subdiagram of the form
$$
\xymatrix{
Q_{(j, c)} \ar[r]^-{\kappa''} \ar@{.}[d]& Q_{(j, d)} \\
Q_{(i, a)} \ar[r]^-{\kappa'} & Q_{(i, b)}
}
$$
Similarly, we say that the gluing diagram $\overline{u}$ has a \emph{left turn} if it contains a subdiagram of the form
$$
\xymatrix{
Q_{(j, c)} \ar[r]^-{\kappa''} & Q_{(j, d)} \ar@{.}[d]\\
Q_{(i, a)} \ar[r]^-{\kappa'} & Q_{(i, b)}
}
$$
In the latter case, let $\gamma = \{(i, b), (j, d)\} \in \widetilde{\Omega}$, whereas $\alpha, \beta \in \widetilde{\Omega}$ be the classes containing
the elements $(i,a)$ and $(j, c)$, respectively. Since the morphism of projective $A$--modules
$$
P_\alpha \oplus P_\beta \xrightarrow{(\kappa', \kappa'')} P_\gamma
$$
is not surjective (the image of this map belongs to the radical of $P_\gamma$), each left turn contributes to cohomology of the complex
$S^\bu(\overline{u})$. As a consequence, $\overline{u}$ contains only finitely many left turns. On the other hand, $\overline{u}$ contains only finitely many right turns, too. Indeed, since $S^\bu(\overline{u})$ belongs to $\mathsf{Hot}^-(A-\mathsf{pro})$, after each right turn should come a left turn.
Now, assume $\overline{u} = (x_0) - x_1\sim x_2 - \dots - x_{2k-1} \sim x_{2k} - \dots $ is a semi--infinite word. Since $S^\bu(\overline{u})$ has bounded cohomology, Lemma \ref{L:NormalizationComplexCohom} implies that there exists some $l \in \NN$ such that for any $k \ge l$, each subsequence
$x_{2k-1} \sim x_{2k}$ of $\overline{u}$ has the form $q_{((i,j), r)}^{(i, j+1), r-1)} \sim q^{((i,j), r)}_{(i, j+1), r-1)}$ for some
$r \in \ZZ$, $1 \le i \le t$ and $1 \le j \le m_i-1$. Moreover, we may assume that for all $k \ge l$, the subword
$x_{2k-1} \sim x_{2k} - x_{2k+1} \sim x_{2k+2}$ is not a right turn. There exists only one possibility for that:
$
\left(x_{2k+1} \sim x_{2k+2}\right) = \left(q_{((\bar{i},\bar{j}), r-1)}^{(\bar{i}, \bar{j}+1), r-1)} \sim q_{((\bar{i},\bar{j}), r-1)}^{(\bar{i}, \bar{j}+1), r-1)}\right),
$
where $(i, j+1) \sim (\bar{i}, \bar{j})$ in $\Omega$. Proceeding inductively and taking into account that $\Omega$ is a finite set, we get
a special cycle $\varpi = (i, j), (\bar{i}, \bar{j})$. The case of an infinite string parameter $\overline{u}$ is analogous.
\end{proof}

\section{Matrix  problem associated with a semi--orthogonal decomposition}\label{S:MPSemiorthDecomp}

\smallskip
\noindent
Let $\catD$ be a triangulated category and $\catD = \langle \catE, \catF\rangle$ be a semi--orthogonal decomposition. Recall that this means that
\begin{itemize}
\item Both $\catE$ and $\catF$ are full triangulated subcategories of $\catD$.
\item For any $E \in \mathsf{Ob}(\catE)$ and $F \in \mathsf{Ob}(\catF)$ we have: $\Hom_{\catD}(F, E) = 0$.
\item For any $D \in \mathsf{Ob}(\catD)$, there exists a distinguished triangle
\begin{equation}\label{E:Decomposition}
E \stackrel{w}\lar F \stackrel{u}\lar D \stackrel{v}\lar E[1]
\end{equation}
for some $E = E_D \in \mathsf{Ob}(\catE), F = F_D\in \mathsf{Ob}(\catF)$ and $w \in \Hom_{\catD}(E, F)$.
\end{itemize}

\begin{definition}\label{D:CommaCat} Let $\catD$ be a triangulated category and $\catD = \langle \catE, \catF\rangle$ a semi--orthogonal decomposition.
Consider the \emph{comma category} $\mathsf{Comma}(\catE, \catF)$ attached to the pair of category embeddings
$
\catE \lar \catD \longleftarrow \catF.
$
In other words, objects of $\mathsf{Comma}(\catE, \catF)$ are triples $(E, F, w) = \bigl(E \stackrel{w}\lar F\bigr)$, where $E \in \mathsf{Ob}(\catE)$, $F \in \mathsf{Ob}(\catF)$
and $w \in \Hom_{\catD}(E, F)$, whereas a morphism $(E, F, w) \lar  (E', F', w')$ is given by  a pair morphisms
$E \stackrel{f}\lar E'$ and $F \stackrel{g}\lar F'$ in $\catD$ making the diagram
$$
\xymatrix{
E \ar[r]^w  \ar[d]_f & F \ar[d]^g \\
E' \ar[r]^{w'}   & F'
}
$$
commutative.
\end{definition}
\begin{proposition}\label{P:MPSemiOrthogDec} For any $D \in \mathsf{Ob}(\catD)$, choose a distinguished  triangle (\ref{E:Decomposition}). Then the following statements are true.
\begin{itemize}
\item The assignment $D \mapsto (E, F, w)$ extends to a functor
\begin{equation}
\catD \stackrel{\EE}\lar \mathsf{Comma}(\catE, \catF),
\end{equation}
which is moreover full and essentially surjective.
\item Moreover, $\EE$ reflects isomorphism classes and indecomposability of objects:
\begin{itemize}
\item For $D_1, D_2 \in \mathsf{Ob}(\catD)$ we have:  $\EE(D_1) \cong \EE(D_2)$ in $\mathsf{Comma}(\catE, \catF)$ if and only if $D_1 \cong D_2$ in $\catD$;
\item For $D \in \mathsf{Ob}(\catD)$ we have:  $\EE(D)$ is indecomposable in $\mathsf{Comma}(\catE, \catF)$  if and only if
$D$ is indecomposable in $\catD$.
\end{itemize}
\end{itemize}
\end{proposition}

\begin{proof} This result is analogous to \cite[Theorem 4.1]{DrozdHomotopyType}.
The functoriality of $\EE$ is a consequence of \cite[Corollaire 1.1.10]{BBD} (see also \cite{GelfandManin}), which also insures that the object $\bigl(E \stackrel{w}\lar F\bigr)$ is unique in $\mathsf{Comma}(\catE, \catF)$ up to an automorphism. It is easy to see that $\EE$ is essentially surjective. Indeed, let $\bigl(E \stackrel{w}\lar F\bigr)$ be any object of $\mathsf{Comma}(\catE, \catF)$ and $D$ be a  cone of $w$. Then it follows from \cite[Corollaire 1.1.10]{BBD} that $\bigl(E \stackrel{w}\lar F\bigr)\cong \EE(D)$ in $\mathsf{Comma}(\catE, \catF)$. The axiom TR3 of triangulated categories implies that  $\EE$ is full. Next, it is not difficult to see  that the kernel $I(D, D')$ of the map
$$
\Hom_{\catD}(D, D') \stackrel{\EE}\lar \Hom_{\mathsf{Comma}(\catE, \catF)}\bigl(\EE(D), \EE(D')\bigr)
$$
consists of those  morphisms $D \stackrel{f}\lar D'$, which admit a factorization
$$
\xymatrix{
                     & E \ar@{.>}[rd] &  \\
D \ar@{.>}[ru] \ar@{.>}[rd] \ar[rr]^f & &  D'   \\
  & F \ar@{.>}[ru] &                    \\
}
$$
for some objects $E$ of $\catE$ and $F$ of $\catF$. As a consequence, for any objects $D, D'$ and $D''$ of the category $\catD$,
$f \in I(D, D')$ and $g \in I(D', D'')$ we have: $gf = 0$.

\smallskip
\noindent
Now we are prepared to  show that the functor $\EE$ reflects the isomorphism classes of objects. Assume, $D, D' \in \mathsf{Ob}(\catD)$ are such that
$\EE(D) \cong \EE(D')$. Let
$\EE(D) \xrightarrow{(g, h)} \EE(D')$ and $\EE(D) \xrightarrow{(g', h')} \EE(D')$ be mutually inverse isomorphisms. Since the functor $\EE$ is full, there exists morphisms $D \stackrel{f}\lar D'$ and $D' \stackrel{f'}\lar D$ such that $\EE(f) = (g, h)$ and $\EE(f') = (g', h')$. It follows that
$1_D - f' f \in I(D, D)$, hence $(1_D - f' f)^2 = 0$. As a consequence, the morphism
$$
f' f = 1_D - (1_D - f' f)
$$
is an automorphism of $D$. Analogously, $f f'$ is an automorphism of $D'$, so both morphisms $f$ and $f'$ are isomorphisms.

\smallskip
\noindent
The final statement about indecomposability follows from the facts that the functor $\EE$ is essentially surjective and reflects isomorphism classes of objects.
\end{proof}

\smallskip
\noindent
Our next goal is to apply
the construction of Proposition \ref{P:MPSemiOrthogDec}  to the semi--orthogonal decomposition
$D^b\bigl(B-\mathsf{mod}) = \bigl\langle D^b\bigl(\bar{A}-\mathsf{mod}), D^b\bigl(H-\mathsf{mod})\bigr\rangle$ obtained in Theorem \ref{T:skewgentleResolution}.  We follow the same notation as in Subsection \ref{SS:TriplesandMP}.

\begin{itemize}
\item Recall that the set $\widetilde\Omega$ parameterizes isomorphism classes of simple and indecomposable projective  $A$--modules. For any $\alpha
\in \widetilde\Omega$, let $S_\gamma := \widetilde\sF(\bar{A}_\gamma)$ be the simple $B$--module corresponding to the simple $A$--module
$\bar{A}_\gamma$ and $\widetilde{P}_\gamma = \widetilde\sF(P_\gamma)$ be the indecomposable projective $B$--module corresponding to the indecomposable projective $A$--module $P_\gamma$.
\item Next, for any $(i, j) \in \Omega$ we denote: $
\widetilde{Q}_{(i, j)} := \sF\bigl(Q_{(i, j)}\bigr) :=
\left(
\begin{array}{c}
Q_{(i, j)} \\
Q_{(i, j)}
\end{array}
\right).
$
Similarly, for any $1 \le i \le t$ and $1 \le b < a \le m_i+1$ we put: $\widetilde{W}{(i, (a, b))} = \sF \bigl({W}_{(i, (a, b))}\bigr)$ and
 $$\widetilde{W}^\bu_{(i, (a, b))} = \sD \sF \bigl({W}^\bu_{(i, (a, b))}\bigr) = \bigl(\dots \lar 0 \lar \widetilde{Q}_{(i, a)} \lar
 \widetilde{Q}_{(i, b)} \lar 0 \lar \dots\bigr).
 $$
 \item The objects of the comma--category $\mathsf{CC}(A):= \mathsf{Comma}\bigl(D^b(\bar{A}-\mathsf{mod}),  D^b(H-\mathsf{mod})\bigr)$ are triples
 $(S^\bu, W^\bu, \theta)$, where
\begin{equation}\label{E:decompositionCCA}
S^\bu = \bigoplus\limits_{\substack{\gamma \in \widetilde\Omega \\ r \in \ZZ}}  S_\gamma[n]^{\oplus n(\gamma, r)} \quad \mbox{\rm and} \quad
W^\bu = \bigoplus\limits_{\substack{1 \le i \le t \\ 1 \le b < a \le m_i+1 \\ r \in \ZZ}} W^\bu_{(i, (a, b))}[r]^{l(i, (a, b), r)}
\end{equation}
for appropriate multiplicities $n(\alpha, r), l(i, (a, b), r) \in \NN_0$, whereas
$S^\bu \stackrel{\theta}\lar W^\bu$ is a morphism in the derived category $D^b(B-\mathsf{mod})$.
\end{itemize}

\smallskip
\noindent
A proof of the following lemma is a straightforward computation.
\begin{lemma} For any $\gamma \in \widetilde{\Omega}$, the corresponding simple $B$--module $S_\gamma$ has projective dimension at most one. Moreover, the minimal projective resolution of $S_\gamma$ is the following:
\begin{itemize}
\item If $\gamma = \{(i, j), (k, l)\}$ is of first type, then we have a resolution:
$$
0 \lar \widetilde{Q}_{(i, j+1)} \oplus \widetilde{Q}_{(k, l+1)} \lar \widetilde{P}_\gamma \lar S_\gamma \lar 0.
$$
\item If $\gamma = \left\{
\begin{array}{cl}
\{((i, j), \pm)\} & \mbox{\it is of the second type} \\
\{((i, j)\} & \mbox{\it or of the third type} \\
\end{array}
\right.
$
then we have a resolution:
$$
0 \lar \widetilde{Q}_{(i, j+1)}  \lar \widetilde{P}_\gamma \lar S_\gamma \lar 0.
$$
\end{itemize}
\end{lemma}

\smallskip
\noindent
Let $\gamma \in \widetilde\Omega$, $1\le i \le t$ and $1 \le b < a \le m_i+1$. Our
next goal is to describe all non--zero  $\Ext$--spaces in $B-\mathsf{mod}$ between the  simple $B$--module $S_\gamma$
and the module  $\widetilde{W}_{(i, (a, b))}$.

\begin{lemma} Let $\gamma \in \widetilde\Omega$, $S_\gamma$ be the corresponding simple $B$--module, $\widetilde{P}_{\gamma}$ the corresponding indecomposable projective $B$--module and $\widetilde{R}_\gamma = \mathsf{rad}(\widetilde{P}_{\gamma})$. Then the all non--vanishing  $\Ext$--spaces
between $S_\gamma$ and  $\widetilde{W}_{(e, (a, b))}$ are one--dimensional and given by the following morphisms between the corresponding projective resolutions.
\begin{itemize}
\item Let $\gamma = \{(i, j), (k, l)\}$ is of the first type.
\begin{itemize}
\item Assume that $j \le m_i-1$ and $l \le m_k-1$. Then we  have a chain of morphisms:
\begin{equation}\label{E:hierarchyofmorphisms2}
\begin{array}{c}
\xymatrix{
   & \widetilde{R}_\gamma \ar[r] \ar[d] & \widetilde{P}_\gamma \ar[d]\\
& \widetilde{Q}_{(i, j+1)} \ar[r] \ar[d] & \widetilde{Q}_{(i, j)} \ar[d] \\
& \dots \ar[d] & \dots \ar[d]\\
& \widetilde{Q}_{(i, j+1)} \ar[r] \ar[d]& \widetilde{Q}_{(i, 1)} \\
& \widetilde{Q}_{(i, j+1)} \ar[d]& \\
\widetilde{Q}_{(i, m_i)} \ar[r]\ar[d] & \widetilde{Q}_{(i, j+1)} \ar[d] & \\
\dots \ar[d] & \dots \ar[d] & \\
\widetilde{Q}_{(i, j+2)} \ar[r] & \widetilde{Q}_{(i, j+1)} &
}
\end{array}
\end{equation}
as well as another such chain  corresponding to  $(k, l) \in \Omega$.
\item Assume that $j \le m_i-1$ and $l = m_{k+1}$. Then we have a chain (\ref{E:hierarchyofmorphisms2}) corresponding to $(i, j)$, whereas
$(k, l)$ gives the chain
\begin{equation}\label{E:anotherchain}
S_\gamma \lar \widetilde{Q}_{(k, m_k)} \lar \dots \lar \widetilde{Q}_{(k, 1)}.
\end{equation}
\item If $i = k$ and $1 \le j < k = m_i$ then we have only one chain of morphisms (\ref{E:hierarchyofmorphisms2}) corresponding to $(i, j)$.
\item Finally, if $j = m_i$ and $l = m_k$ then we have two chains (\ref{E:anotherchain}), corresponding to $(i, m_i)$ and $k, m_k)$.
\end{itemize}
\item Assume that $\gamma = ((i, j), \pm) = \left\{
\begin{array}{cl}
\{((i, j), \pm)\} & \mbox{\it is of the second type} \\
\{((i, j)\} & \mbox{\it or of the third type} \\
\end{array}
\right.
$

\smallskip
\noindent
For $j \le m_i-1$, we have a chain (\ref{E:hierarchyofmorphisms2}) and for $j = m_i$ a chain of the form (\ref{E:anotherchain}).
\end{itemize}
\end{lemma}

\smallskip
\noindent
\textbf{Summary}. Let $(S^\bu, W^\bu, \theta)$ be an object of the category $\mathsf{CC}(A)$. Then fixing decompositions (\ref{E:decompositionCCA}), we can represent the morphism $\theta$ by a collection of decorated matrices $\Theta$. It follows from the description of non--zero morphisms
between $B$--modules $S_\gamma$ and $\widetilde{W}_{(e, (a, b))}$ that there exists a bunch of semi--chains $\dY$ and a functor
$$
D^b(B-\mathsf{mod}) \stackrel{\EE}\lar \Rep(\dY), \quad (S^\bu, W^\bu, \theta) \mapsto \Theta,
$$
which is fully faithful, essentially surjective and reflecting indecomposability and isomorphism classes of objects.

\end{document}